\documentclass[]{amsart}
\usepackage[utf8]{inputenc}  
\usepackage{csquotes}
\usepackage[T1]{fontenc}
\usepackage{graphicx}
\usepackage{tikz}
\usetikzlibrary{positioning}
\usepackage{float}
\usepackage{algorithm}
\usepackage{algorithmic}
\usepackage[utf8]{inputenc}
\usepackage[export]{adjustbox}
\usepackage{wrapfig}
\usepackage{amsfonts}
\usepackage{amsthm}
\usepackage{array}
\usepackage[top=2cm, bottom=2cm, left=2cm, right=2cm]{geometry}
\usepackage{amssymb}
\usepackage{xcolor}
\usepackage{mathtools}
\usepackage{comment}
\usepackage[main=english]{babel}
\usepackage{subcaption}
\usepackage{dsfont}
\usepackage{bbm}
\usepackage{lmodern}
\usepackage{indentfirst}
\usepackage{tabto}
\usepackage{color}
\usepackage{systeme}
\usepackage{eso-pic}
\usepackage{float}
\usepackage{fancyhdr}
\usepackage{appendix}
\usepackage{listings}
\usepackage{pgf, tikz}
\usepackage[utf8]{inputenc}
\usepackage{lastpage}

\usepackage{dsfont}

\definecolor{bf}{rgb}{0,0,0.6} 
\definecolor{mygray}{gray}{0.85}
\definecolor{darkWhite}{rgb}{0.94,0.94,0.94}

\graphicspath{{images/}}

\newcommand{\espcond}[2]{\mathbb{E}\mathopen{}\left[#1\middle|#2\right]}
\newcommand{\esp}[1]{\mathbb{E}\mathopen{}\left[#1\right]}

\newcommand{\reels}{\mathbb{R}}

\usepackage{thmtools,thm-restate}

\newcommand{\FuncDef}[4]{\ensuremath{\left\{\begin{array}{c} #1\longrightarrow #2\\ #3\mapsto #4\\\end{array}\right.}}

\newcommand{\noise}{\theta}

\usepackage[backend=bibtex,natbib=true,style=numeric,maxbibnames=99]{biblatex}
\renewbibmacro{in:}{}
\usepackage{hyperref}

\numberwithin{equation}{section}

\usepackage{amsthm, amssymb}

\DeclareNameAlias{default}{family-given}

\newtheorem{thm}{Theorem}[section]
\newtheorem{definition}[thm]{Definition}
\newtheorem{prop}[thm]{Proposition}
\newtheorem{lemma}[thm]{Lemma}
\newtheorem{hyp}{Hypothesis}
\newtheorem{corol}[thm]{Corollary}
\theoremstyle{definition}
\newtheorem{remarque}[thm]{Remark}
\newtheorem{exemple}[thm]{Example}

\title{Monotone solutions to mean field games master equation in the $L^2-$monotone setting}
\author{Charles Meynard }

\begin{document}

\begin{abstract}
 This paper is concerned with extending the notion of monotone solution to the mean field game (MFG) master equation to situations in which the coefficients are displacement monotone, instead of the previously introduced notion in the flat monotone regime. To account for this new setting, we work directly on the equation satisfied by the controls of the MFG. Following previous works, we define an appropriate notion of solution under which uniqueness and stability results hold for solutions without any differentiability assumption with respect to probability measures. Thanks to those properties, we show the existence of a monotone solution to displacement monotone mean field games under local regularity assumptions on the coefficients and sufficiently strong monotonicity. Albeit they are not the focus of this article, results presented are also of interest for mean field games of control and general mean field forward backward systems. In order to account for this last setting, we use the notion of $L^2-$monotonicity instead of displacement monotonicity, those two notions being equivalent in the particular case of MFG.

\end{abstract}
\maketitle

\section{Introduction}

\subsection{General introduction}
This paper presents a notion of monotone solution for mean field games (MFG) master equation under displacement monotonicity. It is a natural extension to previous works on the matter \cite{bertucci-monotone-finite,seminaire-charles}, allowing the definition of solutions without any differentiability assumption on the space of probability measures. Those solutions enjoy properties of uniqueness under displacement monotonicity of the coefficients, hence the name. Using results of stability we prove for those solutions we give new existence results to mean field games with coefficients having below Lipschitz regularity in the measure argument. 

MFG were first introduced as a forward-backward system of partial differential equations (PDEs) \cite{LasryLionsMFG,Lions-college} to study games in which a continuum of players interacts together through mean field terms. The master equation was presented in \cite{Lions-college} to study MFG as an alternative 
 to the forward backward system in the presence of common noise. The existence of solutions to the MFG master equation over arbitrary long interval of time is in general a tough problem, which is very much linked to the uniqueness of equilibria to the MFG. While it is known to hold on short horizon of time for sufficiently smooth data \cite{lipschitz-sol}, long time uniqueness usually requires additional monotonicity assumptions \cite{Lions-college}. The existence of long time solutions under flat monotonicity was  obtained in \cite{convergence-problem}. For displacement monotonicity, existence was obtained in \cite{Lions-college} for MFG without idiosyncratic noise through Lions' Hilbertians approach and for game with possibly non-degenerate idiosyncratic noise in a series of paper \cite{disp-monotone-1,disp-monotone-2,globalwellposednessdisplacementmonotone} with probabilistic techniques. Those articles address the problem of uniqueness and existence to the MFG master equation for smooth solutions.
 
 In parallel, a lot of efforts have been made toward the development of weaker notions of solution. Indeed, the existence of a smooth solution usually requires differentiability assumptions with respect to the measure argument on the coefficients (possibly up to the second order for games with common noise), a strong assumption that is not verified in many cases of interest. A non-exhaustive list of developments made in that direction includes \cite{cecchin2022weak} in which a notion of weak solutions was presented for potential mean field games, which is to say mean field games in which the master equation can be integrated into a HJB equation on the space of probability measures. In \cite{lipschitz-sol} a notion of solution was presented for master equation which are merely Lipschitz with respect to the measure argument. For MFG under flat monotonicity specifically, several notions of weak solution were introduced either on the forward backward system \cite{weaksol-dipmono, Cardaliaguet-weak-fb,Porretta-weak} or directly at the level of the master equation \cite{bertucci-monotone,bertucci-monotone-finite,,weaksol-dipmono}. In contrasts, there have been few works concerned with the definition of weak solutions to the MFG master equation under displacement monotonicity. Although, it is concerned with finite state space mean field games, let us mention \cite{bertucci-monotone-finite} in which the master equations studied share the same structure as the one we are interested in the article in the absence of idiosyncratic noise \cite{seminaire-charles,Lions-college}. In the present work, we further develop this notion of solution. In particular, we are able to define a notion of monotone solution even in the presence of idiosyncratic noise. 
 Let us end this section by mentioning a possible application of monotone solutions beyond the scope of this paper: the convergence of numerical methods \cite{bertucci-monotone-numeric}.

\subsection{Setting and motivation of our study}
The main motivation of this study is the solvability of the MFG master equation
\begin{equation}
\label{eq: MFG ME with common noise}
\left\{
\begin{array}{c}
\displaystyle -\partial_t U+H(x,m,\nabla_xU(t,x,m))\\
\displaystyle -(\sigma_x+\beta)\int \text{div}_{y}[D_mU](t,x,m,y)m(dy)-(\sigma_x+\beta)\Delta_x U\\
\displaystyle+\int D_m U(t,x,m,y)\cdot D_p H(y,m,\nabla_x U(t,y,m))m(dy)\\
\displaystyle -2\beta \int \text{div}_{x}[D_m U](t,x,m,y)]m(dy)-\beta \int \text{Tr}\left[D^2_{mm}U(t,x,m,y,y')\right]m(dy)m(dy')=0,\\
\text{ for } (t,x,m) \in (0,T)\times \reels^d \times  \mathcal{P}(\reels^d),\\
U(T,x,m)=U_0(x,m) \text{ for } (x,m) \in \reels^d \times  \mathcal{P}(\reels^d),
\end{array}
\right.
\end{equation}
with non-smooth coefficients, precise regularity assumptions shall be made later on. Here $H$ indicates the Hamiltonian of the MFG, arising from the optimisation problem solved by individual players against the crowd. Although this article is not particularly concerned with the underlying game behind this equation and its general interpretation we remind that this equation can be derived explicitly from a mean field game problem the derivation being done in \cite{derivation-ME,Lions-college}. Let us also mention that terms in $\beta$ are associated to the presence of common noise in the system, whereas terms in $\sigma_x$ correspond to idiosyncratic noise. It was observed in \cite{Lions-college,common-noise-in-MFG} that uniqueness and existence of smooth solutions to \eqref{eq: MFG ME with common noise} under displacement monotonicity can be obtained by studying master equations of the form
\begin{equation}
\label{eq: general ME}
\left\{
\begin{array}{c}
\displaystyle \partial_t W+F(x,m,W)\cdot \nabla_x W-(\sigma_x+\beta) \Delta_x W\\
\displaystyle +\int_{\reels^d} F(y,m,W)\cdot D_m W(t,x,m)(y)m(dy)-\sigma_x\int_{\reels^d} \text{div}_y (D_mW(t,x,m)(y))m(dy)\\
\displaystyle -2\beta \int \text{div}_{x}[D_mW ](t,x,\noise,m,y)]m(dy)-\beta \int \text{Tr}\left[D^2_{mm}W(t,x,\noise,m,y,y')\right]m(dy)m(dy')\\
\displaystyle =G(x,m,W) \text{ in } (0,T)\times\reels^d\times\mathcal{P}_2(\reels^d),\\
\displaystyle W(0,x,m)=W_0(x,m) \text{ for } (x,m)\in\reels^d\times\mathcal{P}_2(\reels^d).
 \end{array}
 \right.
\end{equation}
Whenever $F=D_p H, G=-D_xH,W_0=\nabla_x U_0$, \eqref{eq: general ME} is the equation satisfied by $W=\nabla_xU$, in which case working on $W$ is akin to considering the MFG problem at the level of the controls of players, for a precise argument see \cite{lipschitz-sol}. Throughout this article, links with the solvability of \eqref{eq: MFG ME with common noise} will be highlighted but we take on the approach to focus directly on the master equation \eqref{eq: general ME} without necessarily assuming coefficients are gradients. This allows us to present existence and uniqueness results not only for mean field games but also for mean field forward backward systems \cite{mean-field-FBSDE}
\begin{equation}
\label{eq: general MFFBSDE}
\left\{
\begin{array}{cc}
     dX_t=-F(X_t,U_t,\mathcal{L}(X_t|\mathcal{F}_t))dt+\sqrt{2\sigma_x}dB^x_t+\sqrt{{2\beta}}dB^\beta_t & X|_{t=0}=X_0\\
     dU_t=-G(X_t,U_t,\mathcal{L}(X_t|\mathcal{F}_t))dt+Z_t\cdot d(B^x_t,B^\beta_t) & U_T=W_0(X_T,\mathcal{L}(X_T|\mathcal{F}_T)) \\
     \mathcal{F}_t=\sigma\left((B^\beta_s)_{s\leq t}\right)
\end{array}
\right.
\end{equation}
without any additional cost. Indeed if there exists a smooth solution $W$ to $\eqref{eq: general ME}$, it can be checked easily that the couple $(X_t,W(t,X_t,\mathcal{L}(X_t|\mathcal{F}_t))_{t\in[0,T]}$ is a solution to the forward backward system \eqref{eq: general MFFBSDE}. This is so-called decoupling field approach to backward SDE. Let us remind that in the particular case of mean field games \eqref{eq: general MFFBSDE} can be obtained directly from the control problem of players through Pontryagin's maximum principle.  
\subsection{Main contributions}
In this article we define a concept of monotone solutions to the master equation \eqref{eq: general ME}. We show that under such notion, uniqueness holds for merely continuous solutions in the measure argument. In section \ref{section without idio}, we focus on presenting results of uniqueness and stability for those solutions in the absence of idiosyncratic noise $(\sigma_x=0)$. In particular we explain how to define a notion of monotone solution in the presence of common noise without requiring additional regularity with respect to probability measures. This argument is quite general and extend naturally to previously introduced notion of monotone solutions such as in \cite{bertucci-monotone}. 

We also present estimates on FBSDE of the type \eqref{eq: general MFFBSDE} which appear new, both in their generality and in the low regularity required on the coefficients. We use those estimates to show the existence of monotone solutions for master equations without idiosyncratic noise under various monotonicity assumptions. Let us stress that both the coefficients and the associated solutions may have below Lipschitz regularity. Indeed we show the existence of monotone solutions which are only Hölder, or even just uniformly continuous with respect to the measure argument. This represent an improvement in the study of displacement monotone mean field games, results up until now being concerned with coefficients and solutions having at least Lipschitz regularity in the measure argument. 

For master equations with idiosyncratic noise, there is an added difficulty which makes the obtention of comparison results difficult. We give partial results in this setting: existence still holds and we show how idiosyncratic noise helps to get local regularity estimates on solutions to the master equation. However we do not treat fully the question of uniqueness for solutions with below Lipschitz regularity in this case and only show it in a slightly reduced setting.  
\subsection{Organization of the paper}
In section \ref{section without idio}, we present the notion of solution to the master equation \eqref{eq: general ME} for situations in which there is no idiosyncratic noise $(\sigma_x=0)$. The definition relies on Lions' Hilbertian approach \cite{Lions-college}. We start by showing uniqueness and stability results for master equation without noise $(\beta=0)$. Following an idea from \cite{common-noise-in-MFG} those results are then extended to second order master equations without any additional regularity assumptions. In section \ref{section existence}, we start by reminding the notion of Lipschitz solution \cite{lipschitz-sol} to the master equation \eqref{eq: general ME}. Then, we give several existence results for monotone solutions with or without common noise. Those are based on the stability of monotone solutions as well as estimates stemming from strong monotonicity on the forward-backward sytem \eqref{eq: general MFFBSDE}. Finally in section \eqref{section idiosyncratic noise}, we refine the notion of solution to account for master equations with non-degenerate idiosyncratic noise. The additional presence of terms associated to idiosyncratic noise in the master equation makes the analysis delicate. Our method of proof relies on the introduction of an entropic penalization. The addition of this term makes the analysis more technical, but ensure we obtain sufficient information on second order terms on the space of probability measures at points of minimum. 
\subsection{Notation}
\begin{enumerate}
\item[-] let $k\in \mathbb{N}$, $k>0$, for the canonical product on $\reels^k$ we use the notation
\[x\cdot y=\sum_i x_iy_i,\]
and the following notation for the induced norm
\[|x|=\sqrt{x\cdot x}.\]
    \item[-] Let $\mathcal{P}(\reels^d)$ be the set of (Borel) probability measures on $\reels^d$, for $q\geq 0$, we use the usual notation 
\[\mathcal{P}_q(\reels^d)=\left\{\mu\in \mathcal{P}(\reels^d), \quad \int_{\reels^d} |x|^q \mu(dx)<+\infty\right\},\]
for the set of all probability measures with a finite $q$th moment.
\item[-] For two measures $\mu,\nu\in\mathcal{P}(\reels^d)$ we define $\Gamma(\mu,\nu)$ to be the set of all probability measures $\gamma\in\mathcal{
P}(\reels^{2d})$ satisfying 
\[\gamma(A\times \reels^d)=\mu(A) \quad \gamma(\reels^d\times A)=\nu(A),\]
for all Borel set $A$ on $\reels^d$. 
\item[-]The Wasserstein distance between two measures belonging to $\mathcal{P}_q(\reels^d)$ is defined as
\[\mathcal{W}_q(\mu,\nu)=\left(\underset{\gamma\in \Gamma(\mu,\nu)}{\inf}\int_{\mathbb{\reels}^{2d}} |x-y|^q\gamma(dx,dy)\right)^{\frac{1}{q}}.\]
In what follows $\mathcal{P}_q(\reels^d)$ is always endowed with the associated Wasserstein distance, $(\mathcal{P}_q,\mathcal{W}_q)$ being a complete metric space. 
\item[-] We say that a function $U:\mathcal{P}_q(\reels^d)\to \reels^d$ is Lipschitz if
\[\exists C\geq 0, \quad \forall(\mu,\nu)\in\left(\mathcal{P}_q(\reels^d)\right)^2, \quad |U(\mu)-U(\nu)|\leq C\mathcal{W}_q(\mu,\nu).\]
\item[-] We say that a sequence of function $f_n:\reels^k\times \mathcal{P}_q(\reels^d)\to \reels^k$, converges locally uniformly to a function $f$ if the convergence is uniform on compact sets of $\reels^k\times\mathcal{P}_q(\reels^d)$.
\item[-] For a measure $m\in\mathcal{P}(\reels^{d})$, and $1\leq i\leq d$, we use the notation $\pi_i m$ for the marginal of $m$ over its first $i$ variables and $\pi_{-i}$ for its marginal over the last $i$ variables. Which is to say that for any borel subset $A$ of $\reels^i$
\[\pi_i m(A)=m(A\times \reels^{d-i}) \quad \pi_{-i} m(A)=m(\reels^{d-i}\times A).\]
\item[-] Consider $(\Omega, \mathcal{F},\mathbb{P})$ a probability space,
\begin{itemize}
    \item[-] We define 
\[L^q(\Omega, \reels^d)=\left\{ X: \Omega\to \reels^d, \quad \esp{|X|^q}<+\infty\right\}.\]
\item[-] Whenever a random variable $X:\Omega\to \reels^d$ is distributed along $\mu\in \mathcal{P}(\reels^d)$ we use equivalently the notations $\mathcal{L}(X)=\mu$ or  $X\sim\mu$. 
\item[-] If another probability measure $\mathbb{Q}$ is defined on $(\Omega,\mathcal{F})$, the expectation under $\mathbb{Q}$ of a random variable $X:\Omega\to \reels^d$ is noted $\mathbb{E}^\mathbb{Q}\left[X\right]$.
\end{itemize}
\item[-] $\mathcal{M}_{d\times n}(\reels)$ is the set of all matrices of size $d\times n$ with reals coefficients, with the notation $\mathcal{M}_n(\reels)\equiv \mathcal{M}_{n\times n}(\reels)$.
\item[-] $C_b(\reels^d,\reels^k)$ is the set of all continuous bounded functions from $\reels^d$ to $\reels^k$.
\item[-] A function $\omega:\reels^+\to \reels^+$ is said to be a modulus of continuity if it is continuous, increasing and such that $\omega(0)=0$.
\item[-] We say that a function $f:\reels^d\times \mathcal{P}_2(\reels^d)\to \reels^d$ has linear growth with constant $C>0$, if 
\[\forall (x,\mu)\in \reels^d\times \mathcal{P}_2(\reels^d),\quad |f(x,\mu)|\leq C\left(1+|x|+\sqrt{\int_{\reels^d} |y|^2\mu(dy)}\right)\]
\item[-] For a measure $\mu \in\mathcal{P}_q(\reels^d)$ for some $q\geq 0$, we define 
\[E_q(\mu)=\int_{\reels^d} |y|^q\mu(dy),\]
the function returning its $q-$th moment. In particular for random variables $X\in L^q(\Omega,\reels^d)$ we make the misuse of notation
\[E_q(X)=E_q(\mathcal{L}(X))=\esp{|X|^q}.\]
\item[-] We say that $U:\mathcal{P}(\reels^d)\to \reels$ is derivable at $m$ if there exists a continuous map $\psi:\reels^d\to \reels$ such that
\[\forall \nu\in\mathcal{P}(\reels^d) \quad \underset{\varepsilon\to 0}{\lim} \frac{U(m+\varepsilon(\nu-m))-U(m)}{\varepsilon}=\int_{\reels^d}\psi(x)(\nu-m)(dx).\]
In this case we define the derivative of $U$ at $m$, $\nabla_mU(m)$ to be one such $\psi$ such that 
\[\int_{\reels^d} \nabla_mU(m)(y)m(dy)=0.\]
Whenever $\nabla_mU(m)$ is differentiable, we define the Wasserstein derivative
\[D_m U(m,y)=\nabla_y \nabla_mU(m)(y).\]
Derivatives of higher order on the space of measures can be defined by induction.
\item[-] Finally, for a measure $\mu$ and a measurable function $f$, we use the notation $f_\#\mu$ for the pushforward of $\mu$ by $f$. In particular, for $\noise\in \reels^d$ $(id_{\reels^d}+\noise)_\#\mu$ is defined as the pushforward of $\mu$ by the map $x\mapsto x+\noise$.
\end{enumerate}
\subsection{Preliminaries}
In this article we consider a sufficiently rich probability space $(\Omega,\mathcal{A},\mathbb{P})$, featuring a collection of independent brownian motion. In particular we remind that the space $L^2(\Omega,\reels^d)$ endowed with the inner product 
\[\langle \cdot,\cdot \rangle: (X,Y)\mapsto \esp{X\cdot Y},\]
is a Hilbert space. We use the notation $\mathcal{H}=(L^2(\Omega,\reels^d),\langle \cdot,\cdot\rangle)$ and the induced norm is denoted by
\[\forall X\in \mathcal{H}, \|X\|=\sqrt{\langle X,X\rangle}.\]
\begin{definition}
    For a function $F:\reels^d\times \mathcal{P}_2(\reels^d)\to \reels^d$, we define its lift on $\mathcal{H}$, $\tilde{F}$ to be 
    \[\tilde{F}:\left\{
\begin{array}{l}
     \mathcal{H}\to \mathcal{H},  \\
     X\mapsto F(X,\mathcal{L}(X)).
\end{array}
\right.
\]
\end{definition}
 Throughout this article we also make use of this useful variation of Lemma 2.3 of \cite{disp-monotone-1}
\begin{prop}
\label{prop: expectation to pointwise}
    Let $f:\FuncDef{\reels^{2d}}{\reels}{(x,y)}{f(x,y)}$ be a continuous function,  such that 
    \[\forall (x,y)\in \reels^{2d}, \quad f(x,y)=f(y,x), \quad f(x,x)=0.\]
Suppose that for some $\mu\in\mathcal{P}_2(\reels^d)$ with full support on $\reels^d$ the following holds 
\[\forall (X,Y)\in \mathcal{H}^2, X\sim \mu,Y\sim \mu, \quad \esp{f(X,Y)}\geq 0,\]
then the inequality is satisfied pointwise 
\[\forall (x,y)\in \reels^d,\quad f(x,y)\geq 0.\]
\end{prop}
In particular this direct corollary
\begin{corol}
\label{from lip hilbert to lip x}
Let $F:\reels^d\times \mathcal{P}_2(\reels^d)\to \reels^d$ be a continuous function. Suppose that there exists a constant $C_F$ such that
\[ \forall (X,Y)\in \mathcal{H}^2, \quad \|\tilde{F}(X)-\tilde{F}(Y)\|\leq C_F\|X-Y\|.\]
Then for the same constant $C_F$
\[\forall (\mu,x,y)\in \mathcal{P}_2(\reels^d)\times \left(\reels^d\right)^2, \quad |F(x,\mu)-F(y,\mu)|\leq C_F |x-y|.\]
\end{corol}
\begin{proof}
It suffices to apply the above proposition to 
\[f:(x,y)\mapsto C_F^2|x-y|^2-|F(x,\mu)-F(y,\mu)|^2,\]
for fixed $\mu$ with full support in $\reels^d$. For measure that do not have full support, the result then follows by density and the continuity of $F$ in $\mathcal{P}_2(\reels^d)$ as in \cite{disp-monotone-1}.
\end{proof}

\section{The Hilbertian approach to master equations}
\label{section without idio}
 \subsection{Master equation without common noise}
In this setting, we can use Lions' Hilbertians approach of lifting the equation on the Hilbert space $\mathcal{H}$. We first treat the situation of a deterministic evolution $(\beta=\sigma_x=0)$. The master equation we consider is hence 

\begin{equation}
\label{eq: ME without noise}
\left\{
\begin{array}{c}
\displaystyle \partial_t W+F(x,m,W)\cdot \nabla_x W +\int_{\reels^d} F(y,m,W)\cdot D_m W(t,x,m)(y)m(dy)\\
\displaystyle =G(x,m,W) \text{ in } (0,T)\times\reels^d\times\mathcal{P}_2(\reels^d),\\
\displaystyle W(0,x,m)=W_0(x,m) \text{ for } (x,m)\in\reels^d\times\mathcal{P}_2(\reels^d).
 \end{array}
 \right.
\end{equation}
We now introduce the notion of monotonicity we use in this article, it will be key in showing the uniqueness of solution to \eqref{eq: ME without noise}, and more generally to \eqref{eq: general ME} and is at the core of our definition of solution as is classic in monotone solutions \cite{bertucci-monotone,bertucci-monotone-finite}. 
\begin{definition}
A function $W:\reels^d\times\mathcal{P}_2(\reels^d)\to \reels^d$ is $L^2-$monotone if 
\[\forall (X,Y)\in\mathcal{H}^2\quad \esp{(W(X,\mathcal{L}(X))-W(Y,\mathcal{L}(Y)))\cdot (X-Y)}\geq 0.\]
\end{definition}
Let $\tilde{W}$ be the lift of $W$ on $\mathcal{H}$, this definition is equivalent to asking for the monotonicity of $\tilde{W}$ in the classical sense on $\mathcal{H}$, hence the terminology $L^2-$monotonicity. This notion of monotonicity is by no mean new in the study of FBSDE for exemple see \cite{G-monotonicity} or \cite{probabilistic-mfg} for mean field type FBSDE specifically. Let us also recall the strong link between this notion of monotonicity on the Hilbert space of square integrable random variables and displacement monotonicity 
\begin{definition}
    A function $U:\reels^d\times \mathcal{P}_2(\reels^d)\to \reels$ is displacement monotone if and only if its gradient in the space variable $\nabla_x U$ is well defined and $L^2-$monotone. 
\end{definition}
In the particular case of MFG, \eqref{eq: general ME}, corresponds to the gradient of \eqref{eq: MFG ME with common noise}, in which case the introduction of displacement monotonicity is natural as seen in \cite{disp-monotone-1,disp-monotone-2}. 
We now express an Hypothesis that will be in force throughout this whole work and ensure uniqueness of solutions to the master equation \eqref{eq: general ME}. 
\begin{hyp}
\label{hyp: weak monotonicity}
The initial condition $W_0$ is $L^2-$monotone and the coefficients $(G,F)$ are jointly $L^2-$monotone in $(x,w)$ i.e.
\begin{gather*}
\forall (X,Y,U,V)\in \mathcal{H}^4, \\
    \left\{\begin{array}{l}
         \esp{\left(W_0(X,\mathcal{L}(X)-W_0(Y,\mathcal{L}(Y))\right)\cdot(X-Y)}\geq 0,  \\
    \esp{F(X,\mathcal{L}(X),U)-F(Y,\mathcal{L}(Y),V)\cdot(U-V)+(G(X,\mathcal{L}(X),U)-G(Y,\mathcal{L}(Y),V))\cdot(X-Y)}\geq 0,
    \end{array}\right.
\end{gather*}
\end{hyp}
\begin{remarque}
    Let us remark at this point that all results presented in this article extend wihout any difficulty to situation in which the coefficients $(F,G)$ depends on the law of controls (from the interpretation of mean field games). Indeed from the viewpoint of the Hilbertian Lifting it does not matter if the depencency of $F$ on the solution $W$ is of the form 
    \[(t,X)\mapsto F(X,\mathcal{L}(X),W(t,X,\mathcal{L}(X)),\mathcal{L}(X,W(t,X,\mathcal{L}(X)))),\]
    as this can be writen of as a dependency on the random variable $W(t,X,\mathcal{L}(X))$. In this case the appropriate monotonicity assumption on $(G,F)$ becomes 
    \[\esp{F(X,U,\mathcal{L}(X,U))-F(Y,V,\mathcal{L}(Y,V))\cdot(U-V)+(G(X,U,\mathcal{L}(X,U))-G(Y,V,\mathcal{L}(Y,V)))\cdot(X-Y)}\geq 0,\]
    under which uniqueness of solution can still be obtained. This is also true for existence results presented in this article (even outside of the Hilbertian regime) as our estimates on the systems come mainly from considerations on $L^2-$monotonicity. We leave such adaptations to the interested reader, and consider instead a slightly simpler setting to lighten notations.
\end{remarque}
We shall work in the class of solutions with at most linear growth at infinity, which is why we make the following growth assumption on coefficients
\begin{hyp}
    \label{hyp: linear growth in F}
 $(F,G,W_0)$ are continuous and such that there exists a constant $C$ such that for any $(x,\mu,u)\in \reels^d\times \mathcal{P}_2(\reels^d)\times \reels^d$ 
    \begin{gather*}
    |W_0(x,\mu)|+|F(x,\mu,u)|+|G(x,\mu,u)|\leq C\left(1+|x|+|u|+\sqrt{E_2(\mu)}\right).
    \end{gather*}
\end{hyp}
\begin{remarque}
In terms of the forward backward system \eqref{eq: general MFFBSDE} the growth assumption of $F$ in $(x,\mu)$ is very natural, it guaranties that starting from a random variable $X\in L^2(\Omega, \reels^d)$ its flow by $F$ stay in this space. This appears to be very important for the Hilbertian approach which relies on lifting the equation in this space, and it does not seem such assumption can be weakened easily. 
\end{remarque}
We now recall the Hilbertian approach to master equation introduced by Lions in \cite{Lions-college}.
\begin{lemma}
    Suppose $W$ is a smooth solution to \eqref{eq: ME without noise}, then its lift $\tilde{W}$ on $\mathcal{H}$ is a solution to 
    \[\partial_t \tilde{W}+\tilde{F}(X,\tilde{W}(t,X))\cdot \nabla_X\tilde{W}=\tilde{G}(X,\tilde{W}(t,X)) \text{ in } (0,T)\times \mathcal{H},\]
    where $ \tilde{F}(X,\tilde{W}(t,X))\cdot \nabla_X\tilde{W}$ indicates the Gateaux derivative of $\tilde{W}$ in $\mathcal{H}$ in the direction $\tilde{F}(X,\tilde{W}(t,X))$.
\end{lemma}
One of the key feature of this lift is that it allows to extend results of existence and uniqueness developped on the corresponding finite dimensional non linear transport equation 
\[\partial_t W+F(x,W)\cdot \nabla_x W=G(x,W), \]
to \eqref{eq: ME without noise}.
In particular the following uniqueness lemma, proved by a propagation of monotonicity argument.   
\begin{lemma}
\label{lemma: smooth uniqueness Hilbert approach}
    Under Hypotheses \ref{hyp: weak monotonicity} and \ref{hyp: linear growth in F} there exists at most one smooth solution with linear growth to \eqref{eq: ME without noise}. If there exists such a solution then it is $L^2-$monotone
\end{lemma}
\begin{proof} Consider two solutions $W^1,W^2$. We define 
    \[(0,T)\times \mathcal{H}^2\ni (t,X,Y)\mapsto Z(t,X,Y)=\langle \tilde{W}^1(t,X)-\tilde{W}^2(t,Y),X-Y\rangle,\]
    where we use the lifting $\tilde{W}^i(t,X)=W^i(t,X,\mathcal{L}(X))$. 
    Since $W^1$ and $W^2$ are smooth, $Z$ is a solution to 
    \begin{equation}
    \label{eq: Z smooth}
    \partial_t Z+F^1(t,X)\cdot\nabla_X Z+F^2(t,Y)\cdot \nabla_Y Z=\langle G^1(t,X)-G^2(t,Y),X-Y\rangle+\langle F^1(t,X)-F^2(t,Y),W^1(t,X)-W^2(t,Y)\rangle,
    \end{equation}
    with the notation $F^i(t,X)=F(t,X,\mathcal{L}(X),W^i(t,X))$ and similarly for $G$. From Hypothesis \ref{hyp: weak monotonicity}, we deduce that $Z$ is such that 
    \[
\left\{
\begin{array}{c}
     Z|_{t=0}\geq 0, \\
     \partial_t Z+F^1(t,X)\cdot\nabla_X Z+F^2(t,Y)\cdot \nabla_Y Z\geq 0 \quad \forall (t,X,Y)\in (0,T)\times \mathcal{H}^2.
\end{array}
\right.\]
Since $Z$ also has at most quadratic growth by assumptions, by the maximum principle \cite{common-noise-in-MFG} Lemma 4.3, we conclude to 
\[\forall t<T, \quad (X,Y) \quad \langle \tilde{W}^1(t,Y)-\tilde{W}^2(t,Y),X-Y\rangle\geq 0.\]
Taking this inequality for $X=Y+\varepsilon U$ for some $U\in \mathcal{H}$ leads to  
\[\forall (U,Y)\in\mathcal{H}^2\quad  \varepsilon\langle \tilde{W}^1(t,Y+\varepsilon U)-\tilde{W}^2(t,Y),U\rangle_{L^2}\geq 0.\]
Dividing by $\varepsilon$ and taking the limit as it tends to 0 implies 
\[\forall Y\in\mathcal{H} \quad W^1(t,Y,\noise,\mathcal{L}(Y))=W^2(t,Y,\noise,\mathcal{L}(Y))\quad  a.s.\]
Consider now a random variable $Y$ with positive density on $\reels^d$, the above equality implies that $x\mapsto W^1(t,x,\noise,\mathcal{L}(Y))$ and $x\mapsto W^2(t,x,\noise,\mathcal{L}(Y))$ are equal almost everywhere. Because they are both continuous functions by assumptions,
\[\forall x\in\reels^d \quad W^1(t,x,\noise,\mathcal{L}(Y))=W^2(t,x,\noise,\mathcal{L}(Y)).\]
For measure that do no have full support on $\reels^d$, we conclude by density. As a consequence, there is indeed at most one smooth solution $W$. The fact that it is $L^2-$monotone is observed simply by taking $Z$ for $W^1=W^2=W$.
\end{proof}
Let us insist at this point that for smooth solutions such lemma is not necessary to obtain uniqueness, see for example \cite{lipschitz-sol}. Rather, we want to highlight the proof itself. It relies on the study of an auxiliary function, the non negativity of which implies uniqueness. The argument to show that the auxiliary function 
\begin{equation}
\label{auxiliary function smooth}
(t,X,Y)\mapsto \langle \tilde{W}_1(t,X)-\tilde{W}_2(t,Y),X-Y\rangle
\end{equation}
stays non-negative over time under monotonicity relies on a maximum principle argument. However it is well known that such argument can be generalised beyond smooth functions through the notion of viscosity solutions \cite{crandall1992users}, even in Hilbert spaces \cite{Lions-visc-inf}. This is the key idea behind the notion of monotone solution: defining a satisfying, non smooth, notion of solution to \eqref{eq: ME without noise} in such fashion that for two solutions the map \eqref{auxiliary function smooth} is a viscosity supersolution to \eqref{eq: Z smooth}. For more on the notion of monotone solution we refer to the papers introducing the notion \cite{bertucci-monotone-finite,bertucci-monotone}. Before presenting a precise definition for the notion of monotone solution we use in this section, we need to introduce some notations. Namely, a function $\psi:[0,T)\times \mathcal{H}\to \reels$ is said to belong to the space of test functions $H^\mathcal{H}_{test}$ if it is of the form 
\[\psi(t,X)=\varphi(t)+\langle h,X\rangle-\alpha e^{\lambda t}E_4(X),\]
where
\begin{enumerate}
    \item[-] $\varphi$ is uniformly bounded in $C^1([0,\tau],\reels)$ for any $\tau<T$,
    \item[-] $h\in \mathcal{H}$,
    \item[-] $\alpha,\lambda  >0$,
\end{enumerate}
\begin{definition}
    Let $f,g:[0,T]\times \mathcal{H}\to \mathcal{H}$ be continuous functions and further assume that there exists a constant $C$ such that 
    \[\forall (t,X)\in [0,T]\times \mathcal{H},\quad |f(t,X)|\leq C(1+|X|+\|X\|)\]
    we say that a continuous function $w:[0,T]\times \mathcal{H}\to \mathcal{H}$ is a viscosity supersolution to 
    \[\partial_t w+f(t,X)\cdot \nabla_Xw\geq g(t,X),\]
    if at every point of minimum $(t^*,X^*)$ in $(0,T)\times \mathcal{H}\cap L^4(\Omega,\reels^d)$ of $w-\psi$ for $\psi\in H_{test}^\mathcal{H}$ the following holds 
    \[\partial_t \psi(t^*,X^*)+f(t^*,X^*)\cdot \nabla_X \psi(t^*,X^*)\geq g(t^*,X^*).\]
\end{definition}
\begin{remarque}
    The growth hypothesis on $f$ ensures that at points $X^*\in L^4(\Omega,\reels^d)$ of minimum the gateaux derivative of $X\mapsto E_4(X)$ in the direction $f(t^*,X^*)$ is well defined. Indeed it ensures $f(t^*,X^*)\cdot X^*|X^*|^2\in L^1(\Omega,\reels^d)$. More generally it may seems weird that we introduce a set of test functions which are not even continuous on $\mathcal{H}$ ($X\mapsto E_4(X)$, being only lower semi-continuous). To prove uniqueness of solutions such a technicality is in fact not necessary. However restraining the set of test functions in such fashion will be very useful to show the stability of solutions. As for the fact, we can restrict ourselves to linear perturbations beyond this term, it shall be justified later on. 
\end{remarque}
Now that a proper definition of the concept of viscosity supersolution as well as a set of test function have been introduced, we may proceed with the definition of a monotone solution
\begin{definition}
    A continuous function with linear growth $W:[0,T]\times \reels^d\times \mathcal{P}_2(\reels^d)\to \reels^d$ is said to be a monotone solution to \eqref{eq: ME without noise} if for any $(Y,V)\in\mathcal{H}^2$ the function $Z$ defined by 
    \[(t,X)\mapsto Z(t,X)=\langle W(t,X,\mathcal{L}(X))-V,X-Y\rangle,\]
    is a viscosity supersolution of 
    \[\partial_t Z+F(X,W(t,X))\cdot \nabla_X Z \geq \langle F(X,\mathcal{L}(X),W(t,X,\mathcal{L}(X))), W(t,X,\mathcal{L}(X))-V\rangle+\langle G(X,\mathcal{L}(X),W(t,X,\mathcal{L}(X))),X-Y\rangle\]
\end{definition}
We now show that uniqueness still holds in this class of solutions, by an adaption of Lemma \ref{lemma: smooth uniqueness Hilbert approach} to monotone solutions using classic tools from the theory of viscosity solutions.
\begin{thm}
\label{thm: uniqueness hilbert}
Under Hypotheses \ref{hyp: weak monotonicity} and \ref{hyp: linear growth in F} there is at most one monotone solution to \eqref{eq: ME without noise}. Moreover, if there exists such a solution it is $L^2-$monotone.
\end{thm}
\begin{proof}
    Suppose there exist two monotone solutions $W^1$ and $W^2$. We define 
    \[Z(t,s,X,Y)=\langle \tilde{W}^1(t,X)-\tilde{W}^2(s,Y),X-Y\rangle.\]
Following Lemma \ref{lemma: smooth uniqueness Hilbert approach}, it is sufficient to show that 
\[\forall (t,X,Y)\in [0,T]\times \mathcal{H}^2 \quad Z(t,t,X,Y)\geq 0,\]
to conclude that $W^1\equiv W^2$. 
Let us assume by contradiction that there exists a triple $(t,X,Y)$ and $\delta>0$ such that 
\[Z(t,t,X,Y)<-\delta.\]
We first remark that by assumption on $W^1,W^2$, $Z$ has at most quadratic growth (in the $L^2-$norm), as such, for any $\alpha,\gamma,\lambda>0$, the function 
\[Z_{\alpha,\lambda}:(s_1,s_2,X,Y)\mapsto Z(s_1,s_2,X,Y)+\alpha \left(e^{\lambda s_1}(1+E_4(X))+\frac{1}{T-s_1}+e^{\lambda s_2}(1+E_4(Y))+\frac{1}{T-s_2}\right)+\frac{1}{2\gamma^3}|s_1-s_2|^2,\]
is bounded from below and lower semi-continuous in $\mathcal{H}$. By Stegall's Lemma \cite{Stegall1978}, for any $\varepsilon>0$ there exists $h_X,h_Y\in \mathcal{H}$ such that $\|h_X\|,\|h_Y\|\leq \varepsilon$ and 
\[Z_{\alpha,\lambda,\varepsilon}: (s_1,s_2,X,Y)\mapsto Z_{\alpha,\lambda}(s_1,s_2,X,Y)+\langle h_X,X\rangle+\langle h_Y,Y\rangle,\]
reaches a point of strict minimum  $(s_1^*,s_2^*,X^*,Y^*)\in [0,T]^2\times \mathcal{H}^2$. For $\alpha,\lambda$ sufficiently small, this point of minimum is such that 
\begin{equation}
\label{eq: uniqueness monotone solution on hilbert space Z negative}
Z_{\alpha,\lambda,\varepsilon}(s_1^*,s_2^*,X^*,Y^*)\leq -\frac{\delta}{4},
\end{equation}
implying that $(s_1^*,s_2^*)\neq (0,0)$ by assumption on $W_0$.
Indeed by density of $L^4(\Omega,\reels^d)$ in $\mathcal{H}$, there exists a sequence $(X_n)_{n\in\mathbb{N}}\subset L^4(\Omega,\reels^d)$ (take for example $X_n=X\mathds{1}_{|X|\leq n}$) such that 
\[\underset{n\to \infty}{\lim} \|X_n-X\|=0.\]
By the continuity of $Z$, there exists $n_0$ such that for $n\geq n_0$
\[Z_{\alpha,\lambda}(t,t,X_n,Y_n)\leq -\frac{\delta}{2}+C\alpha(1+E_4(X_n)+E_4(Y_n)).\]
Taking $\alpha \leq \frac{\delta}{4C(1+E_4(X_{n_0})+E_4(Y_{n_0}))}$ yields \eqref{eq: uniqueness monotone solution on hilbert space Z negative}. 
 Let us assume that both $s_1^*$ and $s_2^*$ are different from 0. Applying the definition of monotone solution to $W^1$ we get the following inequality
\begin{align*}
&\frac{-\alpha}{(T-s_1^*)^2}+\frac{1}{\gamma^3}(s_1^*-s_2^*)-\alpha\lambda e^{\lambda s_1^*}(1+E_4(X^*))-\langle F^1(s_1^*,X^*),h_X+3\alpha (X^*)^2X^*\rangle\\
&\geq \langle F^1(s_1^*,X^*),W^1(s_1^*,X^*)-W^2(s_2^*,Y^*)\rangle+\langle G^1(s_1^*,X^*),X^*-Y^*\rangle,\\
\end{align*} 
and a symmetric one for $W^2$. Choosing $1\geq \alpha \geq \gamma,\varepsilon$, there exists a constant $\lambda_1$ depending only on the linear growth of $F$ and $W^1$ such that for any $\lambda\geq \lambda_1$
\[-\alpha\lambda e^{\lambda s_1^*}(1+E_4(x^*))-\langle F^1(s_1^*,X^*),h_X+3\alpha (X^*)^2X^*\rangle\leq  0.\]
Letting $\lambda\geq \max\left(\lambda_1,\lambda_2\right)$, where $\lambda_2$ is defined similarly, the following inequalities hold
\begin{gather*}
    \frac{-\alpha}{(T-s_1^*)^2}-\frac{1}{\gamma^3}(s_1^*-s_2^*)\geq \langle \tilde{F}^1(s_1^*,X^*),\tilde{W}^1(s_1^*,X^*)-\tilde{W}^2(s_2^*,Y^*)\rangle+\langle \tilde{G}^1(s_1^*,X^*),X^*-Y^*\rangle,\\
    \frac{-\alpha}{(T-s_2^*)^2}-\frac{1}{\gamma^3}(s_2^*-s_1^*)\geq \langle \tilde{F}^2(s_2^*,X^*),\tilde{W}^2(s_2^*,Y^*)-\tilde{W}^1(s_1^*,X^*)\rangle+\langle \tilde{G}^2(s_2^*,Y^*),Y^*-X^*\rangle,\\
\end{gather*}
with the notation $\tilde{F}^1(s,X)=\tilde{F}(X,\tilde{W}^1(s,X))$, and so on. Summing them and using the joint monotonicity of $(F,G)$ we get a contradiction. This means that for $\alpha$ sufficiently small either $s_1^*=0$ or $s_2^*=0$. We may assume without loss of generality that $s_1^*=0$. Using the quadratic growth of $Z$ combined with the inequality
\[Z_{\alpha,\lambda,\varepsilon}(0,s_2^*,X^*,Y^*)\leq Z_{\alpha,\lambda,\varepsilon}(0,0,0,0)=C\alpha,\]
we deduce that 
\begin{equation}
\label{eq: ineq uniqueness hilbert s1=0}
\alpha\left( E_4(X^*)+E_4(y^*)\right)+\frac{1}{2\gamma^3}|s_2^*|^2\leq C\left(\alpha+\|X^*\|^2+\|Y^*\|^2\right),
\end{equation}
for a constant $C$ depending only on $\lambda,T$ and the linear growth of $W^1,W^2$. Since for every $X\in L^4(\Omega,\reels^d)$ $\|X\|^2\leq \sqrt{E_4(X)}$, we get the following bounds 
\[\alpha E_4(X^*),\alpha E_4(Y^*)\leq C,\]
for possibly another constant with the same dependencies. Finally, reusing \eqref{eq: ineq uniqueness hilbert s1=0} combined with this newfound bound gives
\[s^*_2\leq C'\sqrt{\gamma}. \]
For $\alpha$ sufficiently small we get a contradiction with \eqref{eq: uniqueness monotone solution on hilbert space Z negative} by letting $\gamma$ tend to $0$ thanks to the monotonicity of the initial condition and the continuity of $Z$.
\end{proof}
\begin{remarque}
Let us remark that this notion of solution can be extended very easily to the class of semi-monotone initial conditions instead of purely monotone. Let us assume $W_0$ is only semi-monotone which is to say there exists a constant $C>0$ such that
\[\forall (X,Y)\in \mathcal{H}^2, \quad \langle \tilde{W}_0(X)-\tilde{W}_0(Y),X-Y\rangle \geq -C\|X-Y\|^2.\]
If there exists a $c>0$ such that the triple \[(W_0',G',F'):(t,X,W)\mapsto (\tilde{W}_0(X)+cX,\tilde{G}(X,W)+c\tilde{F}(X,W),\tilde{F}(X,W)),\]
satisfies the assumption \eqref{hyp: weak monotonicity}, then uniqueness of a continuous semi-monotone solution is obtained by showing $W+cX$ is a monotone solution associated to this data. If there exists two such solutions $(W^1,W^2)$, then following this proof of uniqueness they must satisfy
\[\forall (t,X,Y)\in[0,T]\times \mathcal{H}^2, \quad \langle \tilde{W}^1(t,X)-\tilde{W}^2(t,Y),X-Y\rangle \geq -c\|X-Y\|^2,\]
which is only possible if $W^1\equiv W^2$. This idea appears to be similar to the one developed in \cite{mezaros-canonical} and seems quite natural.
\end{remarque}
It has now been established that the notion of monotone solution is sufficiently strong to conserve uniqueness under natural assumptions. We now turn to the question of existence of such solutions. To show the existence of monotone solutions we will proceed by first regularizing the coefficients to get existence of smoother solutions (which is easier) and then use the following stability result to conclude to the existence of a monotone solution under appropriate monotonicity assumptions.
  \begin{thm}
  \label{stability without noise}
Suppose that there exists $(F_n,G_n)_{n\in\mathbb{N}}$, a sequence of functions satisfying Hypothesis \ref{hyp: linear growth in F}, uniformly in $n\in\mathbb{N}$ and converging locally uniformly to a couple $(F,G)$. Suppose there exists an associated sequence of monotone solutions $(W_n)_{n\in\mathbb{N}}$ such that 
\[\exists C>0, \quad \forall n\in \mathbb{N}, \forall (t,x,\mu)\in [0,T]\times\reels^d\times \mathcal{P}_2(\reels^d), \quad |W_n(t,x,\mu)|\leq C\left(1+|x|+\sqrt{E_2(\mu)}\right).\]
If $(W_n)_{n\in\mathbb{N}}$ converges locally uniformly to a function $W$, then $W$ is a monotone solution associated to the data $(F,G)$.
\end{thm}
\begin{proof}
    Fix $(V,Y)\in \mathcal{H}^2$ and let us assume that 
    \[Z: (t,X)\mapsto \langle \tilde{W}(t,X)-V,X-Y\rangle,\]
    is such that $Z-\varphi$ reaches a minimum at $(t^*,X^*)$ in $(0,T)\times \mathcal{H}$ . Considering 
    \[\tilde{\varphi}:(t,X)\mapsto \varphi(t,X)-|t-t^*|^3-\|X-X^*\|^3\]
    We may always assume that $(t^*,X^*)$ is a strict minimum of $Z-\varphi$ and that
\begin{equation}
    \label{ineq: strict minimum hilbert}
    \forall (t,X)\in [0,T]\times \mathcal{H}\quad (Z-\varphi)(t,X)-(Z-\varphi)(t^*,X^*)\geq  |t-t^*|^3+\|X-X^*\|^3.
\end{equation}
Moreover by the growth assumption on the family $(W_n)_{n\in\mathbb{N}}$, and the form of our test functions, there exists a constant $C$ such that uniformly in $n$
\begin{equation}
\label{ineq: growth at infinity stability hilbert}
\forall (t,X)\in [0,T]\times L^2(\Omega, \reels^d) \quad Z_n(t,X)-\varphi(t,X)\geq -C+\frac{\alpha}{2}E_4(X),
\end{equation}
for some $\alpha>0$.
    Since $Z_n-\varphi$ is bounded from below and lower semi-continuous, by Stegall's Lemma \cite{Stegall1978}, we can always find a perturbation $h_n$ such that $\|h_n\|\leq \frac{1}{n}$ and 
    \[(t,X)\mapsto (Z_{n}-\varphi)(t,X)+\langle h_n,X\rangle,\]
    reaches a minimum $(t_n,X_n)$ on $[0,T]\times \mathcal{H}$. By \eqref{ineq: growth at infinity stability hilbert}, $(X_n)_{n\in \mathbb{N}}$ is bounded independently of $n$ in $L^4(\Omega,\reels^d)$, consequently for any $n\in \mathbb{N}$ 
    \[Z_n(t_n,X_n)-\varphi(t_n,X_n)\leq Z_n(t^*,X^*)-\varphi(t^*,X^*)+\frac{C}{n}.\]

    Taking both the $\liminf$ and $\limsup$ in this inequality yields     
    \[\underset{n\to \infty}{\lim} \left( Z_n(t_n,X_n)-\varphi(t_n,X_n)\right) =Z(t^*,X^*)-\varphi(t^*,X^*).\]
    \[Z(t_n,X_n)-\varphi(t_n,X_n)-Z(t^*,X^*)-\varphi(t^*,X^*)\leq \|Z(t_n,X_n)-Z_n(t_n,X_n)\|+\|Z_n(t_n,X_n)-\varphi(t_n,X_n)-(Z(t^*,X^*)-\varphi(t^*,X^*))\|,\]
    we deduce using \eqref{ineq: strict minimum hilbert} that 
    \[\underset{n\to\infty}{\lim} |t_n-t^*|^2+\|X_n-X^*\|^2\leq \underset{n\to \infty}{\lim}\|Z(t_n,X_n)-Z_n(t_n,X_n)\|.\]
    It now remains to show the convergence of the righthand term.  Let us remark at this point that letting $\mathcal{L}(X_n)=\mu_n$, the sequence $(\mu_n)_{n\in\mathbb{N}}$ is uniformly bounded in $\mathcal{P}_4(\reels^d)$. By an application of Hölder's inequality, this means this family has uniformly integrable second moments. Thus we can find a sequence of compact sets $(\mathcal{K}_\varepsilon)_{\varepsilon>0}\subset \reels^d$ such that 
    \[\forall n\in\mathbb{N},\quad \int_{\reels^d\textbackslash \mathcal{K}\varepsilon}(1+|x|^2)\mu_n(dx)\leq \varepsilon. \]
    In particular, by Prokhorov theorem and the uniform integrability of second moments, this family is relatively compact. 
    As a consequence 
\begin{align*}
\|\tilde{W}(t_n,X_n)-\tilde{W}_n(t_n,X_n)\|^2&=\int_{\reels^d} |W(t,x,\mu_n)-W_n(t,x,\mu_n)|^2\mu_n(dx)\\
&\leq C\varepsilon+\int_{K_\varepsilon} |W(t_n,x,\mu_n)-W_n(t_n,x,\mu_n)|^2\mu_n(dx)
\end{align*}
By the uniform convergence of $(W_n)_{n\in\mathbb{N}}$ on compact sets of $[0,T]\times \reels^d\times \mathcal{P}_2(\reels^d)$, we deduce that 
\[\forall \varepsilon>0, \quad \underset{n\to \infty}{\lim }\int_{K_\varepsilon} |W(t_n,x,\mu_n)-W_n(t_n,x,\mu_n)|^2\mu_n(dx)=0.\]
By letting first $n$ to infinity and then $\varepsilon$ to 0 we conclude that 
\[\underset{n\to \infty}{\lim}\|\tilde{W}(t_n,X_n)-\tilde{W}_n(t_n,X_n)\|^2=0.\]
This is sufficient to deduce the same holds for $\|Z(t_n,X_n)-Z_n(t_n,X_n)\|$ and as a consequence 
\[t_n\underset{n\to \infty}{\longrightarrow} t^*, \quad X_n\overset{L^2}{\underset{n\to \infty}{\longrightarrow}} X^*.\]
    In particular, for $n$ sufficiently big $t_n \neq 0$, applying the viscosity property to $W_n$ yields
    \begin{equation}
    \label{eq: W_n monotone solution hilbert}
    \partial_t \varphi(t_n,X_n)+\tilde{F}(t_n,X_n)\cdot D_X \varphi(t_n,X_n)\geq \langle \tilde{F}_n(t_n,X_n),\tilde{W}_n(t_n,X_n)-V\rangle+\langle \tilde{G}(t_n,X_n),X_n-Y\rangle-\frac{C}{n},
    \end{equation}
    with the abuse of notation $\tilde{F}_n(t,X)=\tilde{F}_n(t,X,\tilde{W}_n(t,X)),\tilde{G}_n(t,X)=\tilde{G}_n(t,X,\tilde{W}_n(t,X))$.
    The convergence of the righthand side as $n$ tends to infinity follows from the local uniform convergence of $(F,G,W)$ and the argument we just provided above. For the lefthand side we have to be more careful, indeed $\varphi$ is not continuous on $\mathcal{H}$. We remind that it is of the form 
    \begin{equation}
    \label{eq: test function}
    \varphi(t,X)=\phi(t,X)-\alpha e^{\lambda t}E_4(X),\end{equation}
    where $\phi$ is smooth and $\alpha >0$. For terms depending on $\phi$ there is no particular difficulty, so the main problem here is the convergence of 
    \[-\alpha e^{\lambda t_n} \int_{\reels^d} \left(\lambda |x|^4+F_n(x,\mu_n,W_n(t_n,x,\mu_n)\cdot x|x|^2\right)\mu_n(dx).\]
    By the uniform linear growth of the family $(W_n,F_n)_{n\in\mathbb{N}}$ and the boundedness of $(X_n)_{n\in\mathbb{N}}$ in $L^4(\Omega,\reels^d)$, there exists $\lambda_0,C>0$ such that 
    \[\forall \lambda >\lambda_0, \forall n\in\mathbb{N},\forall x\in \reels^d,\quad \lambda |x|^4+ F_n(x,\mu_n,W_n(t_n,x,\mu_n)\cdot x|x|^2\geq -C\]
   For such $\lambda$ a variation on Fatou's lemma, see \cite{fatou} Theorem 2.8, yields 
   \[\limsup_n -\alpha e^{\lambda t_n} \int_{\reels^d} \left(\lambda |x|^4+F_n(x,\mu_n,W_n(t_n,x,\mu_n)\cdot x|x|^2\right)\mu_n(dx)\leq -\alpha e^{\lambda t^*} \int_{\reels^d} \left(\lambda |x|^4+F(x,\mu^*,W(t^*,x,\mu)\cdot x|x|^2\right)\mu^*(dx).\]
   Which in turns implies 
   \[\partial_t \varphi(t^*,X^*)+F(t^*,X^*)\cdot D_X \varphi(t^*,X^*)\geq \limsup_n \left(\partial_t \varphi(t_n,X_n)+F(t_n,X_n)\cdot D_X \varphi(t_n,X_n)\right),\]
   and then
    \[\partial_t \varphi(t^*,X^*)+F(t^*,X^*)\cdot D_X \varphi(t^*,X^*)\geq \langle F(t^*,X^*),W(t^*,X^*)-V\rangle+\langle G(t^*,X^*),X^*-Y\rangle, \]
    by taking the $\limsup$ in \eqref{eq: W_n monotone solution hilbert}.
    At this point let us remark that convergence only hold for test functions $\varphi$ satisfying \eqref{eq: test function} for $\lambda\geq \lambda_0>0$, with $\lambda_0$ a constant independent of $n$, as a consequence in general $W$ is only a monotone solution on a subset of $H_{test}^\mathcal{H}$. This is not particularly concerning as said subset is sufficiently big to retain uniqueness. Indeed, observe that in the proof of Theorem \ref{thm: uniqueness hilbert}, we only use test functions with a $\lambda$ such that 
    \[\lambda (1+|x|^4)+F(x,\mu,W(t,x,\mu))\cdot x|x|^2\geq c\left(|x|^4-\int_{\reels^d} |y|^4\mu(dy)\right),\]
    so uniqueness of monotone solutions is not affected by such change. 
\end{proof}
\begin{remarque}
To avoid the need of such technicalities on the set of test functions, it is possible to include the necessity of taking sufficiently big $\lambda$ directly in the definition of monotone solution. indeed by defining $H_{test}^\mathcal{H}(\lambda_0)$ to be the set of functions $\varphi$ such that 
\begin{enumerate}
    \item[-]$\varphi:(t,X)\mapsto \phi(t,X)+\alpha e^{\lambda t} E_4(X)\in H_{test}^\mathcal{H},$
    \item[-] $\lambda\geq \lambda_0$.
\end{enumerate}
By our assumption on the growth of the family $(F_n,G_n,W_n)_{n\in\mathbb{N}}$, we can always restrict ourself to $H_{test}^\mathcal{H}(\lambda^*)$ for a $\lambda^*$ independent of $n\in\mathbb{N}$. In this case we say that $W$ is a monotone solution to \eqref{eq: ME without noise}, if it is a monotone solution for test functions in $H_{test}^\mathcal{H}(\lambda^*)$, where 
\[\lambda^*=\sup_{[0,T]\times\reels^d\times \mathcal{P}_2(\reels^d)} \frac{|F(x,\mu,W(t,x,\mu))|}{1+|x|+\sqrt{E_2(\mu)}}.\]

\end{remarque}

\subsection{Master equation with common noise}
\subsubsection{Master equation with a common noise process}
Instead of focus directly on master equation with additive common noise $\beta>0$, we first consider the case of the following master equation
\begin{equation}
\label{eq: ME with noise process}
\left\{
\begin{array}{c}
\displaystyle \partial_t W+F(x,\theta,m,W)\cdot \nabla_x W +\int_{\reels^d} F(y,\theta,m,W)\cdot D_m W(t,x,\theta,m)(y)m(dy)\\
\displaystyle +b(\theta)\cdot \nabla_\theta W-\frac{1}{2}\text{Tr}\left(\Sigma(\theta)\Sigma^T(\theta)D^2_p W\right)=G(x,\theta,m,W) \text{ in } (0,T)\times\reels^d\times\reels^m\times \mathcal{P}_2(\reels^d),\\
\displaystyle W(0,x,\theta,m)=W_0(x,\theta,m) \text{ for } (x,\theta,m)\in\reels^d\times\reels^m\times \mathcal{P}_2(\reels^d).
 \end{array}
 \right.
\end{equation}
Such master equation arises whenever the common noise affect directly the coefficients of the master equation through an additional variable $\theta$ instead of directly bumping measures (as it is the case for additive common noise). See \cite{noise-add-variable,common-noise-in-MFG} for more on such master equations. In terms of a forward backward system, \eqref{eq: ME with noise process} corresponds to 
\begin{equation}
\left\{
\begin{array}{cc}
     dX_t=-F(X_t,\theta_t,U_t,\mathcal{L}(X_t|\mathcal{F}^\theta_t))dt+\sqrt{2\sigma_x}dB^x_t& X|_{t=0}=X_0\\
     dU_t=-G(X_t,\theta_t,U_t,\mathcal{L}(X_t|\mathcal{F}_t))dt+Z_t\cdot d(B^x_t,B^\theta_t) & U_T=W_0(X_T,\theta_T,\mathcal{L}(X_T|\mathcal{F}_T)) \\
     d\theta_t=-b(\theta_t)dt+\Sigma(\theta_t)\cdot dB^\theta_t&\theta_0=\theta\\
     \mathcal{F}^\theta_t=\sigma\left((\theta_s)_{s\leq t}\right)
\end{array}
\right.
\end{equation}
While those master equations comes from interesting problems of modelisation, in the context of this article we mainly study them as a step toward the treatment of common noise in the master equation \eqref{eq: general ME}. Throughout this section we make the following assumption 
\begin{hyp}
\label{hyp: b lip noise process}
$b:\reels^m\to \reels^m$ and $\Sigma: \reels^m\to \mathcal{M}_m(\reels)$ are Lipschitz.
\end{hyp}
we also use the notation \[\forall \theta\in \reels^m \quad \Gamma(\theta)=\frac{1}{2}\Sigma(\theta)\Sigma^T(\theta).\]

We start with some reminder on finite dimensional viscosity solutions of second order \cite{crandall1992users}. For a function $u:\reels^m\to \reels$, we define the subjet of $u$ at $\theta_0\in \reels^m$, $J^-u(\theta_0)$ by
\[J^-u(\theta_0)=\left\{(p,A)\in \reels^m\times M_m(\reels), \quad u(\theta)\geq u(\theta_0)+p\cdot(\theta-\theta_0)+(\theta-\theta_0)\cdot A\cdot (\theta-\theta_0)+o\left(|\theta-\theta_0|^2\right) \right\}.\]
We are now going to use the subjet to define a notion of viscosity supersolution. To avoid the technicality of introducing subjets in infinite dimension, we take the rather unusual approach of treating finite dimensional terms coming from common noise using the semijets definition of viscosity solution, while we use test functions for variables living on a Hilbert space. 
Consider the following linear infinite dimensional PDE 
\begin{equation}
\label{general PDE in H with cn}
\partial_t u+f(t,X,\theta)\cdot \nabla_X u+b(\theta)\cdot \nabla_\theta u-\text{Tr}\left(\Gamma(\theta) D^2_\theta u\right)=g(t,X,\theta)
\end{equation}
\begin{definition}
    Suppose that $f$ is such that
    \[\forall \theta\in \reels^m, \exists C_\theta \quad |f(t,X,\theta)|\leq C(1+|X|+\|X\|).\]
    We say that a continuous function $u:[0,T]\times\mathcal{H}\times \reels^m$ is a viscosity supersolution of \eqref{general PDE in H with cn} if for every $(\varphi,\theta) \in H_{test}^\mathcal{H}\times \reels^n$, at a point of minimum $(t^*,X^*)\in (0,T)\times \mathcal{H}\cap L^4(\Omega,\reels^d)$ of $(t,X)\mapsto u(t,X,\theta)-\varphi(t,X)$, the following holds
    \[\forall (p,A)\in J^-u(t^*,X^*,\theta), \quad \partial_t \varphi+f(t^*,X^*,\theta)\cdot \nabla_X \varphi+b(\theta)\cdot p -\text{Tr}\left(\Gamma(\theta) A\right)\geq g(t,X,\theta),\]
    where $J^-u(t^*,X^*,\theta)$ indicates the subjet of 
    \[\eta\mapsto U(t^*,X^*,\eta)\]
    at $\theta$.
\end{definition}

\begin{definition}
\label{def: monotone solution with common noise}
    A continuous function $W:[0,T]\times \reels^d\times \reels^m\times \mathcal{P}_2(\reels^d)\to \reels^d$ with linear growth is said to be a monotone solution to \eqref{eq: ME without noise} if for any $(Y,V)\in \mathcal{H}^2$ the function $Z$ defined by 
    \[(t,X,p)\mapsto Z(t,X,p)=\langle W(t,X,p,\mathcal{L}(X))-V,X-Y\rangle,\]
    is a viscosity supersolution of 
    \begin{gather*}\partial_t Z+F(X,p,\mathcal{L}(X),W(t,X,p,\mathcal{L}(X)))\cdot \nabla_X Z+b(p)\cdot \nabla_p Z-\sigma_p \Delta_p Z \\
    \geq \langle F(X,p,\mathcal{L}(X),W(t,X,p,\mathcal{L}(X))), W(t,X,p,\mathcal{L}(X))-V\rangle+\langle G(X,p,\mathcal{L}(X),W(t,X,p,\mathcal{L}(X))),X-Y\rangle
    \end{gather*}
\end{definition}
\begin{thm}
    Under hypotheses \ref{hyp: linear growth in F} and \ref{hyp: weak monotonicity}, there exists at most one continuous monotone solution to \eqref{eq: ME with noise process}
\end{thm}
\begin{proof}
The idea is still the same as in Theorem \ref{thm: uniqueness hilbert}, the variable $\theta$ being treated by classical viscosity arguments. Consider two solutions $W^1,W^2$ and let
\[Z(t,s,X,Y,\theta_1,\theta_2)=\langle \tilde{W}^1(t,X,\theta_1)-\tilde{W}^2(s,Y,\theta_2),X-Y\rangle.\]
Let us assume by contradiction that there exists $(t,X,Y,\theta)\in [0,T]\times \mathcal{H}^2\times \reels^m$ and $\delta>0$ such that
\[Z(t,t,X,Y,\theta,\theta)\leq -\delta.\]
For $\alpha,\lambda,\gamma>0$ we define  
\[
\left\{
\begin{array}{l}
\psi_{\alpha,\lambda}:(t,s,X,Y,\theta_1,\theta_2)\mapsto \alpha \left(e^{\lambda t}(1+E_4(X)+|\theta_1|^3)+\frac{1}{T-t}+e^{\lambda s}(1+E_4(Y)+|\theta_2|^3)+\frac{1}{T-s}\right),\\
\psi_{\gamma,\lambda}:(t,s,\theta_1,\theta_2)\mapsto \frac{1}{2\gamma^3}\left(|t-s|^2+e^{\lambda(t+s)}|\theta_1-\theta_2|^2\right),\\
\psi_{\alpha,\lambda,\gamma}:(t,s,X,Y,\theta_1,\theta_2)\mapsto\psi_{\alpha,\lambda}(t,s,X,Y,\theta_1,\theta_2)+\psi_{\gamma,\lambda}(t,s,\theta_1,\theta_2).
\end{array}
\right.
\]
Thanks to our growth assumption on the growth of $W^1,W^2$, the function $Z+\psi_{\alpha,\lambda,\gamma}$ is bounded from below and lower semi-continuous. Hence, there exists a perturbation $(h_X,h_Y)$ such that $\|h_X\,\|h_Y\|\leq \alpha$ and 
\[Z_{\alpha,\lambda,\gamma}:(t,s,X,Y,\theta_1,\theta_2)\mapsto Z(t,s,X,Y,p,q)+\psi_{\alpha,\lambda,\gamma}(t,s,X,Y,\theta_1,\theta_2)+\langle h_X,X\rangle+\langle h_Y,Y\rangle, \]
reaches a minimum at $(t^\gamma,s^\gamma,X^\gamma,Y^\gamma,\theta_1^\gamma,\theta_2^\gamma)$. 
\paragraph{\textit{Step 1: properties of the minimum point  }}
\quad

For $\alpha$ sufficiently small 
\begin{equation}
\label{ineq: point of minimum for Hilbert space with cn}
Z_{\alpha,\lambda,\gamma}(t^\gamma,s^\gamma,X^\gamma,Y^\gamma,\theta_1^\gamma,\theta_2^\gamma)\leq -\frac{\delta}{2}.
\end{equation}
This implies that there exists a constant depending on the linear growth of $W^1$ and $W^2$ only, such that 
\begin{equation}
\label{ineq: bounds alpha hilbert}
\alpha (E_4(X^\gamma)+E_4(Y^\gamma)+|\theta^\gamma_1|^3+|\theta_2^\gamma|^3)\leq C(1+\|X^\gamma\|^2_{L^2}+\|Y^\gamma\|^2_{L^2}+|\theta^\gamma_1|^2+|\theta_2^\gamma|^2).
\end{equation}
As a consequence 
\begin{equation}
\label{bound L4 Ztheta}
\alpha \left(\sqrt{E_4(X^\gamma)}+\sqrt{E_4(Y^\gamma)}+ |\theta^\gamma_1|+ |\theta_2^\gamma|\right)\leq C',
\end{equation}
for another constant $C'$ independant of $\alpha$.
This bound yields for $\gamma \leq \alpha\leq 1$
\begin{equation}
\label{ineq: estimate on difference gamma hilbert}
|t^\gamma-s^\gamma|,|\theta_1^\gamma-\theta_2^\gamma|\leq C'\sqrt{\gamma},
\end{equation}
for $C'$ depending on $C$ only.

Letting 
\begin{gather*}Z_1: (\theta_1)\mapsto \langle \tilde{W}^1(t^\gamma,X^\gamma,\theta_1),X^\gamma-Y^\gamma\rangle,\\
Z_2: (\theta_2)\mapsto \langle \tilde{W}^2(t^\gamma,X^\gamma,\theta_2),Y^\gamma-X^\gamma\rangle,
\end{gather*}
First let us remark that 
\[(t,s,X,Y,\theta_1,\theta_2)\mapsto (D_{\theta_1}\psi_{\alpha,\lambda,\gamma},D_{\theta_2}\psi_{\alpha,\lambda,\gamma})(t,s,X,Y,\theta_1,\theta_2)\]
do not depend on $(X,Y)$. Thus we may forsake the dependency on those arguments for $D_{\theta_1}\psi_{\alpha,\lambda,\gamma},D_{\theta_2}\psi_{\alpha,\lambda,\gamma},D^2_{(\theta_1,\theta_2)}\psi_{\alpha,\lambda,\gamma}$ in what follows. 
Since $(\theta_1^\gamma,\theta_2^\gamma)$ is a minimum of \[(\theta_1,\theta_2)\mapsto Z_1(\theta_1)+Z_2(\theta_2)+\psi_{\alpha,\lambda,\gamma}(t^\gamma,s^\gamma,X^\gamma,Y^\gamma,\theta_1,\theta_2),\] we may apply Theorem 3.2 of \cite{crandall1992users} at this point to deduce that for any $\eta>0$ there exists $(A,B)\in\left(\mathcal{M}_m(\reels)\right)^2$ such that $(-D_{\theta_1}\psi_{\alpha,\lambda,\gamma}(t^\gamma,s^\gamma,\theta_1^\gamma,\theta_2^\gamma),A)\in J^-Z_1(\theta_1^\gamma)$, $(-D_{\theta_2}\psi_{\alpha,\lambda,\gamma}(t^\gamma,s^\gamma,\theta_1^\gamma,\theta_2^\gamma),B)\in J^-Z_2(\theta_2^\gamma)$ and 
\begin{equation}
\label{ineq: crandall lions hilbert}
\left(
\begin{array}{cc}
    A & 0_{\mathcal{M}_m(\reels)} \\
    0_{\mathcal{M}_m(\reels)} & B
\end{array}
\right)
\geq 
(M+\eta M^2),
\end{equation}
with $M=-D^2_{(\theta_1,\theta_2)}\psi_{\alpha,\lambda,\gamma}(t^\gamma,s^\gamma,\theta_1^\gamma,\theta_2^\gamma)$. Taking $\eta=\gamma^3e^{-2\lambda T}$, we deduce by multiplying \eqref{ineq: crandall lions hilbert} by
\[
\left(
\begin{array}{cc}
    \Gamma(\theta_1^\gamma) & \Sigma^T (\theta_1^\gamma)\Sigma(\theta_2^\gamma)  \\
     \Sigma(\theta_2^\gamma)^T \Sigma(\theta_1^\gamma)& \Gamma(\theta_2^\gamma)
\end{array}\right),
\]
that there exists such a couple $(A,B)$ satisfying 
\begin{equation}
\label{ineq tr(A+B) crandall lions hilbert}
-\text{Tr}\left(\Gamma(\theta_1^\gamma)A\right)-\text{Tr}\left(\Gamma(\theta^\gamma_2)B\right)\leq C\left(\alpha e^{\lambda t^\gamma}(1+|\theta^\gamma_1|^3)+\alpha e^{\lambda s^\gamma}(1+|\theta^\gamma_2|^3)+\frac{e^{\lambda (t^\gamma+s^\gamma)}}{\gamma^3}|\theta_1^\gamma-\theta_2^\gamma|^2\right),
\end{equation}
where $C$ is a constant depending on $m$ and $\|\Sigma\|_{Lip}$ only.
\paragraph{\textit{Step 2: Using the definition of monotone solution}}
\quad

The inequality \eqref{ineq: point of minimum for Hilbert space with cn} implies that $(s^\gamma,t^\gamma)\neq (0,0)$. Let us assume that they are both different from $0$. Applying the definition of a monotone solution to both $W^1$ and $W^2$ yields
\begin{gather*}
-\left(\partial_t+\tilde{F}_1(t^\gamma,X^\gamma,\theta_1^\gamma)\cdot \nabla_X+b(\theta_1^\gamma)\cdot \nabla_{\theta_1}\right)\psi_{\alpha,\lambda,\gamma}-\text{Tr}(\Gamma(\theta_1^\gamma) A)\geq \langle \tilde{F}_1(t^\gamma,X^\gamma,\theta_1^\gamma),W_1^\gamma-W_2^\gamma\rangle+\langle \tilde{G}_1(t^\gamma,X^\gamma,\theta_1^\gamma),X^\gamma-Y^\gamma\rangle\\
-\left(\partial_t+\tilde{F}_2(s^\gamma,Y^\gamma,\theta_2^\gamma)\cdot \nabla_Y+b(\theta_2^\gamma)\cdot \nabla_{\theta_2} \right)\psi_{\alpha,\lambda,\gamma}-\text{Tr}(\Gamma(\theta_2^\gamma) B)\geq \langle \tilde{F}_2(s^\gamma,Y^\gamma,\theta_2^\gamma),W_2^\gamma-W_1^\gamma\rangle+\langle \tilde{G}_2(s^\gamma,Y^\gamma,\theta_2^\gamma),Y^\gamma-X^\gamma\rangle
\end{gather*}
where we used the notation $F^i:(t,x,\theta,\mu)\mapsto F(x,\theta,\mu,W^i(t,x,\theta,\mu)), W_1^\gamma=W^1(t^\gamma,X^\gamma,\theta_1^\gamma,\mathcal{L}(X^\gamma))$ and so on. In light of \eqref{ineq tr(A+B) crandall lions hilbert}, there exists a $\lambda^*$ depending on the linear growth $W^1,W^2,F$, $\|\Sigma\|_{Lip},\|b\|_{Lip}$ and $m$ only such that for any $\lambda\geq \lambda^*$ we get by summing those two inequalities
\begin{align*}-\frac{\alpha}{(T-s^\gamma)^2}-\frac{\alpha}{(T-t^\gamma)^2}&\geq \langle \tilde{F}_1(t^\gamma,X^\gamma,\theta_1^\gamma)-\tilde{F}_2(s^\gamma,Y^\gamma,\theta_2^\gamma),W_1^\gamma-W_2^\gamma\rangle+\langle \tilde{G}_1(t^\gamma,X^\gamma,\theta_1^\gamma)-\tilde{G}_2(s^\gamma,Y^\gamma,\theta_2^\gamma),X^\gamma-Y^\gamma\rangle, \\
&\geq \langle F(Y^\gamma,\theta^\gamma_1,\mathcal{L}(Y^\gamma),W^2(s^\gamma,Y^\gamma,\theta_2^\gamma,\mathcal{L}(Y^\gamma))-F(Y^\gamma,\theta^\gamma_2,\mathcal{L}(Y^\gamma),W^2(s^\gamma,Y^\gamma,\theta_2^\gamma,\mathcal{L}(Y^\gamma)),W_1^\gamma-W_2^\gamma\rangle\\
&+\langle G(Y^\gamma,\theta^\gamma_1,\mathcal{L}(Y^\gamma),W^2(s^\gamma,Y^\gamma,\theta_2^\gamma,\mathcal{L}(Y^\gamma))-G(Y^\gamma,\theta^\gamma_2,\mathcal{L}(Y^\gamma),W^2(s^\gamma,Y^\gamma,\theta_2^\gamma,\mathcal{L}(Y^\gamma)),X^\gamma-Y^\gamma\rangle,
\end{align*}
where the last line is obtained by the joint monotonicity of $(F,G)$. Since $W_1^\gamma,W_2^\gamma,X^\gamma,Y^\gamma$ are bounded independently of $\gamma$ from \eqref{ineq: bounds alpha hilbert}, to get a contradiction, it is sufficient to show that 
\begin{equation}
\label{F gamma converge}
\underset{\gamma\to 0}{\lim} \underbrace{\|F(Y^\gamma,\theta^\gamma_1,\mathcal{L}(Y^\gamma),W^2(s^\gamma,Y^\gamma,\theta_2^\gamma,\mathcal{L}(Y^\gamma))-F(Y^\gamma,\theta^\gamma_2,\mathcal{L}(Y^\gamma),W^2(s^\gamma,Y^\gamma,\theta_2^\gamma,\mathcal{L}(Y^\gamma))\|}_{I_\gamma}=0,
\end{equation}
and that a similar consideration holds for the term in $G$. We show this is true for $F$ only, the convergence of the other term following from a similar argument. Since we work for fixed $\alpha$, the family $(Y^\gamma)_{\gamma>0}$ is uniformly bounded in $L^4(\Omega,\reels^d)$ by \eqref{bound L4 Ztheta}. Following an argument similar to the one presented in the proof of Theorem \ref{stability without noise} this means that the sequence of associated measures $(\mu_\gamma)_{\gamma>0}$ defined by $\mu_\gamma=\mathcal{L}(Y^\gamma)$, is relatively compact in $\mathcal{P}_2(\reels^d)$, and that there exists a sequence of compact $(\mathcal{K}_\varepsilon)_{\varepsilon>0}$, such that
\[\forall \varepsilon>0, \gamma>0, \quad \int_{\reels^d\textbackslash \mathcal{K}_\varepsilon}(1+|x|^2)\mu_\gamma(dx)\leq \varepsilon.\]
By the growth of $F,W^2$
\begin{equation}
\label{eq: I gamma}
I_\gamma^2\leq C(1+|\theta^\gamma_1|^2+|\theta^\gamma_2|^2)\varepsilon+\int_{\mathcal{K}_\varepsilon} |F(y,\theta_1^\gamma,\mu_\gamma,W^2(s^\gamma,y,\theta^\gamma_2,\mu_\gamma))-F(y,\theta_2^\gamma,\mu_\gamma,W^2(s^\gamma,y,\theta^\gamma_2,\mu_\gamma))|^2\mu_\gamma(dy).
\end{equation}
By the Heine–Cantor theorem we deduce that for fixed $\varepsilon$,
\[\underset{\gamma\to 0}{\lim}\int_{\mathcal{K}_\varepsilon} |F(y,\theta_1^\gamma,\mu_\gamma,W^2(s^\gamma,y,\theta^\gamma_2,\mu_\gamma))-F(y,\theta_2^\gamma,\mu_\gamma,W^2(s^\gamma,y,\theta^\gamma_2,\mu_\gamma))|^2\mu_\gamma(dy)=0.\]
Letting $\gamma$ tends to 0 and then $\varepsilon$ tends to 0 in \eqref{eq: I gamma} yields \eqref{F gamma converge}, leading to a contradiction for $\gamma$ sufficiently small.  
In contrast, if either $t^\gamma=0$ or $s^\gamma=0$ we also get a contradiction with \eqref{ineq: point of minimum for Hilbert space with cn} for $\gamma$ sufficiently small
\end{proof}
\begin{remarque}
    Let us mention that in the proof $(x,\theta,\mu,w)\mapsto (F,G)(x,\theta,\mu,w)$ need only be continuous in $\theta$ uniformly on compact set of $\reels^d\times \mathcal{P}_2(\reels^d)\times \reels^d$ (which is true if it is continuous). This is because the doubling of variable in $(X,Y)$ comes from considerations from monotonicity and not from technical viscosity. Obviously this proof also shows that the concept of monotone solution extend easily should $F,G$ be function of time and not depend on it only through $W$. Let us also remark that the linear growth of $(F,G,W)$ in $\theta$ can be easily relaxed to polynomial growth in this variable. 
    \end{remarque}
\begin{thm}
\label{stability with common noise}
 Suppose that there exists $(F_n,G_n,b_n,\Sigma_n)_{n\in\mathbb{N}}$, a sequence of function satisfying Hypotheses \ref{hyp: weak monotonicity}, \ref{hyp: linear growth in F} and \ref{hyp: b lip noise process}, uniformly in $n\in\mathbb{N}$ and converging locally uniformly to a couple $(F,G,b,\Sigma)$. Suppose there exists an associated sequence of monotone solution $(W_n)_{n\in\mathbb{N}}$ with linear growth uniformly in $n$. If $(W_n)_{n\in\mathbb{N}}$ converges locally uniformly to a function $W$, then $W$ is a monotone solution associated to the data $(F,G,b,\Gamma)$.
\end{thm}
\begin{proof}
The proof is very similar to Theorem \ref{stability without noise}, considering $(t^*,X^*,\theta^*)$ a point of minimum of 
\[(t,X,\theta)\mapsto \langle W(t,X,\theta,\mathcal{L}(X))-V,X-Y\rangle-\varphi(t,X,\theta),\]
we may always add a perturbation  \[\theta\mapsto -|\theta-\theta^*|^3,\]
to the test function $\varphi$ we consider, and there is no problem in creating points of minimum in $\reels^n$ since it is locally compact. Following the proof of Theorem \ref{stability without noise}, there exists a sequence $(t_n,X_n,\theta_n)_{n\in\mathbb{N}}$ minimum of $Z_n-\varphi$ with 
\[Z_n(t,X,\theta)=\langle W_n(t,X,\theta)-V,X-Y\rangle,\]
such that $(t_n,X_n,\theta_n)\longrightarrow (t^*,X^*,\theta^*)$. Since $(\nabla_\theta \varphi,D^2_\theta \varphi)(t_n,X_n,\theta_n)\in J^- Z_n(t_n,X_n,\theta_n)$ as $\theta_n$ is a point of minimum of $\theta \mapsto Z_n(t_n,X_n,\theta)-\varphi(t_n,X_n,\theta)$, we deduce using the viscosity property that 
\begin{gather*}
\left(\partial_t+\tilde{F}_n(t_n,X_n,\theta_n)\cdot \nabla_X+b_n(\theta_n)\cdot\nabla_\theta-\text{Tr}\left(\Gamma_n(\theta_n) D^2_\theta\right)\right) \varphi(t_n,X_n,\theta_n)\\
\geq \langle \tilde{F}_n(t_n,X_n,\theta_n),\tilde{W}_n(t_n,X_n,\theta_n)-V\rangle+\langle \tilde{G}_n(t_n,X_n,\theta_n),X_n-Y\rangle.
\end{gather*}
By the local uniform convergence of $(F_n,G_n,b_n,\Gamma_n,W_n)_{n\in \mathbb{N}}$ and up to a restriction on the set of test functions to $H_{test}^\mathcal{H}(\lambda^*)$ for $\lambda^*$ sufficiently big, we show as in Theorem \eqref{stability without noise} that
\begin{gather*}
\left(\partial_t+\tilde{F}(t^*,X^*,\theta^*)\cdot \nabla_X+b(\theta^*)\cdot\nabla_\theta-\text{Tr}\left(\Gamma(\theta^*) D^2_\theta\right)\right) \varphi(t^*,X^*,\theta^*)\\
\geq \langle \tilde{F}(t^*,X^*,\theta^*),\tilde{W}(t^*,X^*,\theta^*)-V\rangle+\langle \tilde{G}(t^*,X^*,\theta^*),X^*-Y\rangle.
\end{gather*}
Since this is true for any function $\varphi$ such that $(Z-\varphi)$ reaches a point of minimum at $(t^*,X^*,\theta^*)$, by the equivalence between the semijet definition and the definition through test functions of viscosity solutions (see \cite{flemingsoner} $V$, Lemma 4.1), $W$ is indeed a monotone solution in the sense of Definition \ref{def: monotone solution with common noise}.
\end{proof}
Although we do not make such development in this article, let us mention that the notion of monotone solution can be extended to situation in which the coefficients $b$ associated to the common noise process also depends on the measure argument 
\[b:\left\{
\begin{array}{l}
     \reels^n\times \mathcal{P}_2(\reels^d)\to \reels^n  \\
     (\theta,\mu)\mapsto b(\theta,\mu)
\end{array}
\right.\]
In this case appropriate monotonicity conditions taking this dependency in account are needed to show uniqueness of solutions, see \cite{common-noise-in-MFG} for more.
\subsubsection{Master equation with additive common noise}
Now that we have treated monotone solutions for the master equation \eqref{eq: ME with noise process}. We show how this notion extends to the situation situation of additive common noise
\begin{equation}
\label{eq: ME common noise change variable}
\left\{
\begin{array}{c}
\displaystyle \partial_t W+F(x,m,W)\cdot \nabla_x W+\int_{\reels^d} F(y,m,W)\cdot D_m W(t,x,m)(y)m(dy)-\beta \Delta_x W\\
\displaystyle -2\beta \int \text{div}_{x}[D_mW ](t,x,\noise,m,y)]m(dy)-\beta \int \text{Tr}\left[D^2_{mm}W(t,x,m,y,y')\right]m(dy)m(dy')\\
\displaystyle =G(x,m,W) \text{ in } (0,T)\times\reels^d\times\mathcal{P}_2(\reels^d),\\
\displaystyle W(0,x,m)=W_0(x,m) \text{ for } (x,m)\in\reels^d\times\mathcal{P}_2(\reels^d).
 \end{array}
 \right.
\end{equation}
This is based on the introduction of a new variable allowing to rewrite second order terms on the measure argument as as classic derivative with respect to a finite dimensionnal variable acting on the space of measure through the pushforward \cite{convergence-problem,common-noise-in-MFG}.
\begin{lemma}
\label{lemma: change variable smooth}
Let $W:[0,T]\times \reels^d\times \mathcal{P}_2(\reels^d)$ be a smooth solution to \eqref{eq: ME common noise change variable}, then the function 
\[\mathbf{W}(t,x,\theta,m)\mapsto W(t,x+\theta,(Id_{\reels^d}+\theta)_\#m),\]
is a smooth solution to 
\begin{equation}
\label{eq: ME common noise changed variable}
\begin{array}{c}
\displaystyle \partial_t \mathbf+\mathbf{F}(x,\theta,m,W)\cdot \nabla_x \mathbf{W}+\int_{\reels^d} \mathbf{F}(y,\theta,m,W)\cdot D_m\mathbf{W}(t,x,\theta,m)(y)m(dy)-\beta \Delta_\theta \mathbf{W}\\
\displaystyle =\mathbf{G}(x,\theta,m,\tilde{W}) \text{ in } (0,T)\times\reels^d\times\reels^d\times\mathcal{P}_2(\reels^d),\\
 \end{array}
 \end{equation}
 with $(\mathbf{F},\mathbf{G}):(x,\theta,m,w)\mapsto (F,G)(x+\theta,(Id_{\reels^d}+\theta)_\#m,w)$
\end{lemma}
\begin{proof}
Since W is a smooth function, the result is obtained by differentiating to get the equation satisfied by $\mathbf{W}$ and then noticing that 
\begin{gather*}
\Delta_\noise \mathbf{W}(t,x,\theta,m)\\
=\Delta_x \mathbf{W}(t,x,\theta,m)+\int_{\reels^d}\text{Tr}\left[D^2_{mm}\mathbf{W}\right](t,x,\theta,m,y,y')m(dy)m(dy')+2\int_{\reels^d}\text{div}_x\left[D_m\mathbf{W}\right](t,x,\theta,m,y)m(dy).
\end{gather*}
\end{proof}
The reason to introduce such change of variable is purely technical. Comparison of viscosity in finite dimension relies on a variation of Theorem 3.2 in \cite{crandall1992users}. Remarkably, this argument extend to separable Hilbert space \cite{second-order-viscosity-hilbert}, as a consequence we expect we could use it directly as terms in $\beta$ in \eqref{eq: ME common noise change variable} can be written of as second order terms in $\mathcal{H}$ \cite{Lions-college}. However an equivalent to this result is not, to the extent of our knowledge, available for functions of measures in general. In later sections (namely in the presence of idiosyncratic noise), the Hilbertian approach fails and we will work directly on master equations on the space of measures. Lemma \ref{lemma: change variable smooth}, allows us to bypass altogether the need to discuss about comparison with second order infinite dimensionnal terms with an argument that extends nicely outside the Hilbertian framework, by introducing the following definition for monotone solutions to \eqref{eq: ME common noise change variable}.
\begin{definition}
A continuous function $W$ is a monotone solution to \eqref{eq: ME common noise change variable} if $\mathbf{W}(t,x,\theta,m)\mapsto W(t,x+\theta,(Id_{\reels^d}+\theta)_\#m)$ is a monotone solution to \eqref{eq: ME common noise changed variable} in the sense of Definition \ref{def: monotone solution with common noise}
\end{definition}
\begin{corol}
Results of stability and uniqueness presented for \eqref{eq: ME with noise process} still holds 
\end{corol}
\begin{proof}
It is sufficient to see that $(\mathbf{W}_0,\mathbf{F},\mathbf{G})$ as defined in Lemma \ref{lemma: change variable smooth} inherit the regularity and growth of $(W_0,F,G)$ in the new variable $\theta$. As such the results of uniqueness and stability can be applied without difficulty with 
\[\Gamma\equiv\beta I_d, b\equiv 0.\]
\end{proof}
\begin{remarque}
Obviously it is possible to consider a general matrix $\Gamma=\Sigma\Sigma^T$ instead, for correlated common noise. Or even master equations featuring both additive common noise and a common noise process correlated with each others, see \cite{common-noise-in-MFG} remark 5.5. 
\end{remarque}
\begin{remarque}
    Let us also remark at this point that this treatment of common noise can be applied to notions of monotone solutions defined in different setting. In particular for monotone solutions to flat monotone mean field games in \cite{bertucci-monotone}, this trick allows to relax the definition of monotone solution with common noise by not requiring that solutions have a first order derivative with respect to the measure argument. 
\end{remarque}
\section{Existence of monotone solutions}
\label{section existence}
We start this section by giving some results on solutions to the following master equation. At this point, we focus on presenting Lipschitz solutions and give general results which are also valid in the presence of idiosyncratic noise, hence the presence of terms in $\sigma_x$ compared to \eqref{eq: ME without noise}. 
\begin{equation}
\label{eq: general ME with noise process}
\left\{
\begin{array}{c}
\displaystyle \partial_t W+F(x,\theta,m,W)\cdot \nabla_x W-\sigma_x \Delta_x W+b(\theta)\cdot\nabla_\theta W-\sigma_\theta \Delta_\theta W\\
\displaystyle +\int_{\reels^d} F(y,m,W)\cdot D_m W(t,x,m)(y)m(dy)-\sigma_x\int_{\reels^d} \text{div}_y (D_mW(t,x,m)(y))m(dy)\\
\displaystyle =G(x,\theta,m,W) \text{ in } (0,T)\times\reels^d\times\reels^n\times\mathcal{P}_2(\reels^d),\\
\displaystyle W(0,x,\theta,m)=W_0(x,\theta,m) \text{ for } (x,\theta,m)\in\reels^d\times \reels^n\times\mathcal{P}_2(\reels^d).
 \end{array}
 \right.
\end{equation}
In this first paragraph we present the tools we use to show the existence of monotone solutions. We start by recalling the notion of Lipschitz solutions to the master equation introduced in \cite{lipschitz-sol}. We show some properties of those solutions, in particular that they are monotone solutions. The key idea is to  combine this result with a lemma we prove on Lipschitz regularization of $L^2-$monotone functions. To show the existence of a monotone solution we will then show the existence of a Lipschitz solution for a sequence of regularized problem, and conclude thanks to the stability of monotone solutions under appropriate monotonicity assumptions. 
\subsection{Existence of Lipschitz solutions}
In this presentation of the notion of Lipschitz solution we make slight changes to suit our needs. Since the construction is very much based on idea presented in \cite{lipschitz-sol}, we carry out the proofs concisely in this exposition and refer readers interested in this notion to \cite{lipschitz-sol,cauchy-lipschitz-bertucci}, in which a more in depth analysis is carried out. The definition of Lipschitz solutions is based on a fixed point definition by integrating through the characteristics of the master equation. Consider the following linear PDE
\begin{equation}
\label{eq: linear lip sol}
\left\{
\begin{array}{c}
\displaystyle\partial_t V+A(t,x,\noise,m)\cdot \nabla_x V-\sigma_x \Delta_x V+b(\noise)\cdot\nabla_\noise V-\frac{1}{2}\text{Tr}\left(\Sigma\Sigma^T(\theta)D^2_\noise V\right)\\
\displaystyle +\int_{\reels^d} A(t,y,\noise,m)\cdot D_m V(t,x,\noise,m)(y)m(dy)-\sigma_x \int_{\reels^d} \text{div}_y (D_mV(t,x,\noise,m)(y))m(dy)\\
\displaystyle =E(t,x,\noise,m) \text{ in } (0,T)\times\reels^d\times\reels^n\times\mathcal{P}_q(\reels^d),\\
\displaystyle V(0,x,\noise,m)=V_0(x,\noise,m) \text{ for } (x,\noise,m)\in\reels^d\times\reels^n\times\mathcal{P}_q(\reels^d),
 \end{array}
 \right.
\end{equation}
for 3 vector fields $(A,E,V_0):(0,T)\times\reels^d\times\reels^n\times\mathcal{P}_q(\reels^d)\longrightarrow (\reels^d)^3\times\reels^n$. 
A solution of this linear system is given by integrating along the characteristics for all $t<T$
\begin{equation}
\label{eq: Feynman-kac measure}
\begin{array}{c}
\displaystyle V(t,x,\noise,\mu)=\espcond{V_0(X_t,\noise_t,m_t)+\int_0^t E(t-s,X_s,\noise_s,m_s)ds}{X_0=x,\noise_0=\noise,m_0=\mu},\\
\displaystyle dX_s=-A(t-s,X_s,\noise_s,m_s)ds+\sqrt{2\sigma_x}dB_s,\\
\displaystyle d\noise_s=-b(\noise_s)ds+\Sigma(\theta)\cdot dB^\noise_s,\\
dm_s=\left(-\text{div}\left(A(t-s,x,\noise_s,m_s)m_s\right)+\sigma_x \Delta_x m_s\right)ds,
\end{array}
\end{equation}
for $(B_s,B^\noise_s)_{s\geq 0}$ a $d+n$ dimensional Brownian motion. We claim that whenever $X\mapsto A(s,X,\noise,\mathcal{L}(X))$ is Lipschitz in $\mathcal{H}$ uniformly in $(s,\theta)\in (0,T)\times\reels^n$ and $b$ is Lipschitz, this system of a coupled SDE and SPDE $(X_s,\noise_s,m_s)_{s\in[0,t]}$ is well-defined (see \cite{probabilistic-mfg}). This relies on classic results on Lipschitz SDE and Corolary \ref{from lip hilbert to lip x}.
If $(\tilde{V}_0,\tilde{E})$ is also Lipschitz on $\mathcal{H}$, then so will be the solution $V$. Consider now the functional $\psi$ that to a Lipschitz $(A,E,V_0)$ associate this function V, i.e.
\[\psi(T,A,B,E,V_0)=\FuncDef{(0,T)\times\reels^d\times \reels^n\times\mathcal{P}_q(\reels^d)}{\reels^d,}{(t,x,\noise,\mu)}{V(t,x,\noise,\mu),}\]
where $V$ is given by \eqref{eq: Feynman-kac measure}. The definition of a Lipschitz solution to \eqref{eq: general ME with noise process} is given by a fixed point of this operator. 
\begin{definition}
\label{def: lipschitz sol W2}
Let $T>0$, $W:[0,T)\times \reels^d\times\reels^n\times\mathcal{P}_q(\reels^d)\to \reels^d$ is said to be a Lipschitz solution of \eqref{eq: general ME with noise process} if :
\begin{itemize}
    \item[-] $(x,\noise,\mu)\mapsto W(t,x,\noise,\mu)$ continuous on $\reels^d\times \reels^n\times \mathcal{P}_2(\reels^d)$ uniformly in $t\in[0,\alpha]$ for all $\alpha$ in $[0,T)$.
     \item[-] $(X,\theta)\mapsto W(t,X,\noise,\mathcal{L}(X))$ is Lipschitz on $\mathcal{H}\times \reels^n$ uniformly in $t\in[0,\alpha]$ for all $\alpha$ in $[0,T)$.
    \item[-] for all $t<T$: \[W=\psi(t,F(\cdot,W),G(\cdot,W),W_0).\]
\end{itemize}
\end{definition}
The main difference here, compared to the original paper, is in the Lipschitz norm we consider. Namely, the author considered function $(t,x,\mu)\mapsto W(t,x,\mu)$ Lipschitz in $\reels^d\times \mathcal{P}_2(\reels^d)$ locally in time, we show their result still hold if we consider instead continuous function $W$ such their lift on $\mathcal{H}$ $(t,X)\mapsto W(t,X,\mathcal{L}(X))$ is Lipschitz in $\mathcal{H}$.

Following what we just saw on \eqref{eq: general ME with noise process}, such a definition makes sense provided coefficients have  at least a Lipschitz lift on $\mathcal{H}$. We shall work under the following assumptions
\begin{hyp}
\label{hyp: Lipschitz Wq}
There exists a constant $C>0$ and a modulus of continuity $\omega(\cdot)$ such that
    \begin{enumerate}
    \item[-] Let $\tilde{F}:(X,U,\theta)\mapsto F(X,\theta,\mathcal{L}(X),U)$ be the lift of $F$ on $\mathcal{H}\times \reels^n\times \mathcal{H}$, defining similarly $\tilde{W}_0,\tilde{G}$, the following holds
    \[\|\tilde{F}(\cdot)\|_{Lip(\mathcal{H}^2\times \reels^n)},\|\tilde{G}(\cdot)\|_{Lip(\mathcal{H}^2\times \reels^n)},\|\tilde{W}_0(\cdot)\|_{Lip(\mathcal{H}\times \reels^n)}\leq C\]
        \item[-]$\|b\|_{Lip},\|\Sigma\|_{Lip}\leq C$
        \item[-] Letting $\bar{v}=(W_0,F,G)$, $\forall(x,p,\theta_1,\theta_2,\mu,\nu)\in \left(\reels^{d}\right)^2\times \reels^n\times (\mathcal{P}_2(\reels^d))^2$
        \[|\bar{v}(x,\theta_1,\mu,p)-\bar{v}(x,\theta_2,\nu,p)|\leq \omega\left(\mathcal{W}_2(\mu,\nu)+|\theta_1-\theta_2|\right)\]
    \end{enumerate}
\end{hyp}
\begin{remarque}
The modulus of continuity in $\mathcal{P}_2(\reels^d)$ can easily be relaxed to a local form, for example 
 \[|\bar{v}(x,\theta,\mu,p)-\bar{v}(x,\theta,\nu,p)|\leq C\left(1+|x|+|\theta|+|p|+\sqrt{\int_{\reels^d} |y|^2\mu(dy)}+\sqrt{\int_{\reels^d} |y|^2\nu(dy)}\right)\omega\left(\mathcal{W}_2(\mu,\nu)\right),\]
 to be consistent with the growth conditions imposed on $F,G,W_0$.
\end{remarque}
\begin{thm}
\label{lip sol b(p)}
       Under Hypothesis \ref{hyp: Lipschitz Wq} the following hold:
    \begin{enumerate}
        \item[-] There is always an existence time  $T>0$ such that there is a unique solution of \eqref{eq: general ME with noise process} in the sense of Definition \ref{def: lipschitz sol W2} on $[0,T)$.
        \item[-] There exist $T_c>0$ and a maximal solution $W$ defined on $[0,T_c[$ such that for all Lipschitz solutions $V$ defined on $[0,T)$: $T\leq T_c$ and $W|_{[0,T)}\equiv V$.
        \item[-] If $T_c<\infty$ then $\underset{t\to T_c}{\lim} \|\tilde{W}(t,\cdot)\|_{Lip(\mathcal{H}\times \reels^n)}=+\infty$.
    \end{enumerate}
\end{thm}
\begin{proof}
    The proof follows from argument mainly introduced in the original paper \cite{lipschitz-sol}, the main difference is the space in which the fixed point is carried out 
    Consider $E_M$ to be the set of all functions $W:[0,T)\times\reels^d\times \reels^n\times \mathcal{P}_2(\reels^d)\to \reels^d$ such that 
\[E_M=\left\{W: \sup_{t\in[0,T]}\|\tilde{W}(t,\cdot)\|_{Lip}\leq M,\underset{[0,T)\times\reels^d}{\sup} |W(\cdot,\theta_1,\mu)-W(\cdot,\theta_2,\nu)|\leq M\omega\left( M\mathcal{W}_2(\mu,\nu)+M|\theta_1-\theta_2|\right)\right\}.\]
We are going to show that for a well chosen couple $(T,M)$, for any $W\in E_M$,  
    \[\phi(W)=\psi(T,F(\cdot,W),G(\cdot,W),W_0),\]
    belongs to $E_M$. For now we first remark that 
\begin{equation}
\label{eq: representation formula in L2}
\phi(W)(t,X,p,\mathcal{L}(X))=\espcond{W_0(X_t,\noise^p_t,m_t)+\int_0^t G(X_s,\noise^o_s,m_s,W(t-s,X_s,\theta^p_s,m_s)ds}{X_0=X},
\end{equation}
with 
\begin{equation}
\label{eq: dynamic (x,theta)}
\left\{
\begin{array}{l}
dX_s=-F(X_s,\theta_s,\mathcal{L}(X_s),W(t-s,X_s,\theta_s,m_s)ds+\sqrt{2\sigma_x}dB_s,\\
d\theta^p_s=-b(\theta^p_s)ds+\Sigma(\theta)\cdot dB^\theta_s \quad \theta_0=p
\end{array}
\right.
\end{equation}
and 
\[m_s=\mathcal{L}(X_s|(\theta_u)_{u\leq s}).\]
We let $(Y_s,\theta_s^q)_{s\in[0,t]}$ be the solution to \eqref{eq: dynamic (x,theta)} for the same Brownian motion but with an initial condition $(Y_0,q)$. Denoting $\mathcal{F}^\theta_s=\sigma ((B^\theta_u)_{u\leq s})$ and using the fact $\tilde{F}$ and $\tilde{W}$ are Lipschitz in $\mathcal{H}\times \reels^n$, we get
\[\espcond{|X_s-Y_s|^2}{\mathcal{F}^\theta_s}+|\theta^p_s-\theta^q_s|^2\leq \|X_0-Y_0\|^2+|p-q|^2+\|\tilde{F}\|_{lip}(1+\|\tilde{W}\|_{lip})\int_0^s\left( \espcond{|X_v-Y_v|^2}{\mathcal{F}^\theta_s}+ |\theta_u^p-\theta_u^q|^2\right)dv \quad a.s. \]
By Gronwall's Lemma there exists a constant $C$ depending on $(t,\tilde{F},\tilde{W})$ such that 
\begin{equation}
\label{gronwall estimate measure}
\forall s\leq t \quad \espcond{|X_s-Y_s|^2}{\mathcal{F}_s^\theta}+|\theta_s^p-\theta_s^q|^2\leq C\|X_0-Y_0\|+|p-q|^2 \quad a.s.\end{equation}
In light of \eqref{eq: representation formula in L2} we deduce similarly that 
\begin{align*}
\|\phi(W)(t,X,p,\mathcal{L}(X)-\phi(W)(t,Y,q,\mathcal{L}(Y)\|^2\leq C'\left(\|X-Y\|^2+|p-q|^2\right),
\end{align*}
for a constant $C'$ depending on the Lipschitz norm of $(\tilde{F},\tilde{G},\tilde{W}_0,\tilde{W},b)$ in $\mathcal{H}\times \reels^n$ and t .In particular this implies that $\psi(W)$ is Lipschitz in $\reels^d$ from Corollary \ref{from lip hilbert to lip x}. We now show that it is continuous in $\reels^n\times \mathcal{P}_2(\reels^d)$. To that end, we turn to an estimate with respect to initial condition on the system
\[
\left\{
\begin{array}{l}
\displaystyle dX^{x,\mu,p}_s=-F(X^{x,\mu,p}_s,\noise^{p}_s,m^{\mu,p}_s,W(t-s,X^{x,\mu,p}_s,\noise^{p}_s,m^{\mu,p}_s))ds+\sqrt{2\sigma_x}dB_s\quad X_0=x,\\
\displaystyle d\noise^{x,\mu,p}_s=-b(\noise^{p}_s)ds+\Sigma(\theta)\cdot dB^\noise_s \quad \theta_0=p,\\
dm^{\mu,p}_s=\left(-\text{div}\left(F(x,\noise^{p}_s,m^{\mu,p}_s,W(t-s,x,\noise^{p}_s,m^{\mu,p}_s))m^{\mu,p}_s\right)+\sigma_x \Delta_x m^{x,\mu,p}_s\right)ds \quad m_0=\mu,
\end{array}\right.
\]
Taking the infimum over all couplings, the estimate \eqref{gronwall estimate measure} yields
\begin{equation}
\label{eq: prove lip sol exist}
\mathcal{W}_2(m_s^{\mu,p},m_s^{\nu,q})\leq C\left(\mathcal{W}_2(\mu,\nu)+|p-q|\right).
\end{equation}
Since $\tilde{F},\tilde{G}$ are Lipschitz in $\mathcal{H}$, $(F,G)$ are also Lipschitz in $x$, as a consequence an application of gronwall's lemma yields 
\[\forall s\leq t \quad \|X_s^{x,\mu,p}-X_s^{x,\nu,q}\|\leq C\omega\left(C|p-q|+C\mathcal{W}_2(\mu,\nu)\right),\]
for another constant also denominated $C$.
Finally using the representation formula for $\psi(W)$ we deduce that
\[|\phi(W)(t,x,p,\mu)-\phi(W)(t,x,q,\nu)|\leq C\omega\left(C|p-q|+C\mathcal{W}_2(\mu,\nu)\right),\]
for a constant C depending on the Lischitz norm fo $(\tilde{G},\tilde{F},\tilde{W}_0,b)$ and $t$. Since all of those constants are given by Gronwall type estimate, we know that for $t=0$ they depend only on $W_0$ and grow exponentially with time. In particular, we deduce that for $M$ sufficiently big and $T$ sufficiently small $\psi$ sends $E_M$ into itself.
We now remark that for this $E_M$ the growth of the function in $\phi(E_M)$ is controlled in such way that
\[\forall t<T,\exists C>0 \quad \forall W\in\phi(E_M), \quad \forall (s,x,\theta,\mu)\in [0,t]\times\reels^d\times\reels^n\times \mathcal{P}_2(\reels^d), \quad |W(t,x,\theta,\mu)|\leq C\left(1+|\theta|+|x|+\sqrt{\int_{\reels^d} |y|^2\mu(dy)}\right).\]
This means $\phi(E_M)$ endowed with the norm $\|\cdot\|_M$ defined by
\[\|W\|_M=\underset{[0,T)\times \mathcal{H}\times \reels^n}{\sup}\frac{\|W(t,X,\theta,\mathcal{L}(X)\|}{1+|\theta|+\|X\|},\]
is a complete metric space. For existence and uniqueness, the rest of the proof is then done exactly as in \cite{lipschitz-sol} $\phi$ being a contraction for the norm 
\[\|\cdot\|_{\mathcal{H},\infty}: W\mapsto \underset{[0,T)\times \mathcal{H}\times \reels^n}{\sup}\|W(t,X,\theta,\mathcal{L}(X))\|,\]
for $T$ sufficiently small. As for the blowup condition it remains to show that so long as the Lipschitz norm of its lift does not blow up a Lipschitz solution $W$ is indeed continuous on $\reels^d\times \reels^n\times\mathcal{P}_2(\reels^d)$. For that we first remark that the constant in \eqref{eq: prove lip sol exist} depends only on the data and the Lipschitz norm of $\tilde{W}$. Consequently, so long as $\sup_{s\in[0,t]}\|\tilde{W}(t,\cdot)\|_{Lip}<+\infty$ we can find constants such that 
\[|W(t,x,\theta_1,\mu)-W(t,x,\theta_2,\nu)|\leq C\omega(C\mathcal{W}_2(\mu,\nu),\]
by Gronwall's Lemma. 
\end{proof}
\begin{remarque}
For the case of purely Hilbertian master equation, which is to say master equation that do not result from a lifting on $\mathcal{H}$ but rather are posed from the start on a Hilbert space this notion of solution is equivalent to the one presented in \cite{lipschitz-sol} Section 2. In some sense, we are just showing here that this notion of solution is also valid for master equation with idiosyncratic noise. 
\end{remarque}

\subsection{Link with monotone solutions}
We now show that under reasonable growth assumptions, Lipschitz solutions are monotone solutions. For master equations on finite state space for which a notion of solution was defined in \cite{bertucci-monotone-finite}, this result was proven in \cite{Bertucci-peccot}. This can be seen as a natural extension of this result to master equation in infinite dimension and with noise.
\begin{lemma}
\label{lemma: lip sol are monotone without idio}
    under Hypothesis \ref{hyp: Lipschitz Wq}, letting $W$ be the Lipschitz solution to \eqref{eq: ME with noise process} defined on $[0,T_c)$ for some $T_C>0$, we obtain that $W$ is also a monotone solution. 
\end{lemma}
\begin{proof}
 By the uniqueness of Lipschitz solution and the flow property for Lipschitz SDE, fixing $(t,X,\theta)\in[0,T_c)\times \mathcal{H}\times \reels^n, s\leq t$, the following dynamic programming principle is satisfied by the lift $\tilde{W}$ of $W$
\begin{equation}
\label{eq: lip sol hilbert DPP}
\begin{array}{c}
\displaystyle \tilde{W}(t,X,\noise)=\espcond{\tilde{W}(s,X_{t-s},\noise_{t-s})+\int_0^{t-s}\tilde{G}(X_u,\noise_u,\tilde{W}(t-u,X_u,\theta_u))du}{X_0=X,\theta_0=\theta},\\
\displaystyle dX_u=-\tilde{F}(X_u,\noise_u,\tilde{W}(t-u,X_u,\theta_u))du\\
\displaystyle d\noise_u=-b(\noise_u)ds+\Sigma(\theta)\cdot dB^\noise_u,
\end{array}
\end{equation}
Let $(V,Y)\in\mathcal{H}^2$, it follows naturally that for any $s\leq t$
\[(\tilde{W}(t,X,\noise)-V)\cdot(X-Y)=\espcond{(\tilde{W}(s,X_{t-s},\noise_{t-s})-V)\cdot(X-Y)+\int_0^{t-s}\tilde{G}(X_u,\noise_u,\tilde{W}(t-u,X_u,\theta_u))du}{X_0=X,\theta_0=\theta,V,Y}.\]
By taking the expection in this equation and noticing that
\[X=X_{t-s}+\int_0^{t-s}\tilde{F}(X_u,\theta_u,\tilde{W}(t-u,X_u,\theta_u))du,\]
we can show that $\forall s\leq t$
\[Z(t,X,\theta)=\esp{Z(s,X_{t-s},\theta_{t-s})}+\int_0^{t-s}\left(\langle \tilde{G}(t-u,X_u,\theta_u),X-Y\rangle+\langle \tilde{F}(t-u,X_u,\theta_u),\tilde{W}(s,X_{t-s},\theta_{t-s})-V\rangle\right) du,\]
where $Z(t,X,\theta)=\langle \tilde{W}(t,X,\theta)-V,X-Y\rangle$. Suppose for some function $\varphi: [0,T)\times \mathcal{H}\times \reels^n \to \reels$,
\[(Z-\varphi)(t,X,\theta)=\min_{[0,T)\times\mathcal{H}\times \reels^n}(Z-\varphi)=0.\]
Using the dynamic programming principle satisfied by $Z$, it follows that $\forall s\leq t$
\[\varphi(t,X,\theta)\geq\esp{\varphi(s,X_{t-s},\theta_{t-s})}+\int_0^{t-s}\left(\langle \tilde{G}(t-u,X_u,\theta_u),X-Y\rangle+\langle \tilde{F}(t-u,X_u,\theta_u),\tilde{W}(s,X_{t-s},\theta_{t-s})-V\rangle\right) du.\]
At this point, if $\varphi$ is sufficiently smooth, that is if it admits frechet derivatives that are bounded on bounded sets of $\mathcal{H}\times\reels^n$, an infinite dimensionnal version of Ito's lemma, see \cite{Gawarecki2011} theorem 2.10 for exemple, let us conclude to 
\begin{gather}
\label{ineq: varphi sursol proof lip}
\left(\partial_t+\tilde{F}(t,X,\theta)\cdot \nabla_X+b(\theta)\cdot\nabla_\theta-\text{Tr}\left(\Gamma(\theta) D^2_\theta\right)\right) \varphi(t,X,\theta)\\
\geq \langle \tilde{F}(t,X,\theta),\tilde{W}(t,X,\theta)-V\rangle+\langle \tilde{G}(t,X,\theta),X-Y\rangle.
\end{gather}
by multiplying the equation by $\frac{1}{t-s}$ and taking the limit as $s$ tends to $t$. The slight technicality here is that our set of test functions admits function which are not frechet differentiable in $\mathcal{H}$, namely $X\mapsto \esp{X^4}$. Let us show that this inequality is still satisfied by $\varphi$ so long as the point of minimum lies in $L^4(\Omega,\reels^d)$. First we claim that in that case $(X_s)_{s\in[0,t]}$ is uniformly bounded in $L^4(\Omega,\reels^d)$. Indeed the growth assumption \ref{hyp: linear growth in F} ensures that $F$ also has linear growth in $L^4(\Omega,\reels^d)$, the claim then following from Gronwall lemma. Indeed, letting
\[M=\sup_{s\in[0,t]} \esp{X_s^4},\]
we define \[H_4=\{ X\in \mathcal{H}, \esp{X^4}\leq M\}.\]
$H_4$ is a closed subset of a Hilbert space and so it follows $\mathcal{H}_4=(H_4,\langle \cdot,\cdot \rangle)$ is also a Hilbert space. On this space the function
\[E_4:X\mapsto \esp{X^4},\]
is frechet differentiable (and in fact smooth). Applying Ito's lemma in this space we also get that $\varphi$ satisfies \eqref{ineq: varphi sursol proof lip}, in particular this show this is true for any $\varphi \in H_{test}$. Since this result is valid for any $[0,T)\times \mathcal{H}\times \reels^n$, this ends to show $W$ is a monotone solution to \eqref{eq: ME with noise process} by \cite{flemingsoner} $V$, Lemma 4.1.
\end{proof}
\begin{remarque}
This proof does not rely on monotonicity assumptions on $W_0,F,G$, 
\end{remarque}
\begin{remarque}
For this proof to work we only need a growth condition on $F$ of the form 
\[\tilde{F}(X,U,\theta)\leq C_\theta(1+|X|+|U|+\|X\|+\|U\|),\]
allowing an adaptation of this lemma to equation posed directly on a Hilbert space. 
\end{remarque}
In particular for the master equation \eqref{eq: ME common noise change variable} we recall the following result on Lipschitz solutions which is a consequence of  \cite{common-noise-in-MFG} Lemma 5.3.
\begin{lemma}
\label{lip sol chang variable}
Under Hypothesis \ref{hyp: Lipschitz Wq}, letting $V$ be the Lipschitz solution to \eqref{eq: ME common noise change variable} defined on $[0,T_c)$ for some $T_C>0$, there exists a function $W;[0,T_c)\times \reels^d\times \mathcal{P}_2(\reels^d)\to \reels^d$ such that 
\[\forall (t,x,\theta,\mu)\in[0,T_c)\times\reels^d\times \reels^n\times \mathcal{P}_2(\reels^d), \quad V(t,x,\theta,\mu)=W(t,x+\theta,(Id_{\reels^d}+\theta)_\# \mu).\]
\end{lemma}
\subsection{Lipschitz solutions and FBSDE}
We now present a result on the link between Lispchitz solutions and forward backward systems. To be more precise, whenever they exists Lipschitz solutions are the decoupling field to the mean field forward backward system associated to \eqref{eq: general ME with noise process}. 
\begin{lemma}
\label{lemma: lip sol FBSDE}
    Under Hypothesis \ref{hyp: Lipschitz Wq}, let $W$ be the Lipschitz solution to \eqref{eq: general ME with noise process} defined on $[0,T_c)$ for some $T_c>0$ and consider for $t<T_c$
\[
\left\{
\begin{array}{l}
     dX_s=-F(X_s,\theta_s,\mathcal{L}(X_s|\mathcal{F}^\theta_s),W(t-s,X_s,\theta_s,\mathcal{L}(X_s|\mathcal{F}^\theta_s)))ds+\sqrt{2\sigma_x}dB_s, \quad X_0=X,  \\
     d\theta_s=-b(\theta_s)ds+\Sigma(\theta_s)\cdot dB^\theta_s, \quad \theta_0=\theta,\\
     \mathcal{F}^\theta_s=\sigma( (B^\theta_u)_{u\leq s}).
\end{array}
\right.
\]
Let $\mathcal{F}_s=\sigma((B_u,B^\theta_u)_{u\leq s}),$ then there exists a square integrable $\mathcal{F}-$adapted process $(Z_s)_{s\in[0,t]}$, such that for the process
\[\forall s\leq t \quad W_s=W(t-s,X_s,\theta_s,\mathcal{L}(X_s|\mathcal{F}^\theta_s)),\]
the following holds
\[
\left\{
\begin{array}{l}
     \displaystyle X_s=X-\int_0^sF(X_u,\theta_u,\mathcal{L}(X_u|\mathcal{F}^\theta_u),W_u)du+\sqrt{2\sigma_x}B_s,\\
     \displaystyle \theta_s=\theta-\int_0^s b(\theta_u)du+\int_0^s\Sigma(\theta_u)\cdot dB^\theta_u,\\
     \displaystyle W_s=W_0(X_t,\theta_t,\mathcal{L}(X_t|\mathcal{F}^\theta_t))+\int_s^t G(X_u,\theta_u,\mathcal{L}(X_u|\mathcal{F}^\theta_u),W_u)du-\int_t^T Z_s\cdot d(B_s,B^\theta_s).
\end{array}
\right.
\]
\end{lemma}
\begin{proof}
Since $x\mapsto W(s,x,\theta,\mu)$, $(X,\theta)\mapsto W(s,X,\theta,\mathcal{L}(X))$ are Lipschitz respectively in $C(\reels^d,\reels^d)$ and $C(\mathcal{H}\times \reels^n,\mathcal{H})$, uniformly in the other variables, it follows naturally that there exists a constant $C$ such that
\[\forall (s,x,\theta,X)\in [0,t]\times \reels^d\times \reels^n\times \mathcal{H}, \quad |W(s,x,\theta,\mathcal{L}(X))|\leq C(1+|x|+|\theta|+\|X\|).\]
By Gronwall type estimates, $(X_s,\theta_s,W_s)_{s\in[0,t]}$ have uniformly bounded second moments. Following this statement the process $(M_s)_{s\in [0,t]}$ defined by 
\[M_s=W_s+\int_0^s G(X_u,\theta_u,\mathcal{L}(X_u|\mathcal{F}^\theta_u),W_u)du,\]
is square integrable. Let $\mathcal{F}_s=\sigma (B_u,B^\theta_u)_{u\leq s})$, by the flow property and the uniqueness of Lipchitz solutions for any $s\leq t$ 
\[M_s=\espcond{M_t}{\mathcal{F}_s}.\]
Since $(M_s)$ is a square integrable $\mathcal{F}-$martingale, by the martingale representation theorem there exists a square integrable $\mathcal{F}-$adapted process $(Z_s)_{s\in [0,t]}$ such that 
\[\forall s\leq t, \quad M_s=\int_0^s Z_u\cdot d(B_u,B^\theta_u).\]
The conclusion then follows from the definition of $(M_s)_{s\in[0,t]}$.
\end{proof}
We now present a Lemma on the Lipschitz regularization of $L^2-$monotone function, which is the main reason we use Lipschitz solutions in this article. This statement is a variation on the Hille-Yosida regularization of monotone operators
\begin{lemma}
\label{lemma: regul monotone function}
    Let $F:\reels^d\times \mathcal{P}_2(\reels^d)\to\reels^d$ be a $L^2-$monotone continuous function. Then there exists a sequence of function $(F_\varepsilon)_{\varepsilon>0}$ satisfying the following properties:
\begin{enumerate}
\item[-] for all $\varepsilon>0$, $F_\varepsilon$ is a continuous $L^2-monotone$ function,
    \item[-] for all $\varepsilon>0$, $\tilde{F}_\varepsilon$ is Lipschitz on $\mathcal{H}$ with constant $\frac{1}{\varepsilon}$,
    \item[-] $(F_\varepsilon)_{\varepsilon>0}$ converges locally uniformly to $F$ in $\reels^d\times \mathcal{P}_2(\reels^d)$.
\end{enumerate}
If furthermore on each compact set $\mathcal{O}\subset \reels^d\times\mathcal{P}_2(\reels^d)$, there exists a modulus $\omega_\mathcal{O}(\cdot)$ such that 
\[\forall (x,\mu),(y,\nu)\in \mathcal{O}, \quad |F(x,\mu)-F(y,\nu)|\leq \omega_\mathcal{O}\left(|x-y|+\mathcal{W}_2(\mu,\nu)\right),\]
then there also exists a modulus $\bar{\omega}_\mathcal{O}(\cdot)$ such that uniformly in $\varepsilon$
\[\forall (x,\mu),(y,\nu)\in \mathcal{O}, \quad |F_\varepsilon(x,\mu)-F_\varepsilon(y,\nu)|\leq \bar{\omega}_\mathcal{O}\left(|x-y|+\mathcal{W}_2(\mu,\nu)\right).\]
\end{lemma}
\begin{proof}
    Let $G_\varepsilon:(x,\mu)\mapsto x+\varepsilon F(x,\mu)$. We first observe that $\tilde{G}_\varepsilon$ is strongly $L^2-$monotone uniformly in $\varepsilon$, i.e.
    \begin{equation}
    \label{strong monotonicity yoshida}
    \forall \varepsilon\geq 0, \forall (X,Y)\in \mathcal{H}^2 \quad \langle \tilde{G}_\varepsilon(X)-\tilde{G}_\varepsilon(Y),X-Y\rangle \geq \|X-Y\|^2.
    \end{equation}
    This implies that for any $\varepsilon\geq 0$, $\tilde{G}_\varepsilon$ has a 1 Lipschitz inverse in $\mathcal{H}$. We now introduce the regularization of $\tilde{F}$ in $\mathcal{H}$
    \[\tilde{F}_\varepsilon(X)=F\circ\tilde{G}_\varepsilon^{-1}(X).\]
    We first show that for any $\varepsilon>0$ this procedure indeed regularize $\tilde{F}$ as $\tilde{F}_\varepsilon$ is $\frac{1}{\varepsilon}$ Lipschitz in $\mathcal{H}$. For any $(X,Y)\in \mathcal{H}^2$ 
\begin{align*}
    \langle \tilde{F}_\varepsilon(X)-\tilde{F}_\varepsilon(Y),X-Y\rangle&=\langle \tilde{F}_\varepsilon(X)-\tilde{F}_\varepsilon(Y),\tilde{G}_\varepsilon \circ\tilde{G}_\varepsilon^{-1} (X)-\tilde{G}_\varepsilon\circ \tilde{G}_\varepsilon^{-1}(Y)\rangle\\
    &=\langle \tilde{F}_\varepsilon(X)-\tilde{F}_\varepsilon(Y),\tilde{G}_\varepsilon^{-1}(X)-\tilde{G}_\varepsilon^{-1}(Y)+\varepsilon(\tilde{F}_\varepsilon(X)-\tilde{F}_\varepsilon(Y))\rangle\\
    &\geq \varepsilon \| \tilde{F}_\varepsilon(X)-\tilde{F}_\varepsilon(Y)\|^2 \text{ by the monotonicity of $F$},
\end{align*}
the Lipschitz bound $\|\tilde{F}\|_{Lip(L^2)}\leq \frac{1}{\varepsilon}$, then follows naturally from Cauchy-Schwartz inequality. In particular this also shows $\tilde{F}_\varepsilon$ is monotone for any $\varepsilon>0$. We now show that $\tilde{F}_\varepsilon$ can be obtained as the lift on $\mathcal{H}$ of a continuous function $F_\varepsilon: \reels^d\times \mathcal{P}_2(\reels^d)\to \reels^d$. 
    We first remark that the inverse $\tilde{G}_\varepsilon^{-1}$ satisfies the following property. We show the following equivalent statement 
    \begin{equation}
    \label{equiv tilde G law}
    \mathcal{L}(\tilde{G}_\varepsilon(X))=\mathcal{L}(\tilde{G}_\varepsilon(Y))\implies \mathcal{L}(X)=\mathcal{L}(Y).\end{equation}
    By the strong monotonicity of $\tilde{G}_\varepsilon$ \eqref{strong monotonicity yoshida}, 
    \begin{equation}
    \label{G epsilon bijection measure}
    \forall (X,Y)\in \mathcal{H}^2,\quad X\sim \mu, Y\sim \nu, \quad \|\tilde{G}_\varepsilon(X)-\tilde{G}_\varepsilon(Y)\|\geq \mathcal{W}_2(\mu,\nu).\end{equation}
    We now claim the following: for any $Z\in \mathcal{H}$ with $\mathcal{L}(Z)=\mathcal{L}(\tilde{G}_\varepsilon(X))$, there exists a random variable $X_Z\in\mathcal{H}$ such that $Z=\tilde{G}_\varepsilon(X_Z)$.
    Since $F$ is continuous, by Lemma \ref{prop: expectation to pointwise} $x\mapsto F(x,\mu)$ is monotone for any $\mu\in \mathcal{P}_2(\reels^d)$. This implies for fixed $\mu\in \mathcal{P}_2(\reels^d)$ $x\mapsto G_\varepsilon(x,\mu)$ has a $1-$Lipschitz inverse we denote $G_\varepsilon^{-1}(\cdot,\mu)$. By definition letting $X_Z=\mathcal{L}(G_\varepsilon^{-1}(Z,\mathcal{L}(X)))$
   \[\mathcal{L}(X_Z)=\mathcal{L}(X).\]
   In particular 
   \begin{align*}
       Z=G_\varepsilon(X_Z,\mathcal{L}(X))=G_\varepsilon(X_Z,\mathcal{L}(X_Z))=\tilde{G}_\varepsilon(X_Z).
   \end{align*}
   As a consequence taking the infimum over all couplings in \eqref{G epsilon bijection measure} yields 
   \[\forall (X,Y)\in \mathcal{H}^2, \quad \mathcal{W}_2(\mathcal{L}(\tilde{G}_\varepsilon(X)),\mathcal{L}(\tilde{G}_\varepsilon(Y)))\geq \mathcal{W}_2(\mathcal{L}(X),\mathcal{L}(Y)).
    \]
    This allows the definition of a mapping $h_\varepsilon:\mathcal{P}_2(\reels^d)\to \mathcal{P}_2(\reels^d)$ defined by 
    \[h_\varepsilon(\mu)=\mathcal{L}(\tilde{G}_\varepsilon^{-1}(X)) \text{ for any } X\in \mathcal{H} \text{ with } \mathcal{L}(X)=\mu.\]
    In fact $h_\varepsilon$ is Lipschitz in $\mathcal{W}_2$ with constant 1. Using this mapping, $\tilde{G}_\varepsilon^{-1}$ can be expressed as the lift of a function $g_\varepsilon$ on $\reels^d\times \mathcal{P}_2(\reels^d)$ with
    \[g_{\varepsilon}(x,\mu)= G^{-1}_\varepsilon(x,h_\varepsilon(\mu)).\]
We now turn to an estimate on the regularity of $g_\varepsilon$ in $\mathcal{P}_2(\reels^d)$. Let $(x,\mu,\nu)\in \reels^d\times\mathcal{P}_2(\reels^{2d})$, 
\begin{align*}
    g_\varepsilon(x,\mu)-g_\varepsilon(x,\nu)&=\bar{x}-g_\varepsilon (G_\varepsilon(\bar{x},h_\varepsilon(\mu)),\nu) \text{ with } \bar{x}= g_\varepsilon(x,\mu)\\
    &=g_\varepsilon(G_\varepsilon(\bar{x},h_\varepsilon(\nu)),\nu)-g_\varepsilon (G_\varepsilon(\bar{x},h_\varepsilon(\mu)),\nu)\\
    &\leq |G_\varepsilon(\bar{x},h_\varepsilon(\mu))-G_\varepsilon(\bar{x},h_\varepsilon(\nu))|\\
    &\leq \varepsilon|F(\bar{x},h_\varepsilon(\mu))-F(\bar{x},h_\varepsilon(\nu))|.\\
\end{align*}
By the continuity of $F$, 
\[\lim_{\nu\to \mu} |g_\varepsilon(x,\mu)-g_\varepsilon(x,\nu)|=0.\]
We now notice that $\tilde{F}_\varepsilon$ is the lift on $\mathcal{H}$ of 
\[F_\varepsilon:(x,\mu)\mapsto F(g_\varepsilon(x,\mu),h_\varepsilon(\mu)).\]
As a composition of continuous functions, it is continuous. 
It remains to show that our approximation $F_\varepsilon$ converges locally uniformly to $F$ in $\reels^d\times \mathcal{P}_2(\reels^d)$. It suffices to observe that for any $\mu\in \mathcal{P}_2(\reels^d)$
\begin{align*}
    \mathcal{W}_2(\mu,h_\varepsilon(\mu))&\leq \|X-\tilde{G}^{-1}_\varepsilon(X)\| \text{ for any } X \text{ with law } \mu\\
    &\leq \varepsilon \|  \tilde{F}(\tilde{G}_\varepsilon^{-1}(X))\|,
\end{align*}
and that similarly for any $(x,\mu)\in \reels^d\times\mathcal{P}_2(\reels^d)$
\[|x-g_\varepsilon(x,\mu)|\leq \varepsilon |F_\varepsilon(x,\mu)|.\]
Those are sufficient to establish the uniform convergence of $F_\varepsilon$ to $F$ on compacts sets of $\reels^d\times \mathcal{P}_2(\reels^d)$. 

Let us finally consider a compact set $\mathcal{O}_m$ of $\mathcal{P}_2(\reels^d)$ and $\mathcal{O}_x$ of $\reels^d$. by the continuity of $g_\varepsilon,h_\varepsilon$, the image $g_\varepsilon(\mathcal{O}_x,\mathcal{O}_m), h_\varepsilon(\mathcal{O}_m)$ are compact. As a consequence there exists a compact set $\mathcal{K}\subset\reels^d\times \mathcal{P}_2(\reels^d)$ such that $g_\varepsilon(\mathcal{O}_x,\mathcal{O}_m)\times h_\varepsilon(\mathcal{O}_m)\subset \mathcal{K}$. Let us assume $F$ admits a continuity modulus on each compact set, we deduce that for each $(x,\mu),(y,\nu)\in \mathcal{O}_x\times \mathcal{O}_m$
\begin{align*}
|F_\varepsilon(x,\mu)-F_\varepsilon(y,\nu)|&\leq \omega_{\mathcal{K}}\left(|g_\varepsilon(x,\mu)-g_\varepsilon(y,\nu)|+\mathcal{W}_2(h_\varepsilon(\mu),h\varepsilon(\nu))\right)\\
&\leq \omega_{\mathcal{K}}\left(|x-y|+\varepsilon\omega_\mathcal{K}(\mathcal{W}_2(\mu,\nu))+\mathcal{W}_2(\mu,\nu)\right),
\end{align*}
by previous estimates on the continuity of $g_\varepsilon,h_\varepsilon$. Since for any compact set $\mathcal{O}$ of $\reels^d\times \mathcal{P}_2(\reels^d)$ we can find a compact set $\mathcal{O}_x$ of $\reels^d$ and $\mathcal{O}_m$ of $\mathcal{P}_2(\reels^d)$ such that $\mathcal{O}\subset \mathcal{O}_x\times \mathcal{O}_m$, this ends the proof of the last claim. 
\end{proof}
\begin{remarque}
Many adaptions can be made to this result to consider different situations, for example if $F$ has a global continuity modulus then so does $F_\varepsilon$ uniformly in $\varepsilon>0$. If the modulus of $F$ is explicit then this result can be further refined for exemple if $F$ is known to be $\gamma-$Holder on $\reels^d\times\mathcal{P}_2(\reels^d)$ then there exists a constant $C>0$ independent of $\varepsilon$ such that 
\[\forall (x,y,\mu,\nu)\in \reels^{2d}\times\mathcal{P}_2(\reels^{2d}) \quad |F_\varepsilon(x,\mu)-F_\varepsilon(y,\nu)|\leq C\left(|x-y|^\gamma+\mathcal{W}_2(\mu,\nu)^\gamma+\varepsilon \mathcal{W}_2(\mu,\nu)^{\gamma^2}\right)\]
A notable extension of this result also includes the class of semi monotone function in $\mathcal{H}$, indeed if there exists a constant $C$ such that 
\[\forall (X,Y)\in \mathcal{H}^2 \quad \langle \tilde{F}(X)-\tilde{F}(Y),X-Y\rangle \geq -C \|X-Y\|^2,\]
then Lemma \ref{lemma: regul monotone function} can be applied to $(x,\mu)\mapsto F(x,\mu)+Cx$.
\end{remarque}
\begin{corol}
\label{corol: linear growth regul}
Let $F:\reels^d\times \mathcal{P}_2(\reels^d)\to \reels^d$ be a continuous $L^2-$monotone function. Assume furthermore that there exists a constant $C_F$ such that 
\[\forall (x,\mu)\in \reels^d\times \mathcal{P}_2(\reels^d),\quad |F(x,\mu)|\leq C_F\left(1+|x|+\sqrt{\int_{\reels^d}|y|^2\mu(dy)}\right),\]
Let $(F_\varepsilon)_{\varepsilon>0}$ be the regularization of $F$ as introduced in Lemma \ref{lemma: regul monotone function}. For $\varepsilon<\frac{1}{C_F}$ the following holds 
\[\forall (x,\mu)\in\reels^d\times \mathcal{P}_2(\reels^d), \quad |F_\varepsilon(x,\mu)|\leq \frac{1+C_F}{1-C_F\varepsilon}\left(1+|x|+\sqrt{\int_{\reels^d}|y|^2\mu(dy)}\right).\]
\end{corol}
\begin{proof}
    Reusing the same notations as in the above proof, there exists functions $g_\varepsilon:\reels^d\times \mathcal{P}_2(\reels^d)\to \reels^d, h_\varepsilon:\mathcal{P}_2(\reels^d)\to \mathcal{P}_2(\reels^d)$, such that  
    \[\forall (x,\mu) \in \reels^d\times \mathcal{P}_2(\reels^d), \quad F_\varepsilon(x,\mu)=F(g_\varepsilon(x,\mu),h_\varepsilon(\mu)).\]
    Consequently to get growth estimates on the family $(F_\varepsilon)_{\varepsilon>0}$, it is sufficient to estimate the growth of those two functions. 
    
    by definition of $\tilde{G}_\varepsilon^{-1}$ the following holds 
    \[\forall X\in \mathcal{H}, \quad X=\tilde{G}^{-1}_\varepsilon(X)+\varepsilon \tilde{F}(\tilde{G}^{-1}_\varepsilon(X)).\]
    This implies 
    \[\|\tilde{G}^{-1}_\varepsilon(X)\|\leq \|X\|+\varepsilon\|\tilde{F}(\tilde{G}^{-1}_\varepsilon(X)\|.\]
    Using the growth of $\tilde{F}$, by definition of $h_\varepsilon(\mu)$ we get 
    \[(1-C_F\varepsilon)\sqrt{\int_{\reels^d}|y|^2 h_\varepsilon(\mu)(dy)}\leq C_F+(1+C_F)\sqrt{\int_{\reels^d} |y|^2\mu(dy)}.\]
    as for $(x,\mu)\mapsto g_\varepsilon(x,\mu)$ we already know it is 1-Lipschitz in $x$ uniformly in $\mu$, and that 
    \[\forall (x,\mu)\in \reels^d\times \mathcal{P}_2(\reels^d) \quad |x-g_\varepsilon(x,\mu)|\leq \varepsilon |F(g_\varepsilon(x,\mu),h_\varepsilon(\mu))|\]
\end{proof}
\subsection{A key estimate   on mean field type FBSDE}
In this section we are interested in getting estimates on the following mean field type FBSDE
\begin{equation}
\label{eq: FBSDE}
    \left\{
    \begin{array}{l}
         \displaystyle X_t=X_0-\int_0^t \tilde{F}(X_s,W_s)ds+\sigma_x B_t,  \\
         \displaystyle W_t=\tilde{W}_0(X_T)+\int_t^T \tilde{G}(X_s,W_s)ds-\int_t^T Z_sdB_s.
    \end{array}
    \right.
\end{equation}
Following Lemma \ref{lemma: lip sol FBSDE}, this will translates into results we can use for Lipschitz solutions. In fact since those estimates do not rely on the fact coefficients are Lipschitz, but rather on monotonicity conditions, they, at least formally, still hold for monotone solutions and will be key for our existences results. 

We now introduce different monotonicity conditions under which we are able to get the existence of monotone solutions. In particular we will see (as is well known) the stronger the monotonicity of the coefficent the less regularity we can ask on them.
\begin{hyp}
\label{hyp: weak strong monotonicity in G}
    \begin{gather}\nonumber\exists \alpha>0,L\geq 0,a_0>0, \quad \forall (X,Y,U,V)\in [0,T]\times \left(\mathcal{H}\right)^4\\
\nonumber\langle \tilde{W}_0(X)-\tilde{W}_0(Y),X-Y\rangle \geq a_0 \|\tilde{W}_0(X)-\tilde{W}_0(Y)\|^2,\\
    \nonumber \langle \tilde{G}(X,U)-\tilde{G}(Y,V),X-Y\rangle+\langle \tilde{F}(X,U)-\tilde{F}(Y,V),U-V\rangle\geq 0,\\
    \label{ineq: weak strong monotonicity in G}
    \langle \tilde{G}(X,U)-\tilde{G}(Y,V),X-Y\rangle+\langle \tilde{F}(X,U)-\tilde{F}(Y,V),U-V\rangle \geq \alpha \|X-Y\|^2-L \|U-V\|^2.
    \end{gather}
\end{hyp}
We refer to this kind of assumption on the monotonicity of $(F,G)$ as weak-strong monotonicity in constrast with strong monotonicity which would requires this result to hold for $L=0$. The second inequality is exactly asking that the couple $(F,G)$ be $L^2-$monotone. As for \eqref{ineq: weak strong monotonicity in G}, it is akin to asking that the monotonicity of the system should not be too degenerate in $X$ even though the monotonicity of the system as a whole can be degenerate (hence the term weak-strong). For exemple in finite dimension the matrix 
\[M=\left(\begin{array}{cc}
    1 & 1 \\
    1 & 1
\end{array}\right),\]
is only degenerate monotone but 
\[\forall \bar{x}=\left(\begin{array}{c}
     x  \\
     u 
\end{array}\right)\in \reels^2, \quad \bar{x}^T M\bar{x}\geq x^2-u^2.
\]
More generally the following holds 
\begin{exemple}
Suppose the couple $(F,G)$ is jointly $L^2-$monotone, i.e 
\[\forall (X,Y,U,V)\in \mathcal{H}^4, \langle \tilde{G}(X,U)-\tilde{G}(Y,V),X-Y\rangle+\tilde{F}(X,U)-\tilde{F}(Y,V),U-V\rangle\geq 0.\]
Further assume that $X\mapsto \tilde{G}(X,U)$ is $\alpha-$strongly monotone uniformly in $U$ that is to say
\[\forall (X,Y,U) \in \mathcal{H}^3,\quad \langle \tilde{G}(X,U)-\tilde{G}(Y,U),X-Y\rangle \geq \alpha \|X-Y\|^2.\]
If $U\mapsto \tilde{G}(Y,U)$, $X\mapsto \tilde{F}(X,V)$ are Lipschitz uniformly in $(Y,V)$ then for any $\tilde{\alpha}<\alpha$, there exists a corresponding $L$ such that \eqref{hyp: weak strong monotonicity in G} is satisfied by the couple $(\tilde{G},\tilde{F})$ for the constants $(\tilde{\alpha},L)$.
\end{exemple}

We start with a growth estimate with respect to initial conditions, as well as time continuity of the forward process. 
\begin{lemma}
\label{montone fbsde have bounded growth}
    Suppose that for some initial datum $X_0\in \mathcal{H}$, there exists a strong solution $(X_t,W_t,Z_t)_{t\in[0,T]}\subset \mathcal{H}^3$ to \eqref{eq: FBSDE}. Under Hypotheses \ref{hyp: weak strong monotonicity in G} and \ref{hyp: linear growth in F}, There exists a constant depending only on $a_0,\alpha,L,T$ and the growth of $(F,G,W_0)$ such that 
    \[\forall s\leq t\leq T,\quad  \|X_t\|,\|W_t\|, \frac{\|X_t-X_s\|}{\sqrt{t-s}}\leq C(1+\|X_0\|).\]
\end{lemma}
\begin{proof}
    Let $(Y_t)_{t\in[0,T]}$ be the stochastic process defined by
    \[Y_t=\sigma_x B_t.\]
    Letting 
    \[U_T=W_T\cdot(X_T-Y_T),\]
    Using the assumption \ref{hyp: weak strong monotonicity in G} on $W_0$, it is obvious that
    \[ a_0\|W_T-W_0(Y_T)\|^2\leq \esp{U_T+|W(Y_T)\cdot(X_T-Y_T)|}. \]
    Applying Ito's Lemma to $(U_t)_{t\in[0,T]}$, and using the monotonicity of $(F,G)$ it follows that 
    \[\forall (t,s)\in[0,T]^2,s<t \quad \esp{U_t}\leq \esp{U_s}+\underbrace{\int_s^t (\|\tilde{G}(Y_u,0)\cdot(X_u-Y_u)\|+\|\tilde{F}(Y_u,0)\cdot W_u\|)du}_{C^t_s}.\]
    In parallel, the weak strong monotonicity of $(F,G)$ implies 
    \[\forall t\leq T \quad \alpha \int_t^T \|X_s-Y_s\|^2ds\leq L\int_t^T \|W_s\|^2ds+\esp{U_t}+\|W_0(Y_T)\cdot(X_T-Y_T)\|+C_t^T.\]
    At this point, we remark that there exists a constant depending on $T$ and the growth of $W_0,F,G$ such that 
    \[\forall t\leq T, \quad \|Y_t\|,\|W_0(Y_t)\|,\|\tilde{F}(Y_t,0)\|,\|\tilde{G}(Y_t,0)\|\leq C.\]
    Applying Ito's Lemma to $t\mapsto |W_t|^2$ and taking the expectation, we deduce by the growth of $G$, that
    \[\forall t\leq T, \quad \|W_t\|^2\leq \|W_T\|^2+C\left(1+\int_t^T(\|X_s\|^2+\|W_s\|^2)ds\right).\]
    Owning to previous bounds 
    \[\forall t\leq T ,\quad \|W_t\|^2\leq C\left(\int_t^T\|W_s\|^2ds+1+C_0^T+\esp{U_0}+\|W_0(Y_T)\cdot(X_T-Y_T)\|\right).\]
    By a backward application of Gronwall's lemma, it follows that for another constant depending on $\alpha,L,a_0,T$ and the growth of $F,G,W_0$ only
    \[\forall t\leq T, \quad \|W_t\|^2\leq C\left(1+C_0^T+\esp{U_0}+\|W_0(Y_T)\cdot(X_T-Y_T)\|\right).\]
    
    In particular evaluating at $t=0$ and integrating from $0$ to $T$, Young's inequality allows us to deduce by worsening the constant that 
    \[\forall t\leq T, \quad \|W_t\|^2\leq 2C^2\left(1+\|X_0\|^2+\|W_0(Y_T)\cdot(X_T-Y_T)\|+\int_0^T\|\tilde{G}(Y_u,0)\cdot(X_u-Y_u)\|du\right).\]
    We now turn to an estimate on $\|X_t\|^2$. Thanks to the above bounds, and using the growth of $F$, an application of Gronwall's lemma yields
    \[\forall t\leq T, \quad \|X_t\|^2\leq C'\left(1+\|X_0\|^2+\|W_0(Y_T)\cdot(X_T-Y_T)\|+\int_0^T\|\tilde{G}(Y_u,0)\cdot(X_u-Y_u)\|du\right).\]
    Evaluating for $t=T$, and integrating on $[0,T]$ allows us once again to get rid of terms depending on $(X_t)_{t\in[0,T]}$ by Young's inequality, proving the claim of growth for the couple $(X_t,W_t)_{t\in[0,T]}$. The time continuity of $(X_t)_{t\in[0,T]}$ then follows naturally from 
    \[\forall s\leq t\leq T,\quad \|B_t-B_s\|=\sqrt{t-s}.\]
\end{proof}
we now state our regularity assumptions on the coefficients, to that end we introduce the following notation 
\[\forall (X,Y)\in \mathcal{H}^2, \quad \|X-Y\|_\gamma=\max\left(\|X-Y\|,\|X-Y\|^\gamma\right)\]
\begin{hyp}
\label{hyp: regularity: weak strong monotonicity in G}
there exists  $\gamma \in (0,1]$ such that
    \begin{enumerate}
    \item[-] There exists a constant $C$ such that $\forall (X,Y,U,V)\in \mathcal{H}^4$
    \begin{enumerate}
        \item[-] $\langle \tilde{G}(X,U)-\tilde{G}(Y,V),U-V\rangle \leq C\left(\|X-Y\|^2+\|U-V\|^2\right)$,
        \item[-] $\langle \tilde{F}(X,U)-\tilde{F}(Y,V),X-Y\rangle \geq -C\left(\|X-Y\|^2+\|U-V\|^{2}_\gamma\right)$,
         \end{enumerate}
    \item[-] There exists a continuity modulus $\omega(\cdot)$ such that $\forall (x,y,\mu,\nu)\in (\reels^d)^2\times \left(\mathcal{P}_2(\reels^d)\right)^2$
     \[|F(x,u,\mu)-F(x,u,\nu)|+|G(x,u,\mu)-G(x,u,\nu)|+|W_0(x,\mu)-W_0(x,\nu)|\leq \omega(\mathcal{W}_2(\mu,\nu))\]
    \end{enumerate}
\end{hyp}
\begin{remarque}
    this strong monotonicity assumption on $W_0$ implies that it's lift on $\mathcal{H}$ is a Lipschitz function. Let us also remark that this does not seems to implies that $(x,\mu)\mapsto W_0(x,\mu)$ is Lipschitz on $\reels^d\times \mathcal{P}_2(\reels^d)$. Indeed this does not give any information on the behaviour of $\mu \mapsto W_0(x,\mu)$ for $x$ outside the support of $\mu$. Similar considerations holds for $G$ and $F$, hence the additional regularity assumption with respect to the measure argument. 
\end{remarque}
\begin{remarque}
We are essentially asking that $(X,U)\mapsto \tilde{G}(X,U)$ be half-Lipschitz in $U$ and Lipschitz in $X$, and that $(X,U)\mapsto \tilde{F}(X,U)$ be semi-monotone in $X$ and $\gamma-$holder in $U$. The fact that one sided Lipschitz bounds are sufficient in mean field games has already been observed in \cite{Lions-one-sided-lipschitz}. Moreover when $(F,G)=(\nabla_p H,-\nabla_x H)$ the semi-monotonicity of $F$ is very much related to the fact that $G$ is half Lipschitz since it holds that 
\[\partial_p G(x,\mu,p)=-\partial_x F(x,\mu,p).\]
\end{remarque}
We give a first a priori estimate on the forward backward system \eqref{eq: FBSDE}.
\begin{lemma}
\label{estimate W strong in G}
Under Hypotheses \ref{hyp: weak strong monotonicity in G},\ref{hyp: regularity: weak strong monotonicity in G} and \ref{hyp: linear growth in F}, there exists a constant $C$ such that if there exists a  strong solution $(X,W^x,Z^x)_{t\in[0,T]}$ in $\mathcal{H}^3$ to \eqref{eq: FBSDE} (resp. $(Y,W^y,Z^y)_{t\in[0,T]}$) with initial datum $X_0\in \mathcal{H}$ (resp. $Y_0$),
\[ \forall t\leq T, \quad \|W_t^x-W_t^y\|\leq C\|X_0-Y_0\|\]
\end{lemma}
\begin{proof}
We start by a classic estimates stemming from monotonicity. Applying Ito's Lemma to 
    \[U_T=(X_T-Y_T)\cdot(W_T^x-W_T^y),\] yields 
    \[\forall t\in [0,T], \quad U_T=U_t-\int_t^T \left((F^x_s-F^y_s)\cdot (W^x_s-W^y_s)+(G^x_s-G^y_s)\cdot(X_s-Y_s)\right)ds,\]
    where we have used the notation $F^x_s=\tilde{F}(X_s,W^x_s),G^x_s=\tilde{G}(X_s,W^x_s)$ and so on. 
    By Hypothesis \eqref{hyp: weak strong monotonicity in G}, there exists $a_0>0$, such that 
    \[\esp{U_T}\geq a_0\esp{|W^x_T-W^y_T|^2}.\]
The $L^2-$monotonicity of the couple $(F,G)$ implies that 
\begin{equation}
\label{ineq: ut geq uT geq terminal X}
\forall t\in[0,T] \quad \esp{U_t}\geq \esp{U_T}\geq a_0\esp{|W^x_T-W^y_T|^2}.
\end{equation}
Using the monotonicity of $W_0$ and Hypothesis \ref{hyp: weak strong monotonicity in G} we deduce that 
\begin{equation}
\label{estimate X by W G strong weak monotone}
\forall t\in [0,T] \quad \alpha \int_t^T \|X_s-Y_s\|^2ds\leq L \int_t^T \|W^x_s-W^y_s\|^2ds+ \esp{U_t}.
\end{equation}
Keeping this estimate in mind, we apply Ito's lemma to 
$t\mapsto |W_t^x-W^y_t|^2$ to deduce that 
\begin{align*}
\esp{|W_T^x-W^y_T|^2}&=\esp{|W_t^x-W^y_t|^2-2\int_t^T (G^x_s-G^y_s)\cdot (W^x_s-W^y_s)ds+\int_t^T |Z^x_s-Z^y_s|^2ds}, \\
&\geq \esp{|W_t^x-W^y_t|^2-2\int_t^T (G^x_s-G^y_s)\cdot (W^x_s-W^y_s)ds}.
\end{align*}
Using Hypothesis \ref{hyp: weak strong monotonicity in G} on the regularity of $\tilde{G}$, combined with \eqref{estimate X by W G strong weak monotone} yields that there exists a constant $C$ such that 
\[\forall t\leq T \quad  \|W^x_t-W^y_t\|^2\leq C\esp{U_0}+\int_t^T \|W_s^x-W_s^y\|^2ds,\]
where we have used \eqref{ineq: ut geq uT geq terminal X} to estimate $\|W^x_T-W^y_T\|$.
By Gronwall's lemma we finally get 
\[\forall t\leq T  \quad \|W^x_t-W^y_t\|^2\leq C \esp{U_0} e^{C(T-t)}.\]
Since $U_0=(W_0^x-W_0^y)\cdot (X_0-Y_0)$, evaluating this expression at $t=0$ yields the existence of a constant $C'$ such that 
\[\forall t\leq T \quad  \|W^x_t-W^y_t\|^2 \leq C'\esp{\|X_0-Y_0\|^2}e^{C'(T-t)}\]
\end{proof}
\begin{remarque}
    Let us mention that at this point, no regularity on $F$ was used beyond the one implied by the joint monotonicity of $(F,G)$ and the weak strong monotonicity Hypothesis \ref{hyp: weak strong monotonicity in G} to get this a priori estimate. Similarly since we have only used monotonicity and regularity conditions on $\mathcal{H}$, it does not matter whether the coefficients are of the form
    \[(\tilde{F},\tilde{G})(X,W)=(F,G)(X,W,\mathcal{L}(X,W)),\]
    justifying that our results are indeed valid for mean field games of controls. 
\end{remarque}

\begin{remarque}
We insist again on the fact that so long as $(x,\mu)\mapsto W(t,x,\mu)$ is continuous, a Lipschitz estimate in $L^2$ also implies a Lipschitz estimate on $x\mapsto W(t,x,\mu)$ with the same constant. In fact continuity of the lift of $W$ in $L^2$ combined with continuity of $W$ in the space variable $x$ are sufficient to concluded that if there exists two monotone solutions $W^1,W^2$ then 
\[\forall (t,x) \in[0,T]\times \text{Supp}(\mu) \quad W^1(t,x,\mu)=W^2(t,x,\mu).\]
As a consequence by changing slightly the definition of monotone solution, this estimate may be sufficient to get the existence of a unique solution, in which case the assumptions on the regularity of $F$ can be greatly reduced following the previous remark. However in what follow we are interested in solutions $x\mapsto W(t,x,\mu)$ well-defined even outside the support of $\mu$, even if for problems coming from MFG this additional regularity may not be necessary. 
\end{remarque}

We now give a stability estimate on the equation satisfied by the forward process, since whenever there exists a decoupling field $W$, $\mu_t=\mathcal{L}(X_t)$ is a weak solution (in the sense of distribution) to 
\[\partial_t \mu_t-\text{div}\left(F(x,\mu_t,W(T-t,x,\mu_t))\mu_t\right)+\sigma_x\Delta_x \mu_t,\]
this will also translates into a continuity estimates with respect to the measure argument for solutions of the master equation.  
\begin{lemma}
\label{estimate x strong in G}
Under Hypotheses \ref{hyp: weak strong monotonicity in G},\ref{hyp: regularity: weak strong monotonicity in G} and \ref{hyp: linear growth in F},there exists a constant $C$ such that if there exists a  strong solution $(X,W^x,Z^x)_{t\in[0,T]}$ in $\mathcal{H}^3$ to \eqref{eq: FBSDE} (resp. $(Y,W^y,Z^y)_{t\in[0,T]}$) with initial datum $X_0\in \mathcal{H}$ (resp. $Y_0$),
\[\forall t\leq T\quad  \|X_t-Y_t\|\leq C\|X_0-Y_0\|_\gamma. \]
\end{lemma}
\begin{proof}
Applying Ito's Lemma to $t\mapsto |X_t-Y_t|^2$ we get 
\begin{align*}
\forall t\leq T \quad \esp{|X_t-Y_t|^2}&=\esp{|X_0-Y_0|^2-2\int_0^t (F^x_s-F^y_s)\cdot(X_s-Y_s)ds}\\
&\leq \|X_0-Y_0\|^2+C\int_0^t \|X_s-Y_s\|^2 ds+\int_0^t \|W^x_s-W^y_s\|_\gamma^2 ds,
\end{align*}
using the semi-monotonicity of $F$ and its holder regularity in $W$. Finally from Lemma \ref{estimate W strong in G}, the concavity of $x\mapsto x^\gamma$ and gronwall's lemma we deduce that 
\[\|X_t-Y_t\|^2\leq Ce^{Ct} \left(\|X_0-Y_0\|^2+\|X_0-Y_0\|_\gamma^2\right).\]
\end{proof}
We now have the tools at hand to prove the following 
\begin{lemma}
\label{existence in the Lipschitz framework}
Under Hypothesis \ref{hyp: Lipschitz Wq} and hypothesis \ref{hyp: weak strong monotonicity in G}, there exists a Lipschitz solution to \eqref{eq: ME without noise} in the sense of definition \ref{def: lipschitz sol W2} on $[0,+\infty)$.
\end{lemma}
\begin{proof}
Local existence of a solution $W$ until a time $T_c>0$ follows from Theorem \ref{existence in the Lipschitz framework}. By Lemma \ref{lemma: lip sol FBSDE}, for any $t<T_c$, the estimates we just proved are valid for this solution, with global constants since $(\tilde{F},\tilde{G})$ is Lipschitz. In particular following Lemma \ref{estimate W strong in G}, for any $T>0$ there exists a finite constant depending only on the data and $T$, $C_T$ such that  
\[\forall t<T_c\wedge T\quad \|\tilde{W}(t,\cdot)\|_{Lip}\leq C_T.\]
If $T_c<T$ for some $T>0$, this would lead to a contradiction as 
\[\lim_{t\to T_c} \|\tilde{W}(t,\cdot)\|_{Lip}=+\infty.\]
\end{proof}
Before using this result to show the existence of monotone solution, we prove this last lemma
\begin{lemma}
\label{lemma: estim regul measure}
    For $i=\{1,2\}$, Let $(\mu^i_t)_{t\in[0,T]}\subset \mathcal{P}_2(\reels^d)$, be such that 
    \[\forall t\in [0,T] \quad \mathcal{W}_2(\mu^1_t,\mu^2_t)\leq \delta_{\mu},\]
    for some $\delta_{\mu}\geq 0$. Suppose that for $i=\{1,2\}$, there exists a strong solution $(X^i_t,W^i_t,Z^i_t)_{t\in [0,T]}$ to the system 
    \[
    \left\{
    \begin{array}{l}
         \displaystyle X^i_t=x-\int_0^t F(X^i_s,W^i_s,\mu^i_s)ds+\sigma_x B_t,  \\
         \displaystyle W^i_t=W_0(X^i_T,\mu^i_T)+\int_t^T G(X^i_s,W^i_s,\mu^i_s)ds-\int_t^T Z^i_sdB_s.
    \end{array}
    \right.
    \]
    Under Hypotheses \ref{hyp: weak strong monotonicity in G} and \ref{hyp: regularity: weak strong monotonicity in G} the following holds 
    \[\forall t\in[0,T], \quad \|W^1_t-W^2_t\|,\|X^1_t-X^2_t\|\leq C\max\left(\omega(\delta_\mu),\omega(\delta_\mu)^\frac{\gamma}{2-\gamma}\right).\]
\end{lemma}
\begin{proof}
Since $(F,G,W_0)$ are continuous, by Proposition \ref{prop: expectation to pointwise} we deduce that the following holds $\forall (x,y,u,v,\mu,\nu)\in (\reels^d)^4\times \left(\mathcal{P}_2(\reels^d)\right)^2$
\begin{align*}
(G(x,u,\mu)-G(y,v,\nu))\cdot(x-u)+(F(x,u,\mu)-F(y,v,\nu))\cdot(u-v)\geq -(|u-v|+|x-y|)\omega(\mathcal{W}_2(\mu,\nu)),\\
(G(x,u,\mu)-G(y,v,\nu))\cdot(x-u)+(F(x,u,\mu)-F(y,v,\nu))\cdot(u-v)\geq \frac{\alpha}{2}|x-y|^2-2L|u-v|^2-C\left(\omega(\mathcal{W}_2(\mu,\nu)\right)^2,\\
(W_0(x,\mu)-W_0(y,\nu))\cdot(x-y)+|x-y|\omega(\mathcal{W}_2(\mu,\nu))+(\omega(\mathcal{W}_2(\mu,\nu)))^2\geq |W_0(x,\mu)-W_0(y,\nu)|^2. 
\end{align*}
We start as in Lemma \ref{estimate W strong in G}: by applying Ito's lemma to 
\[U_T=(X^1_T-X^2_T)\cdot(W^1_T-W^2_T),\]
to deduce the following 
\begin{align*}
   \exists C>0,\quad  \forall t\in [0,T], a_0\|W_T^1-W_T^2\|^2&\leq \esp{U_T}+\|X^1_T-X^2_T\|\omega(\delta_\mu)+\left(\omega(\delta_\mu)\right)^2\\
    &\leq \underbrace{\esp{U_0}}_{=0}+\underbrace{\|X^1_T-X^2_T\|\omega(\delta_\mu)+\left(\omega(\delta_\mu)\right)^2+\omega(\delta_\mu)\int_0^T\left(\|X^1_s-X^2_s\|+\|W^1_s-W^2_s\|\right)ds}_{C_0^T}.
\end{align*}
Then we remark that 
\[\forall t\leq T, \quad \frac{\alpha}{2}\int_t^T\|X^1_s-X^2_s\|^2ds\leq 2L\int_t^T \|W^1_s-W^2_s\|^2ds+C\left(\omega(\delta_\mu)\right)^2.\]
This follows from the weak-strong monotonicity of $(F,G)$ from Hypothesis \ref{hyp: weak strong monotonicity in G}, the regularity of those coefficients given by Hypothesis \ref{hyp: regularity: weak strong monotonicity in G} and an application of Young's inequality.
Applying Ito's lemma to $t\mapsto |W^1_t-W^2_t|^2$ yields
\[\forall t\leq T \quad  \|W^1_t-W^2_t\|^2\leq C\int_t^T \|W^1_s-W^2_s\|^2ds+C(\omega(\delta_\mu))^2+CC_0^T,\]
by the regularity of $G$ implied by Hypothesis \ref{hyp: regularity: weak strong monotonicity in G}.
Applying Gronwall's lemma we get that there exists a constant $C$ such that
\[\forall t\leq T \quad \|W^1_t-W^2_t\|\leq C\left((\omega(\delta_\mu))^2+C_0^T\right),\]
Integrating from $0$ to $T$ and using once again Young's inequality to make $\int_0^T\|W^1_s-W^2_s\|^2ds$ on both side, we deduce that 
\[\exists C, \forall t\leq T \quad \|W^1_t-W^2_t\|^2\leq C\left((\omega(\delta_\mu))^2+\|X^1_T-X^2_T\|\omega(\delta_\mu)+\omega(\delta_\mu)\int_0^T \|X^1_s-X^2_s\|ds\right).\]
We now turn to an estimate on $t\mapsto \|X^1_t-X^2_t\|$, which follows from the argument we used in Lemma \eqref{estimate X by W G strong weak monotone}. Dealing with terms depending on $(X^1_t-X^2_t)_{t>0}$ as we did for $W$, we finally conclude to 
\[\forall t\leq T \quad \|X^1_t-X^2_t\|^2\leq C\left((\omega(\delta_\mu))^2+\left(\omega(\delta_\mu)\right)^\frac{2\gamma}{2-\gamma}\right).\]
\end{proof}

\begin{thm}
\label{existence: weak strong G}
    Let $(F,G,W_0)$ be such that Hypotheses \ref{hyp: linear growth in F},\ref{hyp: weak strong monotonicity in G} and \ref{hyp: regularity: weak strong monotonicity in G} are satisfied, then there exists a unique continuous monotone solution to \eqref{eq: ME without noise}
\end{thm}
\begin{proof}
\textit{Step 1: regularization and existence of Lipschitz solutions}

Let us fix $\eta>0$ and consider a regularization $(F^\eta,G^\eta)$ of $(F,G)$,defined by 
\[\forall \mu\in \mathcal{P}_2(\reels^d), F^\eta(x,w,\mu)=\int_{\reels^{2d}} F(x-y,w-v,\mu\star \psi_\eta)\psi_\eta(y)\psi_\eta(v)dydv,\]
with a analogous definition for $G^\eta$, and where $\psi_\eta:x\mapsto \frac{1}{\eta}\psi(\frac{x}{\eta})$, for $\psi$ a non-negative, 1-Lipschitz kernel of mass one. Let us remark that $(G^\eta,F^\eta)$ satisfies the same monotonicity assumptions as $(G,F)$, converges locally uniformly to $(G,F)$ and is $\frac{1}{\eta}-$Lipschitz in the space argument. 
We now introduce a regularization $(F^\varepsilon,G^\varepsilon,W_0^\varepsilon)$ obtained by the regularization procedure of $(F^\eta,G^\eta)$ presented in Lemma \ref{strong monotonicity yoshida}, as for $W_0^\varepsilon$, we can take it equal to $W_0$ since it is already Lipschitz. We insist that since the regularization is operated directly for the couple $(F^\eta,G^\eta)$. It follows that for fixed $\varepsilon$, the couple $(G^\varepsilon,F^\varepsilon)$ is usually a function from $(\reels^{d})\times \mathcal{P}_2(\reels^{2d})$ to $\reels^{2d}$, indeed
\[(\tilde{F}^\varepsilon,\tilde{G}^\varepsilon):(X,U)\mapsto (F^\varepsilon,G^\varepsilon)(X,U,\mathcal{L}(X,U)).\]
For MFG, this means that the regularized problem is a mean field game of control even when the original problem is not.

We now show that there exist constants $(\alpha',L',a_0')$ such that  Hypotheses \ref{hyp: weak strong monotonicity in G} and \ref{hyp: regularity: weak strong monotonicity in G} are satisfied for given $\eta,\varepsilon$. For $W_0^\varepsilon$ this is very easy to see by coming back to its definition in function of $W_0$, in fact for any $\varepsilon>0$ the following holds 
\[\forall (X,Y)\in \mathcal{H}^2 \quad \langle \tilde{W}_0^\varepsilon(X)-\tilde{W}_0^\varepsilon(Y),X-Y\rangle \geq (a_0+\varepsilon)\|\tilde{W}_0^\varepsilon(X)-\tilde{W}_0^\varepsilon(Y)\|^2.\]
Since $W_0$ is Lipschitz in $x$, it is also easy to see that $W_0^\varepsilon$ is $\gamma-$Holder in the measure argument with a constant independent of both $\varepsilon$ and $\eta$.
Letting $E=(G,F)$, We now define
\[H_\varepsilon: (X,W)\mapsto (X,W)+\varepsilon \tilde{E}(X,W),\]
and recall that the regularization $E^\varepsilon=(G^\varepsilon,F^\varepsilon)$ is given by 
\[\tilde{E}^\varepsilon(X,W)=\tilde{E}(\tilde{H}_\varepsilon^{-1}(X,W)).\]
In what follow we use the notation $\tilde{h}_1^{-1}(X,W)=\left((\tilde{H}_\varepsilon^{-1}(X,W))_i\right)_{1\leq i\leq d}$ for the first $d$ coordinate of $\tilde{H}_\varepsilon^{-1}$ and $\tilde{h}^{-1}_2$ for the last $d$ coordinates. 
Fixing $(X,Y,W,V)\in \left(\mathcal{H}\right)^4$, we observe that 
\[\langle \tilde{E}^\varepsilon(X,W)-\tilde{E}^\varepsilon(Y,V),(X,W)-(Y,V)\rangle\geq\varepsilon \|\tilde{E}^\varepsilon(X,W)-\tilde{E}^\varepsilon(Y,V)\|^2+\alpha\|\tilde{h}_1^{-1}(X,W)-\tilde{h}_1^{-1}(Y,V)\|^2-L \|\tilde{h}^{-1}_2(X,W)-\tilde{h}^{-1}_2(Y,V)\|^2.\]
we now use the fact that for any $a>0$ the following holds 
\[ \|\tilde{h}_1^{-1}(X,W)-\tilde{h}_1^{-1}(Y,V)\|^2\geq -\frac{a+4}{a+a^2}\|\tilde{h}_1^{-1}(X,W)-\tilde{h}_1^{-1}(Y,V)-(X-Y)\|^2+\frac{1}{1+a}\|X-Y\|^2.\]
By definition of $\tilde{h}_1^{-1}$
\[X=\tilde{h}_1^{-1}(X,W)+\varepsilon \tilde{G}^\varepsilon(X,W),\]
since an analogous calculation can be performed to estimate $\|\tilde{h}_2^{-1}(X,W)-\tilde{h}_2^{-1}(Y,V)\|^2$ in function of $\|\tilde{F}^\varepsilon(X,W)-\tilde{F}^\varepsilon(Y,V)\|$ and $\|W-V\|^2$, we deduce by choosing $a=\sqrt{\varepsilon}$, that for $\varepsilon$ sufficiently small 
\begin{equation}
\label{conservation of weak strong monotonicity}
\langle \tilde{E}^\varepsilon(X,W)-\tilde{E}^\varepsilon(Y,V),(X,W)-(Y,V)\rangle\geq \frac{\alpha}{1+\sqrt{\varepsilon}}\|X-Y\|^2-(1+\sqrt{\varepsilon})L \|W-V\|^2.\end{equation}
Since $(\tilde{G}^\varepsilon,\tilde{F}^\varepsilon)$ are $\frac{1}{\varepsilon}-$Lipschitz in $\mathcal{H}^2$, for a given $\varepsilon$, $(F,G,W_0)^\varepsilon$ satisfy the assumption of Lemma \ref{existence in the Lipschitz framework}, as a consequence there exists a corresponding sequence of Lipschitz solution $W^\varepsilon$ well-defined on $[0,T]$.

\textit{Step 2: getting estimates on the sequence of solution}
Let us first remark that by Corollary \ref{corol: linear growth regul}, and Hypothesis \ref{hyp: linear growth in F}, for $\varepsilon$ sufficiently small $(F^\varepsilon,G^\varepsilon)$ have linear growth with constants independent of $\varepsilon$. We now turn to the main technical difficulty of this proof, namely the fact that the Hille-Yosida regularization procedure of $(\tilde{G},\tilde{F})$ in $\mathcal{H}^2$ does not preserve semi-monotonicity. However we can show that on bounded set of $\mathcal{H}$, there exists a constant $C$ independent of $\eta,\varepsilon$ such that
\begin{gather}
\nonumber\langle \tilde{F}^\varepsilon(X,W)-\tilde{F}^\varepsilon(Y,V),X-Y\rangle \geq -C\left(\|X-Y\|^2+\|W-V\|^{2\gamma}+\varepsilon(1+\|X\|^2+\|Y\|^2+\|W\|^2+\|V\|^2)\right),\\
\label{almost semi concavity estimates}\langle \tilde{G}^\varepsilon(X,W)-\tilde{G}^\varepsilon(Y,V),X-Y\rangle \leq C\left(\|X-Y\|^2+\|W-V\|^2+\varepsilon(1+\|X\|^2+\|Y\|^2+\|W\|^2+\|V\|^2)\right).
\end{gather}
Let us show this form $\tilde{F}^\varepsilon$ only, $\tilde{G}^\varepsilon$ being treated in a similar fashion. 
\begin{align*}
    \langle \tilde{F}^\varepsilon(X,W)-\tilde{F}^\varepsilon(Y,V),X-Y\rangle&=\langle \tilde{F}(H^{-1}_\varepsilon (X,W))-\tilde{F}(H^{-1}_\varepsilon (Y,V)),\tilde{h}^{-1}_1 (X,W)-\tilde{h}^{-1}_1 (Y,V)\rangle\\
    &+\varepsilon\langle \tilde{F}^\varepsilon(X,W)-\tilde{F}^\varepsilon(Y,V),\tilde{G}^\varepsilon(X,W)-\tilde{G}^\varepsilon(Y,V)\rangle
\end{align*}
 Using the semi-monotonicity of $\tilde{F}$ and its holder continuity in the second argument under Hypothesis \ref{hyp: regularity: weak strong monotonicity in G} yields 
\begin{align*}
    \langle \tilde{F}^\varepsilon(X,W)-\tilde{F}^\varepsilon(Y,V),X-Y\rangle&\geq -C\left(\|\tilde{h}^{-1}_1 (X,W)-\tilde{h}^{-1}_1 (Y,V)\|^2+\|\tilde{h}^{-1}_2 (X,W)-\tilde{h}^{-1}_2 (Y,V)\|^2\right)\\
    &+\varepsilon\langle \tilde{F}^\varepsilon(X,W)-\tilde{F}^\varepsilon(Y,V),\tilde{G}^\varepsilon(X,W)-\tilde{G}^\varepsilon(Y,V)\rangle
\end{align*}
for a constant $C$ depending on $F$ only. Since $H^{-1}_\varepsilon$ is $1-$Lipschitz in $\mathcal{H}^2$ using the linear growth of $(F^\varepsilon,G^\varepsilon)$ to take care of the term in $\varepsilon$, we obtain the announced estimate. At this point, let us insist on the fact that there exists a  $\varepsilon_0$ such that Hypothesis \ref{hyp: linear growth in F} and Hypothesis \ref{hyp: weak strong monotonicity in G} are satisfied uniformly with constants independent of $\varepsilon$ for the family $(F^\varepsilon,G^\varepsilon)_{\varepsilon\leq \varepsilon_0}$. In particular by Lemma \ref{montone fbsde have bounded growth}, for any such $\varepsilon$, there exists a constant $C$ such that for any $(m^\mu_t)_{s\in[0,t]}$ solution to 
\[\partial_t m^\mu_s=-\text{div}\left(F^\varepsilon(x,m^\mu_s,W^\varepsilon(t-s,x,m^\mu_s))m^\mu_s\right) \quad m^\mu|_{s=0}=\mu,\]
we have
\[\forall s\leq t \quad \int_{\reels^d}|y|^2m^\mu_s(dy)\leq C\left(1+\int_{\reels^d}|y|^2\mu(dy)\right).\]
As a consequence, following the proof of Lemma \ref{estimate W strong in G}, \ref{estimate x strong in G}, by treating the terms depending on $\varepsilon$ as a perturbation in \eqref{almost semi concavity estimates}, we can find a constant independent of $\varepsilon,\eta$ such that for$\mu,\nu\in\mathcal{P}_2(\reels^d)$, the following holds for $\varepsilon\leq \varepsilon_0 $ and $s\leq t$
\begin{equation}
\label{almost continuity estimate measure}
\mathcal{W}_2\bigg(W^\varepsilon(t,\cdot,m^\mu_s)_\#m^\mu_s,W^\varepsilon(t,\cdot,m^\nu_s)_\#m^\nu_s\bigg) ,\mathcal{W}_2(m^\mu_s,m^\nu_t)\leq C\left(\mathcal{W}^\gamma_2(\mu,\nu)+\varepsilon^\gamma\!\left(1+\!\int_{\reels^d} |x|^2\mu(dx)+\!\int_{\reels^d} \!|x|^2\nu(dx)\right)\right), 
\end{equation}
for a constant independent of $\varepsilon$ and where 
\[\mathcal{W}^\gamma_2(\mu,\nu)=\max\left(\mathcal{W}_2(\mu,\nu),\left(\mathcal{W}_2(\mu,\nu)\right)^\gamma\right).\]
Following the regularization operated in the space variable, for a given $\pi\in\mathcal{P}_2(\reels^{2d})$ $(x,p)\mapsto F^\varepsilon(x,\pi,p)$ is at least $\frac{1}{\eta}-$Lipschitz uniformly in $\pi\in\mathcal{P}_2(\reels^{2d})$. The key idea is now to remark either by a proof similar to Lemma \ref{lemma: lip sol FBSDE} or directly by taking the conditional expectation that there exists a unique solution to the forward backward system 
\begin{equation}
\label{eq fbsde in x}
\left\{
\begin{array}{l}
     \displaystyle X_s=x-\int_0^s f(s,X_s,U_s)ds+\sigma_xB_s,  \\
     \displaystyle U_s=W_0(X_t,m^\mu_t)+\int^t_s g(s,X_s,U_s)ds-\int_s^t Z_sdB_s,
\end{array}
\right.
\end{equation}
with $U_s=W(t-s,X_s,m^\mu_s)$. Where $f(u,x,p)=F^\varepsilon(x,p,\pi_u^\mu),g(u,x,p)=G(x,p,\pi^\mu_u)$ with $\pi^\mu_u=\mathcal{L}(X,W(t,X,\mathcal{L}(X))$ for any $X\in\mathcal{H}$ with law $m^\mu_u$. First, using the assumptions on $(F,G)$ we deduce that for a given $(\pi,x,y)$ 
\begin{gather}
\nonumber (F^\varepsilon(x,w,\pi)-F^\varepsilon(y,v,\pi))\cdot(x-y)\geq -C\left(1+\frac{\varepsilon}{\eta^2}\right)\left(|x-y|^2+|w-v|^{2\gamma}\right),\\
\label{semi monotonicity in space only}(G^\varepsilon(x,w,\pi)-G^\varepsilon(y,v,\pi))\cdot(w-v)\leq C\left(1+\frac{\varepsilon}{\eta^2}\right)\left(|x-y|^2+|w-v|^2\right).
\end{gather}
Using the estimates \eqref{semi monotonicity in space only} and the inequality \eqref{conservation of weak strong monotonicity}, we deduce using Lemma \ref{estimate W strong in G} that for $\varepsilon<\eta^2$ there exists a constant $C$ independent of both $\varepsilon$ and $\eta$ such that 
\begin{equation}
\label{continuity in x of W eps}
\forall (t,x,y,\mu)\in [0,T]\times (\reels^d)^2\times \mathcal{P}_2(\reels^d), \quad |W^\varepsilon(t,x,\mu)-W^\varepsilon(t,y,\mu)|\leq C|x-y|.
\end{equation}
We now turn to an estimate with respect to the measure argument. Since $(F^\varepsilon,G^\varepsilon)$ is continuous we can apply Lemma \ref{prop: expectation to pointwise} to deduce that for a given $(\pi_1,\pi_2)\in \left(\mathcal{P}_2(\reels^{2d})\right)^2$, the following holds 
\[
\begin{array}{l}
 (F^\varepsilon(x,p,\pi_1)-F^\varepsilon(y,q,\pi_2))\cdot (p-q)+(G^\varepsilon(x,p,\pi_1)-G^\varepsilon(y,q,\pi_2))\cdot (x-y)\geq -(|x-y|+|p-q|)\omega(\mathcal{W}_2(\pi_1,\pi_2)),\\
     (F^\varepsilon(x,p,\pi_1)-F^\varepsilon(y,q,\pi_2))\cdot (p-q)+(G^\varepsilon(x,p,\pi_1)-G^\varepsilon(y,q,\pi_2))\cdot (x-y)\geq \frac{\alpha}{2} |x-y|^2-2L |p-q|^2-C\left(\omega(\mathcal{W}_2(\pi_1,\pi_2))\right)^2,  \\
\end{array}\]
for all $(x,y,p,q)$ in $\reels^{4d}$. By Lemma \ref{lemma: estim regul measure}
, and using the estimate \eqref{almost continuity estimate measure} we deduce that
\[|W^\varepsilon(t,x,\mu)-W^\varepsilon(t,x,\nu)|\leq \omega_\gamma\left(\left(\mathcal{W}^\gamma_2(\mu,\nu)+\varepsilon^\gamma (\mathcal{W}_2(\mu,\delta_{0_{\reels^d}})+\varepsilon^\gamma(\mathcal{W}_2(\nu,\delta_{0_{\reels^d}})\right)\right),\]
for a modulus of continuity depending on $\gamma,T,\alpha,L,a_0$.
Finally, an estimate on $t\mapsto W(t,x,\mu)$ is obtained thanks to Lemma \ref{montone fbsde have bounded growth} and the previous estimates on $(x,\mu)\mapsto W^\varepsilon(t,x,\mu)$. Let $(x,\mu)$ belongs to a bounded set of $\reels^d\times \mathcal{P}_2(\reels^d)$,
\[|W^\varepsilon(t,x,\mu)-W^\varepsilon(s,x,\mu)|\leq |W^\varepsilon(s,X^\mu_{t-s},m^\mu_{t-s})-W^\varepsilon(s,x,\mu)|+C(t-s)(1+|x|+\mathcal{W}_2(\mu,\delta_{0_{\reels^d}})),\]
for a constant depending on the growth of the coefficients $(F,G,W_0)$ only. Owning to previous continuity estimates on $W^\varepsilon$ in $\reels^d\times \mathcal{P}_2(\reels^d)$, as well as an estimate on the continuity with respect to time of $(X^\mu_u,m^\mu_u)_{u\in[0,t-s]}$ stemming from Lemma \eqref{montone fbsde have bounded growth}, we deduce that on bounded sets $O$ of $\reels^d\times \mathcal{P}_2(\reels^d)$ there exists a constant $C_O$ such that $\forall (x,\mu)\in O$, 
\begin{equation}
\label{almost time continuity}
\forall t\leq T,\quad |W^\varepsilon(t,x,\mu)-W^\varepsilon(s,x,\mu)|\leq C_O(|t-s|+\omega_\gamma(C_0|t-s|^\gamma+\varepsilon^\gamma \mathcal{W}_2(\mu,\delta_{0_{\reels^d}})).
\end{equation}

\textit{Step 3: local uniform convergence along a subsequence}
Let us now fix some $(t,\mu)\in \mathcal{P}_2(\reels^d)$,Defining  $W^{\varepsilon,t,\mu}:x\mapsto W^\varepsilon(t,x,\mu)$, we know thanks to the estimate \eqref{continuity in x of W eps} that along a subsequence $(W^{\varepsilon_n,\mu})_{n\geq 0}$,  $(W^{\varepsilon,\mu})_{\varepsilon>0}$ converges locally uniformly to a function $W^{\eta,t,\mu}$. Now by the separability of $[0,T]\times\mathcal{P}_2(\reels^d)$ and a classic diagonalization argument, we can extract a subsequence of $(W^\varepsilon)_{\varepsilon>0}$ such that $(t,x,\mu)\mapsto W^\varepsilon(t,x,\mu)$ converges locally uniformly to some function $W^{\eta}$ on a dense subset of $[0,T]\times \reels^d\times \mathcal{P}_2(\reels^d)$ denoted $D$.   We now show that $(t,x,\mu)\mapsto W^{\eta}(t,x,\mu)$ is continuous. This follows naturally from \eqref{almost continuity estimate measure},\eqref{almost time continuity}, by letting $\varepsilon$ tends to 0 along this subsequence we get that on bounded sets $\mathcal{O}$ of $\reels^d\times \mathcal{P}_2(\reels^d)$
\[|W^{\eta}(s,x,\mu)-W^{\eta}(t,x,\nu)|\leq C_\mathcal{O}(|t-s|+\omega_\gamma(\mathcal{W}^\gamma_2(\mu,\nu)+|t-s|^\gamma)).\]
Since $(t,x,\mu)\mapsto W^{\eta,\mu}(t,x)$ is continuous on a dense subset of $[0,T]\times \reels^d\times \mathcal{P}_2(\reels^d)$, we can extend it into a continuous function on this whole space we call $W^\eta$ toward which the sequence. Since $(F^\varepsilon,G^\varepsilon,W^\varepsilon)$ converges locally uniformly to $(G^\eta,F^\eta,W^\eta)$ along a subsequence, by Theorem \ref{stability without noise}, $W^\eta$ is a monotone solution to \eqref{eq: ME without noise} with the data $(W_0,G^\eta,F^\eta)$. Since we already know that $(F^\eta,G^\eta)$ converges locally uniformly to $(F,G)$, and that the sequence $(W^\eta)_{\eta>0}$ is locally uniformly bounded and equicontinuous, we finally conclude to the existence of a monotone solution by the stability of monotone solutions and convergence along a subsequence thanks to Arzela-Ascoli. 
\end{proof}
\begin{remarque}
Let us remark that since we have estimates on the growth of solutions by Lemma \ref{montone fbsde have bounded growth}, we only need the regularity assumption Hypothesis \ref{hyp: regularity: strong monotonicity in G} to hold locally. Which is to say we could require instead that for any $R>0$, there exists a constant $C_R$ and a modulus of continuity $\omega_R(\cdot)$ such that 
\begin{enumerate}
    \item[-] $\forall X,Y,U,V\in B_R=\left\{Z\in \mathcal{H},\quad \|Z\|\leq R\right\}$
    \begin{enumerate}
        \item[-] $\langle \tilde{G}(X,U)-\tilde{G}(Y,V),U-V\rangle \leq C_R\left(\|X-Y\|^2+\|U-V\|^2\right)$,
        \item[-] $\langle \tilde{F}(X,U)-\tilde{F}(Y,V),X-Y\rangle \geq -C_R\left(\|X-Y\|^2+\|U-V\|^{2}_\gamma\right)$,
         \end{enumerate}
    \item[-]  $\displaystyle \forall X,U\in B_R, \quad \mu,\nu\in B_R^\mathcal{P}=\left\{m\in \mathcal{P}_2(\reels^d),\quad \int_{\reels^d} |y|^2m(dy)\leq R\right\}$
     \[|F(X,U,\mu)-F(X,U,\nu)|+|G(X,U,\mu)-G(X,U,\nu)|+|W_0(X,\mu)-W_0(X,\nu)|\leq \omega_R(\mathcal{W}_2(\mu,\nu)).\]
    \end{enumerate}
\end{remarque}

\subsubsection{Presence of common noise}
We now explain how this existence result extend to the presence of common noise without much difficulty. We start by giving estimates on the following system
\begin{equation}
\label{fbsde with common noise process}
\left\{
\begin{array}{l}
     \displaystyle X_s=X-\int_0^sF(X_u,\theta_u,\mathcal{L}(X_u|\mathcal{F}^\theta_u),W_u)du+\sqrt{2\sigma_x}B_s,\\
     \displaystyle \theta_s=\theta-\int_0^s b(\theta_u)du+\sigma_\theta B^\theta_s,\\
     \displaystyle W_s=W_0(X_T,\theta_T,\mathcal{L}(X_T|\mathcal{F}^\theta_T))+\int_s^T G(X_u,\theta_u,\mathcal{L}(X_u|\mathcal{F}^\theta_u),W_u)du-\int_s^T Z_s\cdot d(B_s,B^\theta_s)\\
     \mathcal{F}^\theta_s=\sigma((B^\theta_u)_{u\leq s}).
\end{array}
\right.
\end{equation}
To that end we make the following assumptions
\begin{hyp}
\label{linear growth with theta}
There exists a constant $C$ such that $\forall (x,\theta,\mu,u)\in \reels^d\times \reels^n\times \mathcal{P}_2(\reels^d)\times \reels^d$,
\[|(F,G)(x,\theta,\mu,u)|+|W_0(x,\theta,\mu)|\leq C\left(1+|x|+|u|+|\theta|+\sqrt{\int_{\reels^d} |y|^2\mu(dy)}\right).\]
\end{hyp}
\begin{hyp}
\label{weak strong common noise}
    The functions $(x,u,\mu)\mapsto (F,G,W_0)(x,\theta,\mu,u)$ satisfy Hypotheses \ref{hyp: weak strong monotonicity in G} and \ref{hyp: regularity: weak strong monotonicity in G} uniformly in $\theta\in \reels^n$. Letting $v=(F,G,W_0)$, there exists a modulus of continuity $\omega(\cdot)$ such that $\forall (x,\theta_1,\theta_2,\mu,u)\in \reels^d\times \left(\reels^n\right)^2\times \mathcal{P}_2(\reels^d)\times \reels^d$,
    \[|v(x,\theta_1,\mu,u)-v(x,\theta_2,\mu,u)|\leq \omega(|\theta_1-\theta_2|).\]
\end{hyp}

\begin{corol}
\label{estimate with common noise}
Under Hypotheses \ref{linear growth with theta} and \ref{weak strong common noise}, there exists a constant $C$ such that if there exists a  strong solution in $\mathcal{H}^4$ $(X,\theta^x,W^x,Z^x)_{t\in[0,T]}$ to \eqref{fbsde with common noise process}(resp. $(Y,\theta^y,W^y,Z^y)_{t\in[0,T]}$) with initial datum $(X_0,\theta^1)\in \mathcal{O}$ (resp. $(Y_0,\theta^2)$), then $\forall t\leq T$
\begin{align*}
\|W_t^x-W_t^y\|\leq C\left(\|X_0-Y_0\|+\left(\omega(|\theta^1-\theta^2|)\right)^\frac{\gamma}{2-\gamma}\right),\\
\|X_t-Y_t\|\leq C\left(\|X_0-Y_0\|_\gamma+\left(\omega(|\theta^1-\theta^2|)\right)^\frac{\gamma}{2-\gamma}\right)
\end{align*}
\end{corol}
\begin{proof}
A classic Gronwall estimates yields the existence of a constant $C_\theta$ depending only on $b,\Sigma,T$ such that 
\[\forall t\leq T \quad |\theta^x_t-\theta^y_t|\leq C_\theta|\theta^1-\theta^2| \quad a.s. \]
From there on, the proof is a direct adaptation of Lemma \ref{lemma: estim regul measure},
the dependency on $\theta$ being treated as a perturbation.
\end{proof}
\begin{thm}
 Let $(F,G,W_0)$ be such that Hypotheses \ref{hyp: weak strong monotonicity in G},\ref{hyp: regularity: weak strong monotonicity in G} and \ref{hyp: linear growth in F} are satisfied, then there exists a unique continuous monotone solution to \eqref{eq: ME common noise change variable}.
\end{thm}
\begin{proof}
We begin with the same regularization $(F^\varepsilon,G^\varepsilon)$ of $(F,G)$ as in Theorem \ref{existence: weak strong G}. Let us now observe that defining 
\[F^{\theta,\varepsilon}:(x,\mu,\theta)\mapsto F^\varepsilon(x+\theta,(Id_{\reels^d}+\theta)_\#\mu),\]
then $X\mapsto \tilde{F}^{\theta,\varepsilon}(X,\theta)$ has same monotonicity and regularity as $\tilde{F}^\varepsilon$. Furthermore, $\theta\mapsto F^{\theta,\varepsilon}(X,\theta)$ inherits the regularity of $X\mapsto \tilde{F}^\varepsilon(X)$. In particular, since $\tilde{F}^\varepsilon$ is Lipschitz on $\mathcal{H}$, $\tilde{F}^{\theta,\varepsilon}$ is Lipschitz on $\mathcal{H}\times \reels^n$ with the same constant. This means that there exists a Lipschitz solution to \eqref{eq: ME common noise changed variable}, with the data $(F^{\theta,\varepsilon},G^{\theta,\varepsilon},W_0^\theta)$ for each $\theta>0$ by Lemma \ref{existence in the Lipschitz framework}, for $\varepsilon>0$, there exists a Lipschitz solution $V^\varepsilon$ to the problem on some time interval $[0,T_\varepsilon)$. We now show that we can take $T_\varepsilon=T$, this is natural consequence of Lemma \ref{estimate with common noise}, since the coefficients are globally Lipschitz, we can take the constants in this Lemma to be global and $\gamma=1$, yielding an estimate on $\|\tilde{V}^\varepsilon(t,\cdot)\|_{Lip(\mathcal{H}\times \reels^n)}$. By Lemma \ref{existence in the Lipschitz framework}, if $T_\varepsilon<0$, then we would get a contradiction. By Lemma \ref{lip sol chang variable}, there exists a sequence of function $(W^\varepsilon)_{\varepsilon>0}$, such that 
\[\forall \varepsilon>0, V^\varepsilon:(t,x,\theta,\mu)\mapsto W^\varepsilon(t,x+\theta,(Id_{\reels^d}+\theta)_\#\mu).\]
To show that there exists a monotone solution to \eqref{eq: ME common noise change variable}, it is sufficient to show that the sequence $(V^\varepsilon)_{\varepsilon>0}$ converges to a monotone solution of \eqref{eq: ME common noise changed variable}. Indeed in this case $(W^\varepsilon)_{\varepsilon>0}$ also converges to a function $W$ and by definition $W$ will be a monotone solution to \eqref{eq: ME common noise change variable}. This is done by getting estimates on the sequence $(V^\varepsilon)_{\varepsilon>0}$. First let us observe that previously presented estimates in the proof of Theorem \ref{existence: weak strong G} are still valid for this equation to estimate $(x,\mu)\mapsto V^\varepsilon(t,x,\theta,\mu)$. The continuity of the sequence with respect to the new variable $\theta$, is controlled by Lemma \ref{estimate with common noise}, an estimate on the time continuity of solutions following form estimates with respect to the others argument and classic estimates on SDE. The rest of the proof is carried out as in Theorem \ref{existence: weak strong G}, using the stability of monotone solutions Theorem \ref{stability with common noise} to conclude. 
\end{proof}  
 \subsection{Existence under different monotonicity assumptions}
It is well known that formally, regularity can be traded for monotonicity for master equations. That is to say under stronger monotonicity assumptions, we can ask for lower regularity of the coefficients. We now demonstrate such property is also true for monotone solutions.
\subsubsection{Strong monotonicity in $X$}
\begin{hyp}
\label{hyp: strong monotonicity in G}
    \begin{gather}\nonumber\exists \alpha>0,a_0>0, \quad \forall (t,X,Y,U,V)\in [0,T]\times \left(\mathcal{H}\right)^4\\
       \nonumber \langle \tilde{W}_0(X)-\tilde{W}_0(Y),X-Y\rangle \geq a_0 \|\tilde{W}_0(X)-\tilde{W}_0(Y)\|^2,\\
    \langle \tilde{G}(X,U)-\tilde{G}(Y,V),X-Y\rangle+\langle \tilde{F}(X,U)-\tilde{F}(Y,V),U-V\rangle \geq \alpha \|X-Y\|^2
    \end{gather}
\end{hyp}
\begin{hyp}
\label{hyp: regularity: strong monotonicity in G}
there exists  $\gamma \in (0,1]$ such that
    \begin{enumerate}
    \item[-] There exists a constant C such that $\forall (X,Y,U,V)\in \mathcal{H}^4$
    \begin{enumerate}
        \item[-] $\langle \tilde{G}(X,U)-\tilde{G}(Y,V),U-V\rangle \leq C\left(\|X-Y\|^{2}_\gamma+\|U-V\|^2\right)$,
        \item[-]  $\langle \tilde{F}(X,U)-\tilde{F}(Y,V),U-V\rangle \geq- C\left(\|X-Y\|^{2}_\gamma+\|U-V\|^2_\gamma\right)$
         \end{enumerate}
    \item[-] There exists a modulus $\omega(\cdot)$ such that $\forall (x,y,\mu,\nu)\in (\reels^d)^2\times \left(\mathcal{P}_2(\reels^d)\right)^2$
     \[|F(x,u,\mu)-F(x,u,\nu)|+|G(x,u,\mu)-G(x,u,\nu)|+|W_0(x,\mu)-W_0(x,\nu)|\leq \omega(\mathcal{W}_2(\mu,\nu))\]
    \end{enumerate}
\end{hyp}

\begin{lemma}
\label{estimate very strong in G}
Under Hypotheses \ref{hyp: strong monotonicity in G},\ref{hyp: regularity: strong monotonicity in G} and \ref{hyp: linear growth in F}, there exists a constant $C$ such that if there exists a  strong solution $(X,W^x,Z^x)_{t\in[0,T]}$ in $\mathcal{H}^3$ to \eqref{eq: FBSDE} (resp. $(Y,W^y,Z^y)_{t\in[0,T]}$) with initial datum $X_0\in \mathcal{O}$ (resp. $Y_0$),
\begin{gather*}
    \sup_{t\in[0,T]}\|W_t^x-W_t^y\|\leq C\|X_0-Y_0\|_\frac{\gamma}{2-\gamma},\\
    \sup_{t\in[0,T]}\|X_t-Y_t\|\leq C \|X_0-Y_0\|_\frac{\gamma^2}{2-\gamma}.
\end{gather*}
\end{lemma}
\begin{proof}
We begin as in Lemma \ref{estimate W strong in G} by studying 
\[U_t=(W^x_t-W^y_t)\cdot (X_t-Y_t).\]
Following the same computations, this yields 
\begin{equation}
\label{estimate from monotonicity strong in G}
\forall t\leq T \quad a_0\|W_T^x-W_T^y\|^2+\alpha \int_t^T \|X_t-Y_t\|^2\leq \esp{U_t}\leq \esp{U_0}.
\end{equation}
Applying ito's lemma to $t\mapsto |W_t^x-W_t^y|^2$ and using Hypothesis \ref{hyp: regularity: strong monotonicity in G}, we get that there exists a constant $C>0$ such that 
\[ \forall t\leq T \quad  \|W_t^x-W_t^y\|^2\leq C \left(\esp{U_0}+\esp{U_0}^{\gamma}+\int_t^T \|W_s^x-W_s^y\|^2ds\right). \]
An application of Gronwall's lemma combined with an evaluation at $t=0$ to estimate $\|W_0^x-W_0^y\|$ in function of $\|X_0-Y_0\|$ finally gives 
\[\exists C>0 \quad \forall t\leq T \quad \|W_t^x-W_t^y\|^2 \leq Ce^{C(T-t)}\left(\|X_0-Y_0\|^2+\|X_0-Y_0\|^2_{\frac{\gamma}{2-\gamma}}\right). \]
To estimate $\|X_t-Y_t\|$, we proceed in a similar fashion, using the above result to obtain the existence of a constant $C>0$ such that
\[\|X_t-Y_t\|^2\leq Ce^{Ct}\left(\|X_0-Y_0\|^2+\|X_0-Y_0\|^{2}_\gamma+ +\|X_0-Y_0\|^2_{\frac{\gamma^2}{2-\gamma}}\right)\]
\end{proof}
\begin{remarque}
    In fact, to get an estimate on $\|X_t-Y_t\|$, it is sufficient to assume that $\tilde{F}$ is uniformly (or locally uniformly on bounded sets of $\mathcal{H}$) continuous and that there exists a concave modulus of continuity $\omega(\cdot)$ such that 
    \begin{gather*}
    \forall (X,Y,U)\in \mathcal{H}^3,\quad  \|\tilde{F}(X,U)-\tilde{F}(Y,U)\|^2\leq \omega\left(\|X-Y\|^2\right) .
    \end{gather*}
    Indeed the concavity of the modulus allows the following computation 
    \[\int_0^t \omega\left(\|X_s-Y_s\|^2\right) ds\leq \omega \left(\int_0^t \|X_s-Y_s\|^2 ds\right).\]
    Reusing \eqref{estimate from monotonicity strong in G}, then yields a control of $t\mapsto \|X_t-Y_t\|$ uniformly in $\|X_0-Y_0\|$ for $t\in[0,T]$.
\end{remarque}
A direct consequence of this result, is that whenever $(X,W)\mapsto (\tilde{G},\tilde{F})(X,W)$ is strongly monotone in $X$, we can show the existence of a monotone solution in weaker class of regularity 
\begin{thm}
\label{thm: existence strong G}
 Let $(F,G,W_0)$ be such that Hypotheses \ref{hyp: strong monotonicity in G}, \ref{hyp: regularity: strong monotonicity in G} and \ref{hyp: linear growth in F} are satisfied, then there exists a unique continuous monotone solution to \eqref{eq: ME without noise}
\end{thm}
\subsubsection{Weak-strong monotonicity in $W$}
We now show that similar estimates are possible whenever the "strong" part of the monotonicity of $(X,U)\mapsto (\tilde{G},\tilde{F})(X,U)$ is on $\tilde{F}$ in $U$. In fact in this case we need less regularity on the initial condition, this was already observed in \cite{Lions-college}.
\begin{hyp}
\label{hyp: weak strong monotonicity in F}
    \begin{gather}\nonumber\exists \alpha>0,L\geq 0, \quad \forall (X,Y,U,V)\in \mathcal{H}^4\\
    \nonumber \langle \tilde{W}_0(X)-\tilde{W}_0(Y),X-Y\rangle \geq 0\\
    \langle \tilde{G}(X,U)-\tilde{G}(Y,V),X-Y\rangle+\langle \tilde{F}(X,U)-\tilde{F}(Y,V),U-V\rangle \geq \alpha \|U-V\|^2-L\|X-Y\|^2
    \end{gather}
\end{hyp}
we now state our regularity assumptions on the coefficients 
\begin{hyp}
\label{hyp: regularity: weak strong monotonicity in F}
There exists  $\gamma \in (0,1]$ such that
    \begin{enumerate}
    \item[-] There exists a constant $C$ such that $\forall (X,Y,U,V)\in \mathcal{H}^4$
    \begin{enumerate}
        \item[-] $\langle \tilde{G}(X,U)-\tilde{G}(Y,V),U-V\rangle \leq C\left(\|X-Y\|^{2}_\gamma+\|U-V\|^2\right)$,
        \item[-]  $\langle \tilde{F}(X,U)-\tilde{F}(Y,V),X-Y\rangle \geq- C\left(\|X-Y\|^{2}+\|U-V\|^2\right)$,
        \item[-] $\|\tilde{W}_0(X)-\tilde{W}_0(Y)\|\leq C\|X-Y\|_\gamma$.
         \end{enumerate}
    \item[-] There exists a modulus $\omega(\cdot)$ such that $\forall (x,y,\mu,\nu)\in (\reels^d)^2\times \left(\mathcal{P}_2(\reels^d)\right)^2$
     \[|F(x,u,\mu)-F(x,u,\nu)|+|G(x,u,\mu)-G(x,u,\nu)|+|W_0(x,\mu)-W_0(x,\nu)|\leq \omega(\mathcal{W}_2(\mu,\nu))\]
    \end{enumerate}
\end{hyp}

\begin{lemma}
\label{weak strong estimate in F}
Under Hypotheses \ref{hyp: weak strong monotonicity in F},\ref{hyp: regularity: weak strong monotonicity in F} and \ref{hyp: linear growth in F}, there exists a constant $C$ such that if there exists a  strong solution $(X,W^x,Z^x)_{t\in[0,T]}$ in $\mathcal{H}^3$ to \eqref{eq: FBSDE} (resp. $(Y,W^y,Z^y)_{t\in[0,T]}$) with initial datum $X_0\in \mathcal{O}$ (resp. $Y_0$),
\begin{gather*}
   \forall t\leq T,\\
    \|W_t^x-W_t^y\|\leq C\|X_0-Y_0\|_\frac{\gamma}{2-\gamma},\\
    \|X_t-Y_t\|\leq C\|X_0-Y_0\|_{\frac{1}{2-\gamma}}.
\end{gather*}
\end{lemma}
\begin{proof}
We begin as in Lemma \ref{estimate W strong in G} by applying Ito's lemma to \[U_t=(W^x_t-W^y_t)\cdot (X_t-Y_t).\]
Making use of the joint monotonicity of $(F,G)$ and the monotonicity of $W_0$ gives
\[\forall t\leq T, \quad \esp{U_0}\geq \esp{U_t}\geq 0.\]
Combining this result with Hypothesis \ref{hyp: regularity: weak strong monotonicity in F} this yields 
\begin{equation}
\label{estimating W in function of X}
\alpha \int_0^t \|W^x_s-W^y_s\|^2ds \leq \esp{U_0}+L \int_0^t \|X_s-Y_s\| ds.
\end{equation}
We now turn to an estimate on $t\mapsto \|X_t-Y_t\|$. By Ito's Lemma 

\begin{align*}
\|X_t-Y_t\|^2&=\|X_0-Y_0\|^2-2\int_0^t \langle F^x_s-F^y_s,X_s-Y-s\rangle ds\\
&\leq \|X_0-Y_0\|^2+C\int_0^t \left(\|X_s-Y_s\|^2+ \|W^x_s-W^y_s\|^2\right)ds,
\end{align*}
for some constant $C$ depending on $F$ and where the second line comes from the assumed regularity of $F$. From inequality \eqref{estimating W in function of X} in combination with an application of Gronwall's lemma we deduce that there exists a constant $C>0$ such that 
\[\forall t\leq T \quad \|X_t-Y_t\|^2\leq C(\|X_0-Y_0\|^2+t  \esp{U_0})e^{Ct}.\]
Since $U_0$ depends on $W^x_0-W^y_0$, it remains to estimate this term in function of $\|X_0-Y_0\|$. For that we proceed as in Lemma \ref{estimate W strong in G}, namely by using Ito's lemma on $t\mapsto \|W^x_t-W^y_t\|^2$ between $[T,t]$ and applying a backward version of Gronwall lemma, in this case we obtain that there exists a constant $C>0$ such that 
\[\forall t\leq T \quad \|W^x_t-W^y_t\|^2\leq C\left(\|W^x_T-W^y_T\|^2+\int_0^T \|X_s-Y_s\|^{2}_\gamma ds\right)e^{C(T-t)}.\]
We now make use of the regularity of $W_0$ and the previous estimate on $t\mapsto \|X_t-Y_t\|$ to obtain that there exists another constant $C$ depending on $F,G,T,W_0$ such that 
\[\forall t\leq T \quad \|W^x_t-W^y_t\|^2\leq C\left(\|X_0-Y_0\|^{2}_\gamma+\esp{U_0}^{\gamma}\right).\]
Evaluating this expression at $t=0$ and applying young's inequality to estimate $\|W_0^x-W_0^y\|$ finally yields
\[\|W^x_0-W^y_0\|^2\leq C\left(\|X_0-Y_0\|^2_{\gamma}+\|X_0-Y_0\|^2_{\frac{\gamma}{2-\gamma}}\right),\]
for another constant $C$ now also depending on $\gamma$. The result then follow easily by plugging in this bound in previous estimates.
\end{proof}
\begin{thm}
 Let $(F,G,W_0)$ be such that Hypotheses \ref{hyp: weak strong monotonicity in F}, \ref{hyp: regularity: weak strong monotonicity in F} and \ref{hyp: linear growth in F} are satisfied, then there exists a unique continuous monotone solution to \eqref{eq: ME without noise}
\end{thm}
\begin{proof}
While the argument used are very closed to Theorem \ref{existence: weak strong G}, we now need to also regularize $W_0$. This is done by Lemma \ref{lemma: regul monotone function} as for $(F,G)$. Let $(F^\varepsilon,G^\varepsilon,W_0^\varepsilon)$ be this regularization, we then deduce that with constant depending on $\varepsilon$, \ref{hyp: regularity: weak strong monotonicity in F} is verified with $\gamma=1$, hence the existence of Lipschitz solutions $W^\varepsilon$ on $[0,T]$. It then a matter of showing that up to a factor going to $0$ as $\varepsilon$ tends to 0, the estimate of Lemma \ref{weak strong estimate in F} is satisfied as in Theorem \ref{existence: weak strong G}. For that an equivalent to the growth estimate of Lemma \ref{montone fbsde have bounded growth} is needed, which follows from similar arguments under the monotonicity condition \ref{hyp: weak strong monotonicity in F}.
\end{proof}

\subsubsection{Strong monotonicity in $W$}
There exists a corresponding result to Theorem \ref{thm: existence strong G} whenever the couple $(X,W)\mapsto (\tilde{G},\tilde{F})(X,W)$ is strongly monotone in $W$
\begin{hyp}
\label{hyp: strong monotonicity in F}
    \begin{gather}\nonumber\exists \alpha>0, \quad \forall (t,X,Y,U,V)\in [0,T]\times \left(\mathcal{H}\right)^4\\
    \nonumber \langle \tilde{W}_0(X)-\tilde{W}_0(Y),X-Y\rangle \geq 0\\
    \langle \tilde{G}(X,U)-\tilde{G}(Y,V),X-Y\rangle+\langle \tilde{F}(X,U)-\tilde{F}(Y,V),U-V\rangle \geq \alpha \|U-V\|^2
    \end{gather}
\end{hyp}

\begin{hyp}
\label{hyp: regularity: strong monotonicity in F}
There exists  $\gamma \in (0,1]$ such that
    \begin{enumerate}
    \item[-] There exists a constant $C$ such that $\forall (X,Y,U,V)\in \mathcal{H}^4$
    \begin{enumerate}
        \item[-] $\langle \tilde{G}(X,U)-\tilde{G}(Y,V),U-V\rangle \leq C\left(\|X-Y\|_\gamma^{2}+\|U-V\|^{2}_\gamma\right)$,
        \item[-]  $\langle \tilde{F}(X,U)-\tilde{F}(Y,V),X-Y\rangle \geq- C\left(\|X-Y\|^{2}+\|U-V\|^{2}_\gamma\right)$,
        \item[-] $\|\tilde{W}_0(X)-\tilde{W}_0(Y)\|\leq C\|X-Y\|_\gamma$.
         \end{enumerate}
    \item[-] There exists a modulus $\omega(\cdot)$ such that $\forall (x,y,\mu,\nu)\in (\reels^d)^2\times \left(\mathcal{P}_2(\reels^d)\right)^2$
     \[|F(x,u,\mu)-F(x,u,\nu)|+|G(x,u,\mu)-G(x,u,\nu)|+|W_0(x,\mu)-W_0(x,\nu)|\leq \omega(\mathcal{W}_2(\mu,\nu))\]
    \end{enumerate}
\end{hyp}
\begin{lemma}
Under Hypotheses \ref{hyp: strong monotonicity in F},\ref{hyp: regularity: strong monotonicity in F} and \ref{hyp: linear growth in F}, there exists a constant $C$ such that if there exists a  strong solution $(X,W^x,Z^x)_{t\in[0,T]}$ in $\mathcal{H}^3$ to \eqref{eq: FBSDE} (resp. $(Y,W^y,Z^y)_{t\in[0,T]}$) with initial datum $X_0\in \mathcal{O}$ (resp. $Y_0$),
\begin{gather*}
    \exists C>0, \quad \forall t\leq T,\\
    \|W_t^x-W_t^y\|\leq C\|X_0-Y_0\|_{\frac{\gamma^2}{2-\gamma^2}},\\
    \|X_t-Y_t\|\leq C\|X_0-Y_0\|_\frac{\gamma}{2-\gamma^2}.
\end{gather*}
\end{lemma}
\begin{proof}
The estimate follows from the exact same argument as Lemma \ref{weak strong estimate in F}, except the strong monotonicity of $(F,G)$ now yields 
\[\forall t\leq T \quad \alpha \int_0^t\|W^x_s-W^y_s\|^2ds\leq \esp{U_0}.\]
Hence we get the existence of a constant $C$ depending on $F,T$ and $\alpha$ such that
\[\forall t\leq T \quad \|X_t-Y_t\|^2\leq C\left(\|X_0-Y_0\|^2+\esp{U_0}^{\gamma}\right),\]
from Gronwall's lemma. Turning to an estimate on $t\mapsto\|W^x_t-W^y_t\|$, we get the existence of another constant $C$ depending on $G,F,W_0,\alpha$ such that 
\[\forall t\leq T \quad \|W_t^x-W_t^y\|^2 \leq C\left(\|X_0-Y_0\|^{2}_\gamma+\esp{U_0}^{\gamma}+\esp{U_0}^{\gamma^2}\right).\]
Evaluating this expression for $t=0$ yields
\[\|W_0^x-W_0^y\|^2\leq C\left(\|X_0-Y_0\|^{2}_\gamma+ \|X_0-Y_0\|^2_{\frac{\gamma}{2-\gamma}}+\|X_0-Y_0\|^2_{\frac{\gamma^2}{2-\gamma^2}}\right),\]
for a constant depending on $T,F,G,W_0,\alpha$ and $\gamma$, ending the proof. 
\end{proof}
Once again a direct consequence is the following existence result
\begin{thm}
     Let $(F,G,W_0)$ be such that Hypotheses \ref{hyp: strong monotonicity in F}, \ref{hyp: regularity: strong monotonicity in F} and \ref{hyp: linear growth in F} are satisfied, then there exists a unique continuous monotone solution to \eqref{eq: ME without noise}
\end{thm}
\begin{remarque}
    Let us remark at this point that those existence results also appear new for corresponding finite state space master equation of the form
    \[\partial_t W+f(x,W)\cdot \nabla_x W=g(x,W),\]
    which is a particular case of \eqref{eq: ME without noise} when coefficients do not depend on the measure argument. In this situation, the definition of monotone solution we presented is exactly the one introduced in \cite{bertucci-monotone-finite}.
\end{remarque}

\section{Master equations with idiosyncratic noise}
\label{section idiosyncratic noise}
\subsection{Treatment of idiosyncratic noise and definition of monotone solutions}
We now turn to the master equation with idiosyncratic noise

\begin{equation}
\label{eq: ME with idiosyncratic}
\left\{
\begin{array}{c}
\displaystyle \partial_t W+F(x,m,W)\cdot \nabla_x W-\sigma_x \Delta_x W\\
\displaystyle +\int_{\reels^d} F(y,m,W)\cdot D_m W(t,x,m)(y)m(dy)-\sigma_x\int_{\reels^d} \text{div}_y (D_mW(t,x,m)(y))m(dy)\\
\displaystyle =G(x,m,W) \text{ in } (0,T)\times\reels^d\times\mathcal{P}_2(\reels^d),\\
\displaystyle W(0,x,m)=W_0(x,m) \text{ for } (x,m)\in\reels^d\times\mathcal{P}_2(\reels^d).
 \end{array}
 \right.
\end{equation}
The main difference is that, due to the presence of idiosyncratic noise among the players, it is now impossible to lift the equation on the Hilbert space of square integrable random variable. Instead one has to work directly on the space of probability measures. While this makes some arguments more technical, there is no particular difficulty in showing uniqueness in the class of smooth solutions. To lighten technical arguments, the assumption $\sigma_x>0$ will be in force throughout the section.
 \begin{prop}
 \label{prop: uniqueness W2 b(p)}
Under Hypotheses \ref{hyp: weak monotonicity} and \ref{hyp: linear growth in F}, there exists at most one smooth solution to \eqref{eq: ME with idiosyncratic} with linear growth, and if there is indeed such solution, it is $L^2-$monotone. 
 \end{prop}
\begin{proof}
Consider two smooth solutions $W^1$ and $W^2$. We define 
    \[\forall \gamma\in\mathcal{P}_2(\reels^{2d})\quad Z(t, \gamma)=\int_{\reels^{2d}}( W^1(t,u,\pi_d  \gamma)-W^2(t,v,\pi_{-d} \gamma))\cdot(x-y)\rangle  \gamma(du,dv).\]
Let us first take a look at the derivative of $y\mapsto D_\gamma Z(t, \gamma,y)$, $D_yD_\gamma Z$, for a fixed $(t, \gamma)$. Letting $y=(y_1,y_2)\in \left(\reels^d\right)^2$, taken at this point, it is a block matrix of the form 
\begin{align*}
    \begin{array}{c}
         D_yD_\gamma Z(t,\gamma,y)=\left(\begin{array}{cc}
           M_{11}  & M_{12} \\
            (M_{12})^T  & M_{22}
         \end{array}\right),  \\
         \displaystyle M_{11}=D_x^2W^{y_1}\cdot(y_1-y_2)+(D_x W^{y_1}+\left(D_x W^{y_1}\right)^T+\int D_{y_1}D_m W^1(t,u,\mu,y_1)\cdot(u-v)\gamma(du,dv),\\
         M_{12}=-D_x W^{y_2}-\left(D_x W^{y_1}\right)^T,\\
         \displaystyle M_{22}=D_x^2W^{y_2}\cdot(y_2-y_1)+(D_x W^{y_2}+\left( D_x W^{y_2}\right)^T-\int D_{y_2}D_m W^2(t,u,\nu,\theta,y_2)\cdot(u-v)\gamma(du,dv),\\
         W^{y_1}=W^1(t,y_1,\mu),\\
         W^{y_2}=W^2(t,y_2,\nu),\\
         \pi_d  \gamma=\mu \quad \pi_{-d} \gamma=\nu.\\
    \end{array}
\end{align*}
Since the terms composing $M_{12}$ and its transpose are the opposite of terms appearing in $M_{11},M_{22}$, they obviously cancel out in 
\begin{align*}
\text{Tr}\left(B_d D_yD_mZ(t,\gamma,y)\right)&=(\Delta_xW^{y_1}-\Delta_x W^{y_2})\cdot(y_1-y_2)\\
&+\int (\text{div}_{y_1}(D_m W^1)(t,u,\mu,y_1)-\text{div}_{y_2}(D_mW^2)(t,v,\nu,m,y_2))\cdot(u-v)\gamma(du,dv).
\end{align*}
Expressing $\partial_t Z$ using the fact that $W^1,W^2$ are solutions of \eqref{eq: ME with idiosyncratic} and using the above equality, we get that $Z$ is a solution of
\begin{equation}
\label{Z b(p) supersol}
\left\{
\begin{array}{c}
\partial_t Z+\displaystyle\int_{\reels^{2d}}\left(\begin{array}{c}
     F^{u}  \\
     F^{v}
\end{array}
\right)(u,v,\gamma)
\cdot D_\gamma Z(t,\gamma,u,v) \gamma(du,dv)\\
-\displaystyle\sigma_x\int_{\reels^{2d}}\text{Tr}\left(B_d D_{(u,v)} \left[D_mZ\right](t,\gamma,u,v)\right)\gamma(du,dv)\\
=\displaystyle\int_{\reels^{2d}} (F^{u}-F^{v})\cdot(W^{u} -W^{v})\gamma(du,dv)+\int_{\reels^{2d}}(G^{u}-G^{v})\cdot(u-v)\gamma(du,dv),\\
\text{ for } (t,\gamma)\in (0,T)\times\mathcal{P}_2(\reels^{2d}),
\end{array}
\right.
\end{equation}
with $W^{u}=W^1(t,u,\pi_d \gamma), W^{v}=W^2(t,v,\pi_{-d}\gamma), F^{u}=F(u,\pi_d \gamma,W^{u})$ and so on. 
By assumptions on $W_0,F,G$ this leads to $Z_{t=0}\geq 0$ and 
\begin{equation*}
\begin{array}{c}
\partial_t Z
+\displaystyle\int_{\reels^{2d}}\left(\begin{array}{c}
     F^{u}  \\
     F^{v}
\end{array}
\right)
\cdot D_\gamma Z(t,\gamma,\noise,u,v) \gamma(du,dv)-\sigma_x\int_{\reels^{2d}}\text{Tr}\left(B_dD_{(u,v)} \left[D_\gamma Z\right](t,\gamma,\noise,u,v)\right)\gamma(du,dv)\geq 0,\\
\text{ for } (t,\gamma,\noise)\in (0,T)\times\mathcal{P}_2(\reels^{2d})\times \reels^n.
\end{array}
\end{equation*}
By the maximum principle, see \cite{common-noise-in-MFG} Lemma 4.3 for exemple, $Z$ is non-negative, which is equivalent to saying that 
\[\forall (t,X,Y)\in[0,T]\times \left(L^2(\Omega,\reels^d)\right)^2 \quad \esp{\left(\tilde{W}^1(t,X)-\tilde{W}^2(t,Y)\right)\cdot\left(X-Y\right)}\geq 0.\]
From the continuity of either $W^1$ or $W^2$, we deduce as in the proof of Lemma \ref{lemma: smooth uniqueness Hilbert approach}, that this implies $W^1=W^2$, and as a byproduct the $L^2-$monotonicity of the unique solution. 
\end{proof}
Following previous sections and in light of the above proof, we may try to define a notion of monotone solutions by considering the auxiliary function 
`\[Z^V:(t,\gamma)\mapsto \int_{\reels^{2d}} (W(t,x,\pi_d \gamma)-V(y,\pi_{-d}\gamma))\cdot (x-y)\gamma(dx,dy),\]
for a solution $W$ to the master equation and sufficiently smooth functions $V$. The main difficulty with this approach lies within the treatment of cross derivatives terms appearing within equation \eqref{Z b(p) supersol}. Indeed in this case, for a given function $V$, $Z_V$ satisfies 
\begin{align}
\nonumber &\partial_t Z^V(t,m)+\int_{\reels^{2d}} F(x,m,W(t,x,\pi_d m))\cdot \nabla_x \nabla_m Z^V(t,m)(x,y)m(dx,dy)\\
\label{eq z v demo }
&-\sigma_x \int_{\reels^d}\Delta_x \nabla_m Z^V(t,m,x,y)m(dx,dy)+2\sigma_x\int_{\reels^{2d}}\text{div}W(t,x,\pi_d m)m(dx,dy)\\
\nonumber&= \int_{\reels^{2d}} F(x,W(t,x,\pi_d m),\pi_dm)\cdot(W(t,x,\pi_dm)-V(y,\pi_{-d}m))m(dx,dy)\\
\nonumber &+\int_{\reels^{2d}} G(x,W(t,x,\pi_d m),\pi_dm^*)\cdot (x-y)m(dx,dy),
\end{align}
where the term involving the divergence of $W$ is related to the cross derivative of $\nabla_m Z^V$ by 
\begin{equation}
\label{eq of diagonal terms sigma}
\int_{\reels^{2d}}\text{Tr}\left(D_xD_y \nabla_m Z\right)(t,m,x,y)m(dx,dy)=-\int_{\reels^{2d}}\left(\text{div}W (t,x,\pi_dm )+\text{div}V(y,\pi_{-d}m)\right)m(dx,dy).\end{equation}
However at a point of minimum $(t^*,m^*)$ of $Z^V-\varphi$ for a smooth function $\varphi$ the information available yields
\[\forall \sigma \in C^{0,1}(\reels^{2d},\mathcal{M}_{2d}(\reels)),\quad \int \text{Tr}\left[\sigma \sigma^T(x,y)D^2_{(x,y)}\nabla_m\left(Z^V-\varphi\right)(t^*,m^*,x,y)\right]m^*(dx,dy)\geq 0,\]
which does not allow to relate directly the cross derivatives of $\nabla_m Z^V$ and $\nabla_m \varphi$. Clearly, if we naively define a notion of monotone solution by just replacing the derivatives of $Z^V$ with that of $\varphi$ in \eqref{eq z v demo } at points of minimum, this lack of information on the cross derivatives terms is going to break the maximum principle and hence the proof of uniqueness we presented in Lemma \ref{prop: uniqueness W2 b(p)}. Similar problems were encountered in the comparison of second order viscosity solutions in $\reels^d$ and solved by the development of Theorem 3.2 \cite{crandall1992users} and its variants. Unfortunately to the extent of our knowledge, we lack the equivalent of such result for second order equations on the space of probability measures. Instead we now remark that if the point of minimum is reached at a sufficiently smooth measure of probability, further information can be obtained. To that end we now introduce the fisher information $\mathcal{I}$, for a measure $\mu\in\mathcal{P}_2(\reels^d)$,
\[\mathcal{I}(\mu)=\left\{\begin{array}{c}
     \displaystyle+\infty \text{ if $\mu$ does not have a density with respects to Lebesgue measure on }\reels^d,  \\
    \displaystyle \int_{\reels^d} |\nabla_x \log(\mu(x))|^2\mu(dx) \text{ otherwise},
\end{array}
\right.\]
\begin{lemma}
Let $U:\mathcal{P}_2(\reels^d)\to \reels$ be a smooth function of measure such that there exists $C>0$
\[\forall (x,m)\in \reels^d\times \mathcal{P}_2(\reels^d), \quad |D_mU(m,x)|,\sqrt{|D_xD_mU(m,x)|}\leq C\left(1+|x|\right).\]
If $U$ reaches a minimum at $\mu\in \mathcal{P}_2(\reels^d)$, a measure with finite fisher information 
\[\mathcal{I}(\mu)=\int_{\reels^d} \left|\nabla_x \log(\mu(x))\right|^2\mu(dx)<+\infty,\]
then, for any continuous function $\Sigma\in C^0(\reels^d,\mathcal{M}_d(\reels))$,
\[ \int \text{Tr}\left[\Sigma(x) D_{x}D_mU(\mu,x)\right]\mu(dx)=0.\]
\end{lemma}
\begin{proof}
We start by proving this result for $\Sigma \in C^{1,1}(\reels^d,\mathcal{M}_d(\reels^d))\cap C_b(\reels^d,\mathcal{M}_d(\reels^d))$. By the regularity and growth assumptions on $U$,it follows (see \cite{probabilistic-mfg}) that its lift on $L^2(\Omega,\reels^d)$ $\tilde{U}$ is everywhere differentiable in the Frechet sense and 
\begin{equation}
\label{gradient eq gradient}
\forall X\in L^2(\Omega,\reels^d), \quad \nabla \tilde{U}(X)=D_mU(\mathcal{L}(X),X).\end{equation}
Since $\mathcal{I}(\mu)<+\infty$, let us also remark that for any $X\sim \mu$, the random variable
\[Y=\text{div}\Sigma(X)+\Sigma(X)\nabla_x \log(\mu(X)),\]
belongs to $L^2(\Omega,\reels^d)$. Let us fix such $X$, since $U$ reaches its minimum at $\mu$, $\tilde{U}$ also reaches a minimum at $X$ it follows that 
\[\forall Z\in L^2(\Omega,\reels^d) \quad \langle \nabla \tilde{U}(X),Z\rangle=0.\]
Taking this expression for $Z=Y$ and using \eqref{gradient eq gradient} yields 
\[\int_{\reels^d} D_mU(m,x)\cdot\left(\text{div}\Sigma(x)+\Sigma(x)\nabla_x\log (\mu(x))\right)\mu(dx)=0.\]
This implies 
\[\int_{\reels^d} D_mU(m,x)\cdot \text{div}(\Sigma\mu)(x)dx=0,\]
and the conclusion follows by integration by part. For less regular functions $\Sigma$, we conclude by a density argument and Fatou's lemma.
\end{proof}
\begin{remarque}
Taking $\sigma:\reels^d\to \reels$ for simplicity, this is a natural consequence of the fact that for sufficiently smooth measures, the linear Fokker-Planck equation 
\[\partial_t \mu_t-\Delta_x (\sigma(x) \mu_t)=0,\]
can be re-casted as a non linear transport equation 
\[\partial_t \mu_t-\text{div}\left(\left(\nabla_x \sigma(x)+\nabla_x \log(\mu_t(x))\right)\mu_t\right)=0.\]
This explains why second order terms on the space of probability measures of the form $\int D_x D_m U(m,x)dm$ are not truly second order. 
\end{remarque}
Following previous article \cite{bertucci-lions-wasserstein-distance,Daudin} dealing with viscosity solutions on the space of probability measures, we introduce a penalization based on the entropy to ensure the fisher information is finite at points of minimum. Let us remind that for a measure $\mu\in \mathcal{P}_2(\reels^d)$, its entropy $\mathcal{E}(\mu)$ is defined by
\[\mathcal{E}(\mu)=\left\{\begin{array}{c}
     +\infty \text{ if $\mu$ does not have a density with respects to Lebesgue measure on }\reels^d  \\
     \displaystyle\int_{\reels^d} \log(\mu(x))\mu(x)dx \text{ otherwise}.
\end{array}\right.\]
The entropy of a probability measure is related to the Fisher information in that formally (in that the computation does not usually hold in a proper sense)
\[D_\mu \mathcal{E}(\mu,x)=\nabla_x \log(\mu(x)),\]
and so 
\[\mathcal{I}(\mu)=\int_{\reels^d} |D_\mu \mathcal{E}(\mu,x)|^2d\mu.\]
\begin{lemma}
\label{lemma: integ part fisher}
    Let $V\in C^1(\reels^d,\reels^d)$, be a function with linear growth and  $\mu\in\mathcal{P}_2(\reels^d)$ such that 
    \[\mathcal{I}(\mu)<+\infty.\]
    If $\left(\text{div}(V)\right)^-\in L^1(\mu)$,then the following holds
    \[\int_{\reels^d} \text{div}(V)(x)\mu(dx)=-\int_{\reels^d} V(t,\mu,x) \cdot \nabla_x \log (\mu(x))\mu(dx),\]
\end{lemma}
\begin{proof}
Let $(\eta_k)_{k\geq 0}\subset C^\infty(\reels^d,[0,1])$ be a sequence of smooth functions such that for all $x\in \reels^d$
\begin{enumerate}
    \item[-]$|x|\leq k \implies  \eta_k(x)=1,$
    \item[-]$ |x|\geq k+1\implies  \eta_k(x)=0,$
    \item[-] $|\nabla_x \eta_k(x)|\leq C_\eta$ for a constant $C_\eta$ independent of $k$.
\end{enumerate}
Let us fix $t>0$, for each $k\geq 0$, the computation 
\[\int_{\reels^d} \text{div}(V\eta_k)(t,\mu,x)\mu(dx)=-\int_{\reels^d} \eta_kV(t,\mu,x) \cdot \nabla_x \log (\mu(x))\mu(dx),\]
is justified as $V\eta_k$ is $C^1$ and bounded. On the right hand side, using the growth of $V$, the convergence as $k\to \infty$ follows from the dominated convergence theorem. On the left hand side 
\[\int_{\reels^d} \text{div}(V\eta_k)(t,\mu,x)\mu(dx)=\underbrace{\int_{\reels^d} \text{div}(V)^+\eta_kd\mu}_{I_k^+}+\underbrace{\int_{\reels^d} \text{div}(V)^-\eta_k d\mu}_{I_k^-}+\underbrace{\int_{\reels^d} V\cdot \nabla_x \eta_k d\mu}_{I_k^\eta}.\]
The convergence of $I^+_k$ follows from the monotone convergence theorem, and the convergence of $I_k^-$ from the dominated convergence theorem. Finally using the growth of $V$ we deduce that there exists a constant $C_V$ independent of $k$ such that
\[I_k^\eta \leq C_VC_\eta \int_{|x|\geq k}(1+|x|)\mu(dx).\]
Since $\mu\in\mathcal{P}_2(\reels^d)$
\[\lim_{k\to \infty} I^\eta_k=0,\]
ending the proof. 
\end{proof}
Before we give the notion of solution we use for \eqref{eq: ME with idiosyncratic}, Let us present some notations: a function $\psi:[0,T)\times \mathcal{P}_2(\reels^{2d})\to \reels$ is said to belong to the space of test functions $H_{test}$ if it is of the form 
\[\varphi(t,m)=\psi(t)+\int_{\reels^{2d}}f_\varphi(\bar{x})m(d\bar{x})-\alpha e^{\lambda t}\left(E_3(\pi_dm)+E_3(\pi_{-d}m)\right),\]
where
\begin{enumerate}
    \item[-] $\psi$ is uniformly bounded in $C^1([0,\tau],\reels)$ for any $\tau<T$,
    \item[-] $\alpha,\lambda  >0$,
    \item[-] $\|f_\varphi\|_{C^2}=\|f_\varphi\|_\infty+\|Df_\varphi\|_\infty+\|D^2f_\varphi\|_\infty<+\infty$,
    \item[-] $\forall \mu\in \mathcal{P}_2(\reels^d), \quad E_3(\mu)=\int_{\reels^d} |x|^3\mu(dx).$
\end{enumerate}
We define a monotone solution as follow
\begin{definition}
\label{def: monotone sol sig}
A function $W \in C([0,T]\times \reels^d\times\mathcal{P}_2(\reels^d), \reels^d)$ is said to be a monotone solution to \eqref{eq: ME with idiosyncratic} iff for every function $V \in C([0,T]\times \reels^d\times\mathcal{P}_2(\reels^d), \reels^d)$, with at most linear growth, the function $Z:[0,T]\times \mathcal{P}_2(\reels^{2d})$ defined by
\[Z:(t,m)\mapsto \int_{\reels^{2d}} (W(t,x,\pi_dm)-V(y,\pi_{-d}m))\cdot(x-y)m(dx,dy), \]
is such that at every point of minimum $(t^*,m^*)\in (0,T)\times \mathcal{P}_2(\reels^{2d})$ of 
\[(t,m)\mapsto Z(t,m)-\varphi(t,m)+\kappa\mathcal{E}(m),\]
on $[0,T]\times \mathcal{P}_2(\reels^{2d})$ for a $\varphi\in H_{test}$ and $\kappa>0$
the following holds 
\[\int_{\reels^{2d}} |\nabla_x\log(m^*(x,y))|^2m^*(dx,dy)<+\infty,\]
and
\begin{align*}&\partial_t \varphi(t^*,m^*)+\int_{\reels^{2d}} F(x,m^*,W(t^*,x,\pi_d m^*))\cdot \nabla_x \nabla_m \varphi(t^*,m^*)(x,y)m^*(dx,dy)-\sigma_x \int_{\reels^{2d}}\Delta_x \nabla_m \varphi(t^*,m^*,x,y)m^*(dx,dy)\\
&-2\sigma_x\int_{\reels^{d}}W(t^*,x,\pi_d m^*)\cdot \nabla_x\log(m^*(x,y))m^*(dx,dy)-\sigma_x \kappa\int_{\reels^{2d}} |\nabla_x \log(m^*(x,y))|^2m^*(dx,dy)\\
&-\kappa\int_{\reels^{d}} F(x,m^*,W(t^*,x,\pi_d m^*))\cdot \nabla_x \log( m^*(x,y))\pi_dm^*(dx,dy)\\
&\geq \int_{\reels^{2d}} F(x,W(t,x,\pi_d m^*),\pi_dm^*)\cdot(W(t,x,\pi_dm)-V(y,\pi_{-d}m^*))m^*(dx,dy)\\
&+\int_{\reels^{2d}} G(x,W(t,x,\pi_d m^*),\pi_dm^*)\cdot (x-y)m^*(dx,dy).
\end{align*}
\end{definition}
\begin{remarque}
The term 
\[\int_{\reels^{2d}} W(t^*,x,\pi_dm^*)\cdot \nabla_x\log(m^*(x,y)))m^*(dx,dy)=-\int_{\reels^{2d}}\text{div}_x(W)(t^*,x,\pi_dm^*)m^*(dx,dy),\]
comes from an integration by part and not from the entropic penalization. Indeed this allows the definition of solutions in $C([0,T)\times \reels^d\times \mathcal{P}_2(\reels^d),\reels^d)$ instead of $C([0,T)\times \mathcal{P}_2(\reels^d),C^1(\reels^d,\reels^d))$. 
\end{remarque}
Let us insist that besides the addition of an entropic penalization, this definition appears as a natural extension to the notion of monotone solution we presented for equations without idiosyncratic noise. 
\begin{thm}
\label{stability with idio}
   Let $(F_n,G_n)_{n\in\mathbb{N}}$ be a sequence of function satisfying Hypothesis \ref{hyp: linear growth in F} uniformly in $n\in\mathbb{N}$, $(\sigma_n)_{n\in\mathbb{N}}\subset \reels^+$, and $(W_n)_{n\in\mathbb{N}}$ be an associated sequence of monotone solutions. Suppose $(F_n,G_n,\sigma_n)_{n\in\mathbb{N}}$ converges locally uniformly to $(F,G,\sigma_x)$ in  $C(\reels^d\times\mathcal{P}_2(\reels^d)\times\reels^d,\reels^d)$, $(W_n)_{n\in\mathbb{N}}$ converges locally uniformly in $C([0,T]\times \reels^d\times \mathcal{P}_2(\reels^d),\reels^d)$ to a function $W$ and that 
   \[\exists C>0, \quad \forall n\in\mathbb{N}, \quad  \forall (t,\mu,x)\in [0,T]\times \mathcal{P}_2(\reels^d)\times \reels^d, \quad |W_n(t,x,\mu)|\leq C\left(1+|x|+\sqrt{E_2(\mu)}\right).\]
    Then $W$ is a monotone solution associated with the data $(F,G,\sigma_x)$.
\end{thm}
\begin{proof}
    Assume that for a function $V\in C([0,T]\times \reels^d\times \mathcal{P}_2(\reels^d),\reels^d)$ with linear growth and a $\varphi \in H_{test}$
    \[Z:(t,m)\mapsto \int_{\reels^{2d}} (W(t,x,\pi_d m)-V(y,\pi_{-d}m))\cdot (x-y)m(dx,dy)-\varphi(t,m)+\kappa \left(\mathcal{E}(\pi_dm)+\mathcal{E}(m)\right),\]
    reaches a strict minimum at $(t^*,m^*)$ with $t^*\neq 0$. We now consider 
    \[Z_n:(t,m)\mapsto \int_{\reels^{2d}} (W_n(t,x,\pi_d m)-V(y,\pi_{-d}m))\cdot (x-y)m(dx,dy)-\varphi(t,m)+\kappa \mathcal{E}(m),\]
   By our assumptions on test functions, and the growth of the family $(W_n)_{n\in\mathbb{N}}$, there exists a constant $C$ such that 
   \[\forall (n,t,m)\in\mathbb{N}\times [0,T]\times \mathcal{P}_2(\reels^{2d}) \quad Z_n(t,m)\geq -C +C\int_{\reels^{2d}} |x|^3m(dx).\]
   Let us also remind that by Lemma 2.8 in \cite{Daudin}, the function $\mu\mapsto \mathcal{E}(\mu)$ is lower semi-continuous for the topology of weak convergence and 
\begin{equation}
\label{ineq entropy} \forall \mu\in \mathcal{P}_2(\reels^d), \quad \mathcal{E}(\mu)\geq -\pi \int_{\reels^d}|x|^2\mu(dx).
\end{equation}
By the continuity of $(t,s,m)\mapsto Z(t,s,m)$ and the lower semi-continuity in $\mathcal{W}_2$  of 
\[m\mapsto \int_{\reels^{2d}} |\bar{x}|^3m(d\bar{x}),\]
$Z_n$ is bounded from below and lower semi-continuous on $\left(\reels^+\right)^2\times\mathcal{P}_2(\reels^{2d})$. By a variation of stegall's lemma (as in \cite{bertucci-monotone} Lemma 2.1), for any $n\in \mathbb{N}$, there always exists a function  $f_n\in C^2(\reels^{2d},\reels)$, such that $\|f\|_{C^2}\leq \frac{1}{n}$ and 
\[Z(t,m)\mapsto Z_{n}(t,m)+\int_{\reels^{2d}} f_n(x,y)m(dx,dy),\]
reaches a strict minimum at some point $(t^*_n,m^*_n)$. Furthermore since 
    \[\forall n\in\mathbb{N}\quad Z_n(t^*_n,m^*_n)\leq Z_n(0,\mu_g)-\varphi(0,\mu_g)+\kappa\mathcal{E}(\mu_g),\]
    for $\mu_g$ the law of a centered gaussian random variable with variance $I_{2d}$, the sequence $(m^*_n)_{n\in\mathbb{N}}$ has a uniformly bounded third moment and so the sequence $(t^*_n,m^*_n)_{n\in\mathbb{N}}$ is compact in $\reels^+\times\mathcal{P}_2(\reels^{2d})$ and converges to some $(\bar{t},\bar{m})$ along a subsequence. 
    Let us also observe that 
    \[\forall n\in\mathbb{N} \quad Z_n(t_n^*,m_n^*)\leq Z_n(t^*,m^*),\]
    which yields 
    \[\underset{n\to\infty}{\liminf}Z_n(t^*_n,m^*_n)\leq Z(t^*,m^*).\]
    By the local uniform convergence of $(W_n,V_n)$ to $(W,V)$ in $C([0,T]\times\reels^d\times \mathcal{P}_2(\reels^d),\reels^d)$ and the lower semi-continuity of $-\varphi, \mathcal{E}$ we conclude as in Theorem \ref{stability without noise} that 
    \[Z(\bar{t},\bar{m})\leq\underset{n\to\infty}{\liminf}Z_n(t^*_n,m^*_n)\leq Z(t^*,m^*), \]
    which implies $(\bar{t},\bar{m})=(t^*,m^*)$.
 Since $t^*\neq0$ the following holds 
    \[\exists N_0 \quad \forall n\geq N_0 \quad t^*_n\neq 0.\]
    Applying the definition of monotone solution to $W_n$, for $n\geq N_0$ we get the following inequality:
\begin{align}
\nonumber&\partial_t \varphi_n+\int_{\reels^{2d}} F_n(x,m^*_n,W_n(t^*_n,x,\pi_d m^*_n))\cdot \nabla_x \nabla_m \varphi_n(x,y)dm^*_n-\sigma_n \int_{\reels^{2d}}\Delta_x \nabla_m \varphi_ndm^*_n\\
\nonumber&-2\sigma_n\int_{\reels^{d}}W_n(t^*_n,x,\pi_d m^*_n)\cdot \nabla_n\log(\pi_dm^*_n(x))\pi_dm_n^*(dx)-\sigma_n \kappa \int_{\reels^{2d}} |\nabla_x \log(m^*_n(x,y))|^2m^*_n(dx,dy)\\
\label{eq: min with idio stab}
&-\kappa\int_{\reels^{d}} F_n(x,\pi_dm^*_n,W_n(t^*_n,x,\pi_d m^*_n))\cdot \nabla_x \log(\pi_dm^*_n(x))\pi_dm^*_n\\
\nonumber &\geq \int_{\reels^{2d}} F_n(x,W_n(t^*_n,x,\pi_d m^*_n),\pi_dm^*_n)\cdot(W_n(t^*_n,x,\pi_dm^*_n)-V(y,\pi_{-d}m^*_n))m^*_n(dx,dy)\\
\nonumber &+\int_{\reels^{2d}} G_n(x,W_n(t^*_n,x,\pi_d m^*_n),\pi_dm^*_n)\cdot (x-y)m^*_n(dx,dy),
\end{align}
with the notation
\[\varphi_n=\varphi(t^*_n,m^*_n)-\int_{\reels^{2d}}f_ndm^*_n.\]
Let us first show that 
\[I_x(m^*_n)= \int_{\reels^{2d}} |\nabla_x \log(m^*_n(x,y))|^2m^*_n(dx,dy),\]
 can be bounded independently of $n\in \mathbb{N}$. This follows from an argument presented in  \cite{bertucci-lions-wasserstein-distance}, namely by using the growth of $(W_n,V_n,F_n,G_n,\varphi_n)_{n\in\mathbb{N}}$, we find a constant $C$ independent of $n\in\mathbb{N}$ such that 
\[\forall n\in\mathbb{N},\quad \sigma_n\kappa  \mathcal{I}_x(m^*_n)\leq C\left(1+E_3(m_n^*)+(\kappa+\sigma_n) \sqrt{E_2(m^*)}\sqrt{\mathcal{I}(\pi_dm^*_n)}\right).\]
Using the convexity of $(u,v)\mapsto \frac{u^2}{v}$, Jensen's inequality yields 
\[I(\pi_d m^*_n)\leq I_x(m^*_n),\]
and since $\sigma_x>0$, for $n$ sufficiently large we have 
\[ \mathcal{I}_x(m^*_n)\leq C\left(1+E_3(m^*_n)+\frac{2(\sigma_n+\kappa)}{\sigma_n\kappa}E_2(m^*)\right).\]
In particular, using Fatou's lemma and the boundedness of $(m^*_n)_{n\in\mathbb{N}}$ in $\mathcal{P}_3(\reels^d)$
\[\mathcal{I}_x( m^*)\leq C\left(1+\limsup_nE_3(m^*_n)+\frac{2(\sigma_x+\kappa)}{\sigma_x\kappa}E_2(m^*)\right)<+\infty. \]
Let $(X_n,Y_n)_{n\in\mathbb{N}}$  denote a sequence of random variable $(X_n,Y_n) \sim m^*_n$ converging strongly to $X\sim \pi_dm^*$, and  $Z_n=\nabla_x\log(m^*_n(X_n,Y_n))$. Since the sequence $(\mathcal{I}_x(m^*_n))_{n\in\mathbb{N}}$ is bounded, the sequence $(Z_n)_{n\in\mathbb{N}}$ is bounded in $\mathcal{H}$ and hence converges weakly to some limit $Z$ and by Lemma 6.3 of \cite{bertucci-lions-wasserstein-distance} $Z=\nabla_x \log(m^*(X,Y))$. By the strong convergence in $\mathcal{H}$ of $(W_n(t_n,X_n,\mathcal{L}(X_n)))_{n\in\mathbb{N}}$, 
\begin{gather*}\lim_{n\to\infty}\int_{\reels^d} W_n(t^*_n,x,\pi_dm^*_n)\cdot \nabla_x \log(\pi_dm^*_n(x))\pi_dm^*_n(dx)=\lim_{n\to\infty} \langle W_n(t^*_n,X_n,\mathcal{L}(x_n)), Z_n\rangle\\
=\int_{\reels^d} W(t^*,x,\pi_dm^*)\cdot \nabla_x \log(\pi_dm^*(x))\pi_dm^*(dx),\end{gather*}
the convergence of $(\langle F_n(X_n,\mathcal{L}(X_n),W_n(t^*_n,X_n,\mathcal
{L}(X_n))),Z_n\rangle)_{n\in\mathbb{N}}$ also follows from this argument. Finally for terms depending on $\varphi$, we remind that 
\[\varphi(t,m)=\psi(t)+\int f(\bar{x})m(d\bar{x})-\alpha e^{\lambda t}E_3(m),\]
for some $\alpha,\lambda>0$ and $\|f\|_{C^2}<+\infty$, and that using a variation of Fatou's lemma, we have for $\lambda>\lambda^*$ 
\begin{gather*}\partial_t \varphi(t^*,m^*)+\int_{\reels^d} F(x,m^*,W(t^*,x,m^*))\cdot \nabla_x \nabla_m \varphi(t^*,m^*)m^*(dx,dy)-\sigma_n \int_{\reels^{2d}}\Delta_x \nabla_m \varphi(t^*,m^*) dm^*\\
\geq \limsup_n\left(\partial_t \varphi_n+\int_{\reels^{2d}} F_n(x,m^*_n,W_n(t^*_n,x,\pi_d m^*_n))\cdot \nabla_x \nabla_m \varphi_n(x,y)dm^*_n-\sigma_n \int_{\reels^{2d}}\Delta_x \nabla_m \varphi_ndm^*_n\right),
\end{gather*}
where $\lambda^*$ is a constant depending only on the growth of $(F_n,W_n)_{n\in\mathbb{N}}$.
Taking the $\limsup$ in \eqref{eq: min with idio stab} we then deduce that, up to a non consequential restriction on the space of test functions $W$ satisfy the property of monotone solutions at the minimum $(t^*,m^*)$. Since this is true for any continuous function $V$ with at most linear growth, $W$ is indeed a monotone solution associated to the coefficients $(F,G,\sigma_x)$.  
\end{proof}
Now that stability of the notion has been established, we turn to the existence of monotone solutions. First, we show that for sufficiently smooth functions, points of minimum are indeed reached for measures with finite fisher information. The following intermediary lemma which is an adaptation of \cite{Daudin} Proposition 3.1 to our setting 
\begin{lemma}
\label{lemma: fisher information is finite when V lipschitz}
    Let $V:\reels\times \mathcal{P}_2(\reels^d)\to \reels^d$ be a continuous function with linear growth, $\mathcal{F}:\mathcal{P}_2(\reels^d)\to \reels$ a lower semi-continuous, bounded from below function, $f\in C^2_b(\reels^{2d},\reels)$ and $W:\reels\times \mathcal{P}_2(\reels^d)\to \reels^d$ a Lipschitz function and define 
    \[Z:m\mapsto \int_{\reels^{2d}} (W(x,\pi_dm)-V(y,\pi_{-d}m))\cdot (x-y)m(dx,dy)+\mathcal{F}(\pi_{-d}m)+\int_{\reels^{2d}} f(x,y)m(dx,dy).\]
    Then at points of minimum $m^*$ of 
    \[m\mapsto Z(m)+\mathcal{E}(m)+E_3(\pi_{d}m),\]
 $m^*$ has a density with respect to Lebesgue measure and 
 \[\int_{\reels^{2d}} |\nabla_x\log(m^*(x,y))|^2m^*(dx,dy)<+\infty \]
\end{lemma}
\begin{proof}
    We start by fixing $v\in C_b^\infty(\reels^{2d},\reels^d)$ and consider the flow of measure 
    \[\partial_t m_t+\text{div}\left(\left(\begin{array}{c}
    v\\
         0_{\reels^d}  
    \end{array}\right)m_t\right)=0, \quad m_0=m^*.\]
    It has already been established in \cite{Daudin} that 
    \[\mathcal{E}(m_t)-\mathcal{E}(\mu^*)=-\int_0^t\int_{\reels^d}\text{div}_x(v(x,y))m_s(dx,dy).\]
    On the other hand, since $m\mapsto Z(m)+E_3(m)+\mathcal{E}(m)$ reaches a minimum at $m^*$
    \begin{equation}
    \label{ineq entropy}
    \mathcal{E}(m_t)+E_3(\pi_{d}m_t)-\left(\mathcal{E}(m^*)+E_3(\pi_dm^*)\right)\geq Z(m^*)-Z(m_t).\end{equation}
    By the definition of $Z$, there exists a constant $C$ depending on the Lipschitz constant of $W$ and the growth of $V$ such that 
    \[|Z(m^*)-Z(m_t)|\leq C\left(1+E_3(\pi_d m^*)+E_2(m^*)\right)\int_0^t\sqrt{\int_{\reels^d} |v(x,y)|^2m_s(dx,dy)}ds+o(t).\]
    Dividing by $t$ and letting it tends to 0 in \eqref{ineq entropy} let us conclude to
    \begin{equation}
    \label{test func entropy}
    \forall v\in C^\infty_b(\reels^{2d},\reels^d), \quad \int_{\reels^d} \left(|x|x\cdot v(x,y)-\text{div}_x(v(x,y))\right)m^*(dx,dy)\geq -C(1+E_2(m^*)+E_3(\pi_{-d}m^*))\|v\|_{L^2(m^*)}\end{equation}
Let us denote by $\nabla_x \log(m^*)$ the distribution defined by
\[\forall \varphi \in C^\infty_c(\reels^{2d},\reels^d), \quad \langle \nabla_x\log(m^*),\varphi\rangle =-\int_{\reels^d} \text{div}_x(\varphi(x,y))m^*\]
it follows from \eqref{test func entropy} that $|x|x+\nabla_x \log(m^*)$ is bounded in operator norm in $L^2(m^*)$ and by duality admits a representative in this space also denoted $|x|x+\nabla_x \log(m^*)$ . Defining a sequence of function $(\eta_k)_{k\in\mathbb{N}}$ as in Lemma \ref{lemma: integ part fisher}, since $x|x|\eta_k\in L^2(\mu^*)$ we deduce that $\nabla_x \log(m^*)\eta_k\in L^2(m^*)$ and consequently that 
\[\int_{\reels^d}\nabla_x \log(m^*(x,y))\cdot x|x|\eta_k(x)m^*(dx,dy)\]
is well defined. By using argument presented in Lemma \ref{lemma: integ part fisher},
\[\left|\int_{\reels^d}\nabla_x \log(m^*(x,y))\cdot x|x|\eta_k(x)m^*(dx)\right|\leq 2(1+C_\eta)(d+1)E_2(\pi_{d}m^*).\]
Using the fact that $x\mapsto (|x|x+\nabla_x \log(m^*(x,y)))\eta_k(x)$ is bounded in $L^2(m^*)$ independently of $k$ we find a constant $C$ such that 
\[\forall k\in \mathbb{N},\quad \int_{\reels^d} |\nabla_x \log(m^*(x,y))|^2\eta_k(x)m^*(dx,dy)+\int_{\reels^d} |x|^4\eta_k(x)\pi_{d}m^*(dx)\leq C(1+E_3(\pi_{d}m^*)).\]
By the monotone convergence theorem it follows that $\nabla_x \log(m^*)\in L^2(m^*)$. Since 
\[\forall v\in C^\infty_b(\reels^{2d},\reels^d) \quad \|v\|_{L^2(m^*)}\leq \|v\|_\infty,\]
\eqref{test func entropy} implies that $m^*$ has a weak derivative $\nabla_x m^*\in L^1(\reels^{2d})$ and from the definition of $\nabla_x(\log m^*)$ we conclude to 
\[\nabla_x \log(m^*)=\frac{\nabla_xm^*}{m^*} \text{ a.e in Supp($m^*$)}\]
\end{proof}

\begin{lemma}
    under Hypothesis \ref{hyp: Lipschitz Wq}, let $W$ be the Lipschitz solution to \eqref{eq: ME with idiosyncratic} defined on $[0,T_c)$ for some $T_C>0$ then $W$ is also a monotone solution in the sense of Definition \ref{def: monotone sol sig}.
\end{lemma}
\begin{proof}
We first start by treating the case for which $(F,G,W_0)$ are smooth in $x$, in which case $W$ is also smooth in this variable for as long as it exists. 
Let $V:\reels^d\times \mathcal{P}_2(\reels^d)\to \reels^d$ be a continuous function with bounded derivatives and define 
\[Z:(t,m)\mapsto \int_{\reels^{2d}} (W(t,x,\pi_dm)-V(y,\pi_{-d}m))\cdot(x-y)m(dx,dy).\]
Let $\varphi\in H_{test}$ and suppose that
\[(t,m)\mapsto (Z-\varphi)(t,m)+\kappa \left( \mathcal{E}(\pi_dm)+\mathcal{E}(\pi_{-d}m)\right),\]
reaches its strict minimum at $(t^*,m^*)\in(0,T)\times \mathcal{P}_2(\reels^d)$.
Since $W$ is Lipschitz, Lemma \ref{lemma: fisher information is finite when V lipschitz} implies that 
\[\mathcal{I}_x(m^*)<+\infty.\]
Using the dynamic programming principle satisfied by Lipschitz solutions as in the proof of Lemma \ref{lemma: lip sol are monotone without idio} as well as a variant of Proposition 5.6 of \cite{Daudin} to deals with terms involving the entropy, we get that 
\begin{align}
\nonumber &\partial_t \varphi(t^*,m^*)+\int_{\reels^{2d}} F(x,\pi_dm^*,W(t^*,x,\pi_d m^*))\cdot \nabla_x \nabla_m \varphi(t^*,m^*)(x,y)m^*(dx,dy)\\
&\nonumber-\sigma_x \int_{\reels^d}\Delta_x \varphi(t^*,m^*,x,y)m^*(dx,dy)+2\sigma_x\int_{\reels^{2d}}\text{div}(W)(t^*,x,\pi_d m^*)m^*(dx,dy)\\
\label{eq minimum lip sol sigma}
&-\kappa \int_{\reels^d}F(x,\pi_dm^*,W(t^*,x,\pi_dm^*))\cdot\nabla_x\log(m^*(x,y))m^*(dx,dy)-\kappa \sigma_x \mathcal{I}_x(m^*)\\
\nonumber&\geq \int_{\reels^{2d}} F(x,W(t^*,x,\pi_d m^*),\pi_dm^*)\cdot(W(t^*,x,\pi_dm^*)-V(y,\pi_{-d}m^*))m^*(dx,dy)\\
\nonumber &+\int_{\reels^{2d}} G(x,W(t^*,x,\pi_d m^*),\pi_dm^*)\cdot (x-y)m^*(dx,dy),
\end{align}
Since $W$ is Lipschitz, and $\mathcal{I}_x(m^*)<+\infty$, a variant of Lemma \ref{lemma: integ part fisher} yields
\[\int_{\reels^{d}}\text{div}(W)(t^*,x,\pi_d m^*)m^*(dx,dy)=-\int_{\reels^d}W(t^*,x,\pi_dm^*)\cdot \nabla_x \log(m^*(x,y))m^*(dx,dy),\]
ending the proof for smooth $(F,G,W_0)$. When coefficients are only Lipschitz, we consider $(F_n,G_n,W_0^n)_{n\in\mathbb{N}}$ a sequence of regularization smooth in the $x$ variable, with the same Lipschitz constant as $(F,G,W_0)$. The conclusion then follows from the stability of monotone solutions Theorem \ref{stability with idio} and a bootstrapping argument. 
\end{proof}
Since the estimates we proved in the previous section still hold in the presence of idiosyncratic noise this essentially answer the question of existence for coefficients with low regularity. It remains to treat the question of uniqueness. To that end we first need to justify the equality \eqref{eq of diagonal terms sigma}. 
\begin{lemma}
\label{gateaux differentiability of Z}
Let $U,V\in C(\reels^d\times \mathcal{P}_2(\reels^d),\reels^d)$ be functions with at most linear growth and define  
\[Z:m\mapsto \int_{\reels^{2d}} (U(x,\pi_dm)-V(y,\pi_{-d}m)\cdot(x-y)m(dx,dy).\]
If $\mathcal{I}(m)<+\infty$, letting 
\[m_t=\mathcal{L}(X,Y-t\nabla_x\log(m(X,Y))),\]
for any $(X,Y)\in \left(L^2(\Omega,\reels^d)\right)^2$ such that $(X,Y)\sim m$,
\begin{equation*}
Z(m_t)-Z(m)=\int_0^t \int_{\reels^{2d}} \left(\begin{array}{c}
     U(x,\pi_dm_s) \\
     V(y,\pi_{-d}m_s)
\end{array}\right)\cdot \nabla_{(x,y)}\log (m(x,y))m(dx,dy).
\end{equation*}
\end{lemma}
\begin{proof}
    Let us first remark that if $V$ is smooth and has bounded derivatives, then this result follows naturally from the Frechet differentiability of 
    \[Y\mapsto \tilde{Z}(X,Y),\]
    and integration by part. If $V$ belongs to $C(\reels^d\times \mathcal{P}_2(\reels^d), \reels^d)$ only, then by \cite{regul_measure}, there exists a sequence of smooth function $(V_n)_{n\in\mathbb{N}}$, with linear growth uniformly in $n\in\mathbb{N}$, and converging locally uniformly to $V$. Defining $Z_n$ as we did $Z$ but for $(U,V_n)$, it follows that 
    \[Z_n(m_t)\underset{n\to \infty}{\longrightarrow}Z(m_t).\]
    Using the uniform linear growth of the sequence $(V_n)_{n\in\mathbb{N}}$,
    \[\int_{\reels^{2d}} V_n(y,\pi_{-d}m_s)\cdot \nabla_y\log(m(x,y))m(dx,dy)\underset{n\to \infty}{\longrightarrow}\int_{\reels^{2d}} V(y,\pi_{-d}m_s)\cdot \nabla_y\log(m(x,y))m(dx,dy),\]
    by the dominated convergence theorem. 
\end{proof}
\begin{definition}
	We say that a monotone solution $W:[0,T]\times\reels^d\times\mathcal{P}_2(\reels^d),\reels^d)$ associated with the coefficients $(F,G)$ lies in the closure of Lipschitz solutions if there exists a sequence of coefficients $(F_n,G_n)_{n\in\mathbb{N}}$ satisfying Hypothesis \ref{hyp: Lipschitz Wq} and an associated sequence of Lipschitz solution $(W_n)_{n\in\mathbb{N}}\subset C([0,T]\times \reels^d\times \mathcal{P}_2(\reels^d),\reels^d)$ such that $(F_n,G_n,W_n)_{n\in\mathbb{N}}$ converges locally uniformly in $C([0,T]\times \reels^d\times \mathcal{P}_2(\reels^d),\reels^d)$ to $(F,G,W)$ and 
	\begin{gather*}\exists C>0, \quad \forall n\in \mathbb{N}, (t,x,\mu)\in [0,T]\times \reels^d\times \mathcal{P}_2(\reels^d),\\	
 |F_n(x,\mu)|,|G_n(x,\mu)|,|W_n(t,x,\mu)|\leq C\left(1+|x|+\sqrt{E_2(\mu)}\right)
\end{gather*}
\end{definition}
\begin{lemma}
\label{lemma: sol in closure of lip sol}
Suppose $W,V$ are two monotone solutions to \eqref{eq: ME with idiosyncratic} and that $W$ lies in the closure of Lipschitz solutions. Let 
\[\varphi: (t,s,m)\mapsto \alpha\left((e^{\lambda s}+e^{\lambda t})(1+E_3(\pi_dm)+E_3(\pi_{-d}m)+\frac{1}{T-s}+\frac{1}{T-t}\right)+\frac{1}{2\gamma}|t-s|^2+\int_{\reels^{2d}} f(\bar{x})m(d\bar{x}),\]
for some $\alpha,\lambda,\gamma>0$ and $f$ with $\|f\|_{C^2}<+\infty$.
There exists a $\lambda^*>0$ such that for any $\lambda>\lambda^*$ if 
\[Z^\varphi:(t,s,m)\mapsto \int_{\reels^{2d}} (W(t,x,\pi_d m)-V(s,y,\pi_{-d}m))\cdot (x-y)m(dx,dy)-\varphi(t,s,m)+\kappa\mathcal{E}(m),\]
reaches a strict minimum $(t^*,s^*,m^*)$ in $[0,T]^2\times \mathcal{P}_2(\reels^d)$ and $t^*\neq 0, s^*\neq 0$, then 
\[\mathcal{I}(m^*)<+\infty,\]
and 
\begin{align*}&\partial_t \varphi(t^*,m^*)+\partial_s \varphi(t^*,s^*,m^*)+\int_{\reels^{2d}} 
\left(\begin{array}{c}
F(x,W(t^*,x,\pi_d m^*),\pi_d m^*)\\
F(y,V(s^*,y,\pi_{-d}m^*),\pi_{_d}m^*)	
\end{array}
\right)
\cdot \left(D_m \varphi(t^*,s^*,m^*)(x,y)-\kappa \nabla \log(m^*)\right)m^*(dx,dy)\\
&-\sigma_x \int_{\reels^{2d}}\text{Tr}\left(B D_{(x,y)}D_m \varphi(t^*,m^*,x,y)\right)m^*(dx,dy)-\sigma_x \kappa\int_{\reels^{2d}} |\nabla_x \log(m^*(x,y))-\nabla_y \log(m^*(x,y))|^2m^*(dx,dy)\\
&\geq \int_{\reels^{2d}} \left(F(x,W(t,x,\pi_d m^*),\pi_dm^*)-F(y,V(s^*,y,\pi_{-d} m^*),\pi_{-d}m^*)\right)\cdot(W(t,x,\pi_dm)-V(y,\pi_{-d}m^*))m^*(dx,dy)\\
&+\int_{\reels^{2d}} \left(G(x,W(t,x,\pi_d m^*),\pi_dm^*)-G(y,V(s^*,y,\pi_{-d} m^*),\pi_{-d}m^*)\right)\cdot (x-y)m^*(dx,dy).
\end{align*}
\end{lemma}
\begin{proof}
Let us fix some $\alpha,\gamma,\lambda,\kappa>0$ and assume $Z^\varphi$ reaches a point of strict minimum at $(t^*,s^*,m^*)\in (0,T)^2\times \mathcal{P}_2(\reels^d)$. By assumption, there exists a sequence of Lipschitz solutions $(W_n,F_n,G_n)_{n\in \mathbb{N}}$ converging locally uniformly to $(W,F,G)$. For a fixed $n\in \mathbb{N}$, we define $Z^\varphi_n$ by replacing $W$ with $W_n$. Arguing as in Theorem \ref{stability with idio}, we can find a sequence of perturbation $(f_n)_{n\in \mathbb{N}}$, such that $\|f_n\|_{C^2}\leq \frac{1}{n}$ and 
\[(t,s,m)\mapsto Z_n(t,s,m)+\int_{\reels^{2d}} f_n(\bar{x})m(d\bar{x}),\]
reaches a point of strict minimum $(t^*_n,s^*_n,m^*_n)$. Moreover 
\[\lim_{n\to \infty} (t^*_n,s^*_n,m^*_n)=(t^*,s^*,m^*),\]
in particular this implies that there exists a $n_0\in \mathbb{N}$ such that 
\[\forall n\geq n_0, \quad t^*_n,s^*_n>0.\]
For fixed $n\in {\mathbb{N}}$, we know that 
\[\mathcal{I}(m^*_n)<+\infty,\]
in particular we can find a sequence $(f_k)_{k\in \mathbb{N}}$ of smooth compactly supported function such that 
\[f_k\overset{L^2(m^*_n)}{\underset{k\to \infty}{\longrightarrow}} \nabla_x \log(m^*_n).\]
Letting 
\[\varphi_n:(t,s,m)\mapsto \varphi(t,s,m)+\int f_ndm,\]
for fixed $k\in \mathbb{N}$ we introduce $m^*_{n,k,t}=(id-t \left(\begin{array}{c}0_{\reels^d}\\f_k\end{array}\right))\#m^*_n$ and it follows from the definition of $m^*_n$ that 
\[\forall t>0, \quad (Z-\varphi_n)(t^*_n,s^*_n,m^*_{n,k,t})-(Z-\varphi_n)(t^*,s^*,m^*_n)+\kappa \left(\mathcal{E}(m^*_{n,k,t})-\mathcal{E}(m^*)\right)\geq 0. \]
Dividing by $t$ and using Lemma \ref{gateaux differentiability of Z} we deduce that there exists a constant $C$ depending only on the growth of $V$ and the Lipschitz constant of $W_n$ such that
\begin{gather*}C(1+E_2(m^*_n)+\|f_k\|_{L^2(m^*_n)})\|f_k-\nabla_x\log(m^*_n)\|_{L^2(m^*_n)}+\int_{\reels^{2d}}\nabla_y\nabla_m\varphi_n (t^*_n,s^*_n,m^*_n,x,y)\cdot f_k(x,y)m^*_n(dx,dy)\\
-\int_{\reels^{2d}} \nabla_y\log(m^*_n)\cdot f_k dm^*_n+\int_{\reels^{2d}} \left(\begin{array}{c}
W_n(t^*_n,x,\pi_dm^*_n)\\
V(s^*_n,y,\pi_{-d}m^*_n)
\end{array}
\right)
\cdot \nabla \log(m^*_n)dm^*_n\geq 0.
\end{gather*}	
Finally using the strong convergence of $(f_k)_{k\in \mathbb{N}}$ in $L^2(m^*_n)$, we get that $\forall n\in \mathbb{N},$
\begin{gather*}
-\int_{\reels^{2d}}\text{Tr}\left(\nabla_x\nabla_y\nabla_m\varphi_n (t^*_n,s^*_n,m^*_n,x,y)\right)m^*_n(dx,dy)
+\int_{\reels^{2d}} \nabla_y\log(m^*_n)\cdot \nabla_x\log(m^*_n) dm^*_n\\
\geq -\int_{\reels^{2d}} \left(\begin{array}{c}
W_n(t^*_n,x,\pi_dm^*_n)\\
V(s^*_n,y,\pi_{-d}m^*_n)
\end{array}
\right)
\cdot \nabla \log(m^*_n)dm^*_n.
\end{gather*}
Using this inequality as well as the definition of monotone solution for both $W_n$ and $V$, for $n\geq n_0$ the following holds 
\begin{align*}&\partial_t \varphi_n(t^*_n,,s^*_n,m^*_n)+\partial_s \varphi_n(t^*_n,s^*_n,m^*_n)\\
&+\int_{\reels^{2d}} 
\left(\begin{array}{c}
F(x,W_n(t^*_n,x,\pi_d m^*_n),\pi_d m^*_n)\\
F(y,V(s^*_n,y,\pi_{-d}m^*_n),\pi_{_d}m^*_n)	
\end{array}
\right)
\cdot \left(D_m \varphi_n(t^*,s^*,m^*)(x,y)-\kappa \nabla \log(m^*_n)\right)m^*_n(dx,dy)\\
&-\sigma_x \int_{\reels^{2d}}\text{Tr}\left(B D_{(x,y)}D_m \varphi_n(t^*_n,s^*_n,m^*_n,x,y)\right)m^*_n(dx,dy)-\sigma_x \kappa\int_{\reels^{2d}} |\nabla_x \log(m^*_n(x,y))-\nabla_y \log(m^*_n(x,y))|^2m^*_n(dx,dy)\\
&\geq \int_{\reels^{2d}} \left(F(x,W_n(t^*_n,x,\pi_d m^*_n),\pi_dm^*)-F(y,V(s^*,y,\pi_{-d} m^*_n),\pi_{-d}m^*_n)\right)\cdot(W(t,x,\pi_dm^*_n)-V(y,\pi_{-d}m^*_n))m^*_n(dx,dy)\\
&+\int_{\reels^{2d}} \left(G(x,W_n(t^*_n,x,\pi_d m^*_n),\pi_dm^*_n)-G(y,V(s^*_n,y,\pi_{-d} m^*_n),\pi_{-d}m^*_n)\right)\cdot (x-y)m^*_n(dx,dy).
\end{align*}
Using arguments already presented in the proof of Theorem \ref{stability with idio}, for any sequence $(X_n,Y_n)\sim m^*_n$ converging strongly to $(X,Y)\sim m^*$, $(\nabla \log(m^*_n(X_n,Y_n)))_{n\in \mathbb{N}}$ converges weakly in $\mathcal{H}^2$ to $\nabla \log(m^*(X,Y))$. By the weak lower semi continuity of norm in $\mathcal{H}^2$, 
\[-\int_{\reels^{2d}} |\nabla_x \log(m^*(x,y))-\nabla_y \log(m^*(x,y))|^2m^*(dx,dy)\geq \limsup_n -\int_{\reels^{2d}} |\nabla_x \log(m^*_n(x,y))-\nabla_y \log(m^*_n(x,y))|^2m^*_n(dx,dy),\]
consequently the result follows by taking the limit superior, other terms  being treated as in Theorem \ref{stability with idio}.
\end{proof}

\begin{thm}
\label{uniqueness monotone idio}
Under Hypothesis \ref{hyp: weak monotonicity} and \ref{hyp: linear growth in F}, if there exists a monotone solution $W\in C([0,T]\times \reels^d\times\mathcal{P}_2(\reels^d),\reels^d)$ to \eqref{eq: ME with idiosyncratic} in the closure of Lipschitz solutions, then it is the unique monotone solution with linear growth. Furthermore, if there exists such a solution, then it is $L^2-$monotone. 
\end{thm}
\begin{proof}
    Consider two solutions $W^1,W^2$, $W_1$ lying in the closure of Lipschitz solutions. We define $Z:[0,T]^2\times \mathcal{P}_2(\reels^{2d})\to \reels$ by 
    \[Z:(t,s,m)\mapsto \int_{\reels^{2d}} (W^1(t,x,\pi_d m)-W^2(s,y,\pi_{-d}m))\cdot(x-y)m(dx,dy).\]
    We now assume by contradiction that there exists $\delta>0$ and a couple $(t,m)$ such that 
    \[Z(t,t,m)<-\delta.\]
    Taking $\alpha,\lambda,\gamma,\kappa>0$, we define
    \[Z_{\alpha,\lambda,\gamma,\kappa}(t,s,m)\mapsto Z(t,s,m)+\alpha\left((e^{\lambda s}+e^{\lambda t})(1+E_3(\pi_dm)+E_3(\pi_{-d}m)+\frac{1}{T-s}+\frac{1}{T-t}\right)+\frac{1}{2\gamma}|t-s|^2+\kappa \mathcal{E}(m),\]
    with $E_3(m)=\int_{\reels^{2d}} |\bar{x}|^3m(d\bar{x})$. By Stegall's Lemma, for any $\varepsilon>0$, there exists $f_\varepsilon\in C^2(/reels^d,\reels)$ such that $\|f_\varepsilon\|_{C^2}\leq \varepsilon$ and 
    \[Z_{\alpha,\lambda,\gamma,\kappa,\varepsilon}:(t,s,m)\mapsto Z_{\alpha,\lambda,\gamma,\kappa}(t,s,m)+\int_{\reels^{2d}} f_\varepsilon(\bar{x})m(d\bar{x}),\]
    reaches a strict point of minimum $(t^*,s^*,m^*)$ in $[0,T]^2\times\mathcal{P}_2(\reels^{2d})$.
Taking $(\alpha,\kappa,\varepsilon)$ sufficiently small, we may assume that 
\begin{equation}
\label{eq: Z at minimum smaller than 0 idio}
Z_{\alpha,\lambda,\gamma,\kappa,\varepsilon}(t^*,s^*,m^*)<-\frac{\delta}{2}.
\end{equation}
By assumption on $W_0$ we know that 
\[\forall m\in \mathcal{P}_2(\reels^{2d}), \quad Z(0,0,m)\geq 0.\]
Using \eqref{ineq entropy}, we may also assume for $\kappa,\varepsilon$ sufficiently small before $\alpha$ that 
\[\forall m\in \mathcal{P}_2(\reels^{2d}), \quad Z_{\alpha,\lambda,\gamma,\kappa,\varepsilon}(0,0,m)\geq 0.\]
which means either $s^*\neq 0$ or $t^*\neq 0$. Suppose $s^*=0$ the case $t^*=0$ being treated in the same way. Since 
\[Z(t^*,s^*,m^*)\leq 0,\]
by the growth of $W^1,W^2$, and hence of $Z$, there exists a constant depending on the growth of $W^1,W^2$ only such that  
\[\alpha E_3(m^*),\frac{1}{\gamma}|s^*|^2\leq C(1+(1+\kappa)E_3(m)^\frac{2}{3}).\]
For $\gamma\leq \alpha^3,\kappa\leq 1$, this implies 
\[|s^*|\leq C\sqrt{\gamma},\]
for another constant with the same dependencies. By the non negativity of $m\mapsto Z_{\alpha,\lambda,\gamma,\kappa,\varepsilon}(0,0,m)$ and the continuity of $Z$, for $\gamma$ sufficiently small and $\kappa,\alpha$ sufficiently small before $\delta$ with $\kappa \leq \alpha$, this leads to a contradiction with \eqref{eq: Z at minimum smaller than 0 idio}. 
Taking $\gamma,\alpha$ sufficiently small and $\kappa\leq \alpha$, we may assume that both $t^*$ and $s^*$ are different from 0.
Applying the definition of a monotone solution to both $W^1$ and $W^2$, we first deduce that 
\[\mathcal{I}(m^*)<+\infty,\]
and that it is bounded independently of $\varepsilon,\lambda$. Using the notation
\[\mathcal{I}_{x-y}(m^*)=\int_{\reels^{2d}} \!|\nabla_x\! \log(m^*(x,y))-\nabla_y\! \log(m^*(x,y))|^2m^*(dx,dy),\]
Lemma \ref{lemma: sol in closure of lip sol} yields
\begin{align*}&-\!\frac{\alpha}{(T-t^*)^2}\!-\!\frac{\alpha}{(T-s^*)^2}\!-\!\lambda \alpha (e^{\lambda t^*}+e^{\lambda s^*})(1\!+\!E_3(\pi_dm^*)+\!E_3(\pi_{-d}m^*))\!+\sigma_x\!\int_{\reels^{2d}} \!\text{Tr}\left(B D^2_{(x,y)}\phi(t^*,s^*,x,y)\right)m^*(dx,dy)\!\\
&-\!\int_{\reels^{2d}}\! \left(
\begin{array}{c}
F(x,\pi_d m^*,W^1(t^*,x,\pi_dm^*)\\
F(y,\pi_{-d}m^*,W^2(t^*,y,\pi_{-d}m^*))
\end{array}
\right)
\cdot \left(\nabla_{(x,y)} \phi(t^*,s^*,x,y)\!+\!\kappa\nabla_{(x,y)} \log(m^*(x,y))\right)m^*(dx,dy)-\!\sigma_x\kappa  \mathcal{I}_{x-y}(m^*)\geq 0
\end{align*}
with 
\[\phi(t,s,x,y)=\alpha (e^{\lambda t}+e^{\lambda s})(|x|^3+|y|^3)+f(x,y).\]
Let us now remark that for any function $g\in C(\reels^d,\reels^d)$, 
\[\int_{\reels^{2d}} g(x)\cdot \nabla_x\log(m^*(x,y))m^*(dx,dy)=\int_{\reels^{2d}} g(x)\cdot \left(\nabla_x\log(m^*(x,y))-\nabla_y\log(m^*(x,y))\right)m^*(dx,dy).\]
By Cauchy-schwartz inequality in $L^2(m^*)$, there exists a constant $C_{1,2}$, depending on the linear growth of $F,W^1,W^2$ only such that
\begin{gather*}\!\int_{\reels^{2d}}\! \left(\begin{array}{c}
F(x,\pi_dm^*,W^1(t^*,x,\pi_dm^*))\\
     F(y,\pi_{-d}m^*,W^2(t^*,y,\pi_{-d}m^*))
\end{array}\right)\cdot \left(\nabla_{(x,y)} \phi(t^*,s^*,x,y)\!+\!\kappa\left(\begin{array}{c}
     \nabla_x \log(\pi_dm^*(x))  \\
     \nabla_y \log(\pi_{-d}m^*(y))
\end{array}\right)\right)m^*(dx,dy)\\
\leq C_{1,2} \left(2(\alpha+\varepsilon) (e^{\lambda t^*}+e^{\lambda s^*})(1+E_3(m^*))+\kappa (E_3(m^*))^\frac{1}{3}\sqrt{\mathcal{I}_{x-y}(m^*)}\right)
\end{gather*}
Choosing $\varepsilon \leq \min(\alpha,\frac{1}{2}\kappa \sqrt{\mathcal{I}(m^*)})$ and $\kappa \leq \sigma_x \alpha$, for 
\[\lambda \geq 24 (C_{1,2}+\sigma_x d)+1,\]
we have
\[-\frac{\alpha}{(T-t^*)^2}-\frac{\alpha}{(T-s^*)^2}\geq 0\]
which is absurd and hence contradicts \eqref{eq: Z at minimum smaller than 0 idio}. As a consequence
\[\forall t\in [0,T) ,m\in \mathcal{P}_2(\reels^{2d}) \quad Z(t,t,m)\geq 0.\]
By the continuity of $Z$, this also holds for $t=T$, uniqueness and $L^2-$monotonicity of a solution then follow from the argument presented in the proof of Lemma \ref{lemma: smooth uniqueness Hilbert approach}
\end{proof}

\begin{thm}
\label{existence: weak strong G with sig}
    Let $(F,G,W_0)$ be such that Hypotheses \ref{hyp: linear growth in F},\ref{hyp: weak strong monotonicity in G} and \ref{hyp: regularity: weak strong monotonicity in G} are satisfied, then there exists a unique monotone solution to \eqref{eq: ME with idiosyncratic}
\end{thm}
\begin{proof}
To show existence of a solution, the scheme of proof is exactly the same as in Theorem \ref{existence: weak strong G}. It is sufficient to notice that all estimates we used are still valid in the presence of idiosyncratic noise. Because of the way we construct our monotone solution it naturally lies in the closure of Lipschitz solutions, hence the uniqueness. 
\end{proof}

\subsection{On the regularity of solutions in the presence of idiosyncratic noise}
Let us begin this section with some reminders on parabolic systems. Introducing the notation 
\[[f]_{\alpha,\mathcal{O}}=\sup_{\begin{array}{c}
(x,y)\in \mathcal{O}^2\\
x\neq y
\end{array}}
\frac{|f(x)-f(y)|}{|x-y|^\alpha},\]
for the Hölder semi-norm and 
\[\|f\|_{k+\alpha,\mathcal{O}}=\sum_{i\leq k} \|D^i f\|_{\infty,\mathcal{O}}+[D^kf]_{\alpha,\mathcal{O}},\]
for the Hölder norm of order $k+\alpha$, we say that $f\in C^{k+\alpha}_{loc}(\reels^d)$ if and only if, for any compact set $\mathcal{O}\subset \reels^d$, 
\[\|f\|_{k+\alpha,\mathcal{O}}<+\infty.\]
We now remind the following interior estimates of which a proof can be found in \cite{krylov}
\begin{lemma}
\label{classic Hölder interior estimate}
	Let $\alpha>0$, and $u\in C([0,T],C^{2+\alpha}_{loc}(\reels^d))$ be a solution to 
	\[\partial_t u+f(t,x)\cdot \nabla_x u-\sigma_x\Delta_x u=g(t,x), (t,x)\in (0,T)\times \reels^d \]
	with $f,g\in C^\alpha_{loc}(\reels^d)$ and $\sigma_x>0$. Letting $B_r=\{ x\in \reels^d, |x|\leq r\},$for any $r>0$, there exists a constant $C_r$ depending only on $\sup_{[0,T]}\|f(t,\cdot)\|_{\alpha,B_{2r}},\sigma_x,\alpha,d,r$ such that 
	\[\forall t\in [0,T], \quad \|u(t,\cdot)\|_{2+\alpha,B_r}\leq C_r\left(\sup_{[0,T]}\|u(t,\cdot)\|_{\infty,B_{2r}}+\sup_{[0,T]} \|g(t,\cdot)\|_{\alpha,B_{2r}}+\|u(0,\cdot)\|_{2+\alpha,B_{2r}}\right).\]
\end{lemma}
We now show how the presence of idiosyncratic noise $\sigma_x>0$ regularize solutions of the master equations in the space variable $x$, even for coefficients with low regularity. The main idea is to exploit the strong link between the solution of the master equation \eqref{eq: ME with idiosyncratic} and parabolic systems. Indeed, for smooth solutions,  the characteristics of \eqref{eq: ME with idiosyncratic} are given by a PDE system composed of a forward Fokker-Planck equation and a backward non-linear parabolic system. We show this is still true at the level of Lipschitz solutions and use this property to get additional estimates on solutions.
\begin{lemma}
\label{lemma: lip sol sol}
    Let $(F,G,W_0)$ satisfy Hypothesis \ref{hyp: Lipschitz Wq}, and let $W$ be a Lipschitz solution to \eqref{eq: ME with idiosyncratic} defined on $[0,T_c)$ for some $T_c>0$. For fixed $(t,\mu)\in[0,T_c)\times \mathcal{P}_2(\reels^d)$, consider the following Mackean-Vlasov equation
    \begin{equation}
    \label{mk vlasov lip}
    \begin{array}{c}
         \partial_s\mu_s=-\text{div}\left(F(t-s,x,\mu_s,W(t-s,x,\mu_s))\mu_s\right)+\sigma_x\Delta_x\mu_s,  \\
    \end{array}
    \end{equation}
    and let $f:(s,x,u)\mapsto F(s,x,\mu_{t-s},u), g:(s,x,u)\mapsto G(s,x,\mu_{t-s},u)$. Then letting $w:(s,x)\mapsto W(s,x,\mu_{t-s})$,  $w\in C^{1,2}((0,t)\times \reels^d)$ and satisfies 
    \[\left\{\begin{array}{c}
          \partial_s w+f(s,x,w(s,x))\cdot\nabla_xw-\sigma_x \Delta_x w=g(s,x,w(s,x)) \text{ for } (s,x)\in(0,t)\times \reels^d  \\
         w_0(x)=W_0(x,\mu_t) \text{ for } x\in\reels^d
    \end{array}
    \right.
    \]
    
\end{lemma}
\begin{proof}
By the definition of Lipschitz solution, for $(s,x)\in [0,t]\times \reels^d$,
\begin{gather*}
    W(s,x,\mu_{t-s})=\esp{W_0(X_s,m_s)+\int_0^s G(X_u,m_u,W(s-u,X_u,m_u))du},\\
    \left\{\begin{array}{c}
         dX_u=-F(X_u,m_u,W(s-u,X_u,m_u))du+\sqrt{2\sigma_x}dB_u, \quad X_0=x,  \\
         \partial_u m_u=-\text{div}\left(F(x,m_u,W(s-u,x,m_u))m_u\right)+\sigma_x \Delta_x m_u, \quad m_0=\mu_{t-s}
    \end{array}
    \right.
\end{gather*}
Up to a change of origin by considering the measure $(\nu_u)_{u\in[t-s,t]}$ defined by $\nu_{t-s+u}=m_u$, it is easy to see by the uniqueness of weak solutions to the Fokker-Planck equation that, 
\[\forall u\in[0,s], \quad m_u=\mu_{t-s+u}.\]
As a consequence, we deduce that $\forall (s,x)\in [0,t]\times \reels^d$, 
\begin{gather}
\label{eq: prob repre of w}
    w(s,x)=\esp{w_0(X_s)+\int_0^s g(s-u,X_u,w(s-u,X_u)))du},\\
         dX_u=-f(s-u,X_u,w(s-u,X_u))du+\sqrt{2\sigma_x}dB_u, \quad X_0=x, 
\end{gather}
Since $W$ is a Lipschitz solution,  $f\circ w$, $g\circ w$, are Lipschitz in $x$ and $\frac{1}{2}-$Hölder in time on $[0,t]\times \reels^d$. Following \cite{unbounded-linear-parabolic} Theorem 2.1, there exists a unique solution $v\in C^{1,2}((0,t)\times \reels^d)$ to the problem
\[
\left\{\begin{array}{c}
     \partial_sv+f(s,x,w(s,x))\cdot \nabla_x v-\sigma_x\Delta_x v=g(s,x,w(s,x)) \text{ in }(0,t)\times \reels^d,  \\
     v|_{s=0}=w_0,
\end{array}\right.
\]
which is exactly given by the representation formula \eqref{eq: prob repre of w}. By uniqueness of Lipschitz solution $v=w$ which ends the proof. 
\end{proof}
By coming back to the characteristics of the mean field game system in this fashion, we can use parabolic regularity arguments to obtain further regularity of solutions to the master equation in the space variable $x$. To that end we now show that the regularization introduced in Lemma \ref{lemma: regul monotone function} conserves Hölder regularity. 
\begin{lemma}
\label{holder regul}
    Let $F:\reels^d\times \mathcal{P}_2(\reels^d)\to \reels^d$ be a $L^2-$monotone continuous function with linear growth, i.e
    \[\exists C_F, \forall (x,\mu)\in \reels^d\times \mathcal{P}_2(\reels^d) \quad |F(x,\mu)|\leq C_F \left(1+|x|+\sqrt{\int_{\reels^d} |y|^2 \mu(dy)}\right). \]
    
    Suppose that on bounded sets $\mathcal{O}_x\times \mathcal{O}_m \subset \reels^d\times \mathcal{P}_2(\reels^d)$, there exists a constant $C_\mathcal{O}$    
    \[\sup_{\mu\in\mathcal{O}_m}\|F(\cdot,\mu)\|_{\alpha,\mathcal{O}_x}\leq C_\mathcal{O}.\]
    Letting $(F^\varepsilon)_{\varepsilon>0}$, be the regularization introduced in Lemma \ref{lemma: regul monotone function},on each bounded sets $\mathcal{O}_x\times \mathcal{O}_m \subset \reels^d\times \mathcal{P}_2(\reels^d)$, there exists a constant $C'_\mathcal{O}$ and a modulus $\omega'_\mathcal{O}$ such that $\forall \varepsilon<\frac{1}{2C_F}$
    \[\sup_{\mu\in\mathcal{O}_m}\|F^\varepsilon(\cdot,\mu)\|_{\alpha,\mathcal{O}_x}\leq C'_\mathcal{O}.\]
    If furthermore there exists a modulus of continuity $\omega_\mathcal{O}$ such that 
    \[\quad \forall \mu,\nu\in\mathcal{O}_m,\quad \|F(\cdot,\mu)-F(\cdot,\nu)\|_{\alpha,\mathcal{O}_x}\leq \omega_\mathcal{O}(\mathcal{W}_2(\mu,\nu)),\]
    then there also exists a modulus $\omega'_\mathcal{O}$ such that $\forall \varepsilon<\frac{1}{2C_F}$
    \[\forall \mu,\nu\in\mathcal{O}_m,\quad \|F^\varepsilon(\cdot,\mu)-F^\varepsilon(\cdot,\nu)\|_{\alpha,\mathcal{O}_x}\leq \omega'_\mathcal{O}(\mathcal{W}_2(\mu,\nu)).\]

\end{lemma}
\begin{proof}
Reusing the notations introduced in Lemma \ref{lemma: regul monotone function}, there exists $g^\varepsilon,h_\varepsilon$, such that 
\[F^\varepsilon(x,\mu)=F(g^\varepsilon(x,\mu),h_\varepsilon(\mu)).\]
where $x\mapsto g^\varepsilon(x,\mu)$, $\mu\mapsto h_\varepsilon(\mu)$, are 1-Lipschitz maps and $g_\varepsilon$ is such that $\forall (x,\mu,\nu) \in \reels^d \times \left(\mathcal{P}_2(\reels^d)\right)^2$
\begin{gather*}
|g_\varepsilon(x,\mu)-g_\varepsilon(x,\nu)|\leq \varepsilon|F(g_\varepsilon(x,\mu),h_\varepsilon(\mu))-F(g_\varepsilon(x,\mu),h_\varepsilon(\nu))|.
\end{gather*}
Since we have already established in Corollary \ref{corol: linear growth regul} that those functions are bounded on bounded sets of $\reels^d\times \mathcal{P}_2(\reels^d)$ uniformly in $\varepsilon\leq \frac{1}{2C_F}$, the result follows naturally from the assumed regularity of $F$ and the regularity of $g_\varepsilon,h_\varepsilon$.
\end{proof}
\begin{hyp}
\label{hyp: local holder}
	There exists $\alpha>0$, such that on any bounded set $\mathcal{O}_m\times \mathcal{O}_x$ of $\mathcal{P}_2(\reels^d)\times \reels^d$ there exists a constant $C_\mathcal{O}$ such that 
	\[\sup_{\mu \in \mathcal{O}_m}\left(\|W_0(\cdot,\mu)\|_{2+\alpha,\mathcal{O}_x}+\|F(\cdot,\mu)\|_{\alpha,\mathcal{O}_x\times \mathcal{O}_x}+\|G(\cdot,\mu)\|_{\alpha,\mathcal{O}_x\times \mathcal{O}_x}\right)\leq C_\mathcal{O}\]
\end{hyp}
\begin{thm}
\label{thm: local regularity in x}
 Suppose $(F,G,W_0)$ are such that Hypotheses \ref{hyp: linear growth in F},\ref{hyp: weak strong monotonicity in G} and \ref{hyp: regularity: weak strong monotonicity in G} are satisfied, and let $W$ be the monotone solution of \eqref{eq: ME with idiosyncratic}.
 If furthermore Hypothesis \ref{hyp: local holder} holds, then for any bounded set $\mathcal{O}_m\times \mathcal{O}_x$ of $\mathcal{P}_2(\reels^d)\times \reels^d$, 
 \[\sup_{(t,\mu)\in [0,T]\times \mathcal{O}_m} \|W(t,\cdot,\mu)\|_{2+\alpha,\mathcal{O}_x}<+\infty\]
 \end{thm}
\begin{proof}
Defining a sequence of regularized problem associated with the coefficients $(W_0^\varepsilon,F^\varepsilon,G^\varepsilon)$ as in Theorem \ref{existence: weak strong G}, Lemma \ref{holder regul} guarantees that for $\varepsilon$ sufficiently small, $F^\varepsilon,G^\varepsilon$ are locally Hölder independently of $\varepsilon$. Letting 
$(\mu^\varepsilon_s)_{s\in[0,t]}$ be the weak solution of 
\[ \partial_s \mu_s=-\text{div}\left(F^\varepsilon(x,\mu_s,W(t-s,x,\mu_s))\mu_s\right)+\sigma_x\Delta_x \mu_S,\]
Lemma \ref{lemma: lip sol sol} implies that $w^\mu:(s,x)\mapsto W^\varepsilon(t-s,x,\mu_s)$ is a strong solution to 
\[\partial_s w^\mu +f^\mu(s,x)\cdot \nabla_x w^\mu-\sigma_x w^\mu= g^\mu(s,x),\]
with the notation $f^\mu(s,x)=F^\varepsilon(x,w^\mu(s,x),\mu_s),g^\mu(s,x)=G^\varepsilon(x,w^\mu(s,x),\mu_s)$. We already have an estimate indicating that $x\mapsto w^\mu(s,x)$ is Lipschitz uniformly in $s\in[0,t]$, for any $r>0$, and Gronwall's Lemma gives a bound on the second moment of $(\mu_s)_{s\in [0,t]}$ depending only on the growth of $F,G$ on $t$ and on $E_2(\mu)$. Consequently, for $\mu\in \mathcal{O}_m$ a bounded set of $\mathcal{P}_2(\reels^d)$, we can bound $\sup_{s\in[0,t]}\| f^\mu(s,\cdot)\|_{\alpha,B_r}$, $\sup_{s\in[0,t]}\| g^\mu(s,\cdot)\|_{\alpha,B_r}$ by a constant depending only on $r$ and $\mathcal{O}_m$. Applying Lemma \ref{classic Hölder interior estimate} to $w^\mu$ we deduce that for any $r>0$ and $\mathcal{O}_m$ bounded set of $\mathcal{P}_2(\reels^d)$ there exists a constant $C$ depending on $r,\mathcal{O}_m,F,G,W_0$ such that 
\[ \sup_{\mu\in \mathcal{O}_m} \|W^\varepsilon(t,\cdot,\mu)\|_{2+\alpha,B_r}\leq C.\]
Since said constant is independent of $\varepsilon$, the conclusion follows by letting it tends to 0. 
\end{proof}
\begin{remarque}
	Even if the initial condition does not belong to $C^{2+\alpha}_{loc}(\reels^d)$ for fixed $\mu \in \mathcal{P}_2(\reels^d)$, solutions belongs to this space for strictly positive times by the regularizing effect of the heat kernel. The above proof can be adapted to obtain local in time estimates blowing up as time tends to 0. It is also possible to obtain Hölder estimates on the regularity of $s\mapsto W(s,x,\mu_{t-s})$ in this fashion. However, the bottleneck in the regularity of $W$ with respect to time comes in general from the measure argument, which is why we did not insist on this estimate. 
\end{remarque}

\subsection{The case of master equations with common noise}
We now explain how those results can be extended to the master equation with common noise \eqref{eq: general ME}. First there is no problem in extending the definition of monotone solution. The change of variable argument we presented in Section \ref{section without idio} is still valid for master equation with idiosyncratic noise, and our uniqueness and existence results are consequently still valid in this setting. The main difficulty from this addition is the fact the characteristics of the master equations become stochastic and arguments to show $C^{k+\alpha}_{loc}$ regularity as we did with Theorem \ref{thm: local regularity in x} becomes much more technical. In \cite{convergence-problem}, the author showed how it is still possible to show parabolic regularity estimates in this setting by mean of Duhamel formula.  Based on the notions of solution introduced in this article and arguments of regularity presented in \cite{convergence-problem}, we believe the extension of Theorem \ref{thm: local regularity in x} to the situation of the master equation \eqref{eq: general ME} to be straightforward.

\section{Concluding remark and future perspectives}
We have introduced a notion of monotone solutions for the master equation \eqref{eq: general ME}. Such solutions only need to be continuous with respect to the measure argument to be well-defined, and uniqueness holds under $L^2-$monotonicity. Whenever there is no idiosyncratic noise, we use the stability of those solutions to show the existence of a monotone solution under various monotonicity assumptions. Those results appear new in their generality and regularity, both for MFG and FBSDE. 

For master equations with idiosyncratic noise, we show that similar results holds although the analysis is more delicate. In particular, we do not treat the question of uniqueness in full generality. We believe it can be addressed by a more careful treatment of second order terms introduced by the presence of idiosyncratic noise, the main difficulty lying in treating terms of the form 
\[ \int_{\reels^{2d}} \nabla_x \log(m(x,y))\cdot \nabla_y\log(m(x,y))m(dx,dy),\]
coming from the cross derivatives of the entropy. 
  \section*{acknowledgements}
\begin{center}
We would like to thank Charles Bertucci for his helpful remarks
\end{center}
\printbibliography
\end{document}